\documentclass[10pt]{article}

\usepackage[a4paper,left=2cm,right=2cm,top=2.5cm,bottom=2.5cm]{geometry}
\usepackage[colorlinks=true]{hyperref}
\usepackage{graphicx}
\usepackage{amsmath}
\usepackage{amsthm}
\usepackage{amssymb}
\usepackage{nicefrac}
\usepackage{natbib}
\usepackage{suffix}
\usepackage{mathtools}
\usepackage{bm}
\usepackage[only,llbracket,rrbracket]{stmaryrd}
\usepackage{enumitem}
\usepackage{pgfplots}
\usepackage{orcidlink}
\pgfplotsset{compat=1.18}
\usepackage{subfig}

\newtheorem{theorem}{Theorem}[section]
\newtheorem{proposition}[theorem]{Proposition}%
\newtheorem{lemma}[theorem]{Lemma}%
\newtheorem{corollary}[theorem]{Corollary}%

\newtheorem{remark}{Remark}[section]%

\newtheorem{definition}{Definition}[section]%
\newtheorem{assumption}{Assumption}

\numberwithin{equation}{section}

\newcommand\vect[1]{\bm{#1}}
\newcommand\normal{\vect{\eta}}
\newcommand\tangential{\vect{\tau}}
\newcommand\x{\vect{x}}

\newcommand\polOrder{\ell}

\newcommand\real{\mathbb{R}}
\newcommand\nat{\mathbb{N}}
\newcommand\Cp[1][0]{C^{#1}}
\newcommand\Lp[1][2]{L^{#1}}
\newcommand\Hk[1]{H^{#1}}
\newcommand\Wkp[1]{W^{#1}}
\newcommand\ConeCzero{$\Cp[1]$-$\Cp[0]$}
\newcommand\poly[1][\polOrder]{\mathbb{P}_{#1}}

\newcommand\ball[2]{B_{#2}(#1)}

\newcommand\grad{\nabla}
\newcommand\diver{\grad\cdot}
\newcommand\hess{\nabla^2}
\newcommand\laplace{\Delta}
\newcommand\biharmonic{\laplace^2}

\newcommand\dx{\, \mathrm{d}\x}
\newcommand\ds{\, \mathrm{d} s}

\newcommand\mesh[1][h]{\mathcal{T}_{#1}}
\newcommand\edges[1][h]{\mathcal{E}_{#1}}
\newcommand\edgesB[1][h]{\mathcal{E}_{#1}^{\mathcal{B}}}
\newcommand\edgesI[1][h]{\mathcal{E}_{#1}^{\mathcal{I}}}
\newcommand\element{K}
\newcommand\edge{e}

\newcommand\uEps{u_{\varepsilon}}
\newcommand\uEpsCof{\Phi_\varepsilon}
\newcommand\uEpsInterp{u_{\varepsilon,I}}
\newcommand\uh{u_{\varepsilon h}}
\newcommand\eh{e_{h}}

\newcommand\vemProj[2]{\Pi^{#1}_{#2}}
\newcommand\valueProj[1][\element]{\vemProj{#1}{0}}
\newcommand\gradProj[1][\element]{\vemProj{#1}{1}}
\newcommand\hessProj[1][\element]{\vemProj{#1}{2}}
\newcommand\edgeProj[1][\edge]{\vemProj{#1}{0}}
\newcommand\normalProj[1][\edge]{\vemProj{#1}{1}}
\newcommand\valueGlobal{\vemProj{h}{0}}
\newcommand\gradGlobal{\vemProj{h}{1}}
\newcommand\hessGlobal{\vemProj{h}{2}}
\newcommand\LtwoProj[2]{\mathcal{P}^{#1}_{#2}}
\newcommand\LtwoProjEle[1][\polOrder]{\LtwoProj{\element}{#1}}

\newcommand\Hnc[2][\polOrder]{\Hk{2,\mathrm{nc}}_{#1,#2}}
\newcommand\vemSpace[1]{V_{h,\polOrder,#1}}
\newcommand\localEnlargedVemSpace[1][\element]{\widetilde{V}_{h,\polOrder}^{#1}}
\newcommand\localVem[1][\element]{V_{h,\polOrder}^{#1}}

\DeclarePairedDelimiter\abs{\lvert}{\rvert}
\DeclarePairedDelimiter\norm{\lVert}{\rVert}
\DeclarePairedDelimiter\seminorm{\lvert}{\rvert}
\DeclarePairedDelimiter\jmp{\llbracket}{\rrbracket}

\newcommand\AQL{A_{QL}}
\newcommand\AL{A_{L}}
\newcommand\AQLh{A_{QL,h}}
\newcommand\ALh{A_{L,h}}
\newcommand\aQL{a_{QL}}

\newcommand\aQLh{a_{QL,h}}

\newcommand\bQL{b_{QL}}

\newcommand\bQLh{b_{QL,h}}

\newcommand\Stab[2][\element]{S_{#2}^{#1}}
\newcommand\StabQL[1][\element]{\Stab[#1]{\rho}}
\newcommand\StabQLa[1][\element]{\Stab[#1]{a}}
\newcommand\StabQLb[1][\element]{\Stab[#1]{b}}
\newcommand\StabL[1][\element]{\Stab[#1]{\kappa}}

\newcommand{\tnorm}[2][]{#1\lvert\!#1\lvert\!#1\lvert #2 #1\rvert\!#1\rvert\!#1\rvert}
\WithSuffix\newcommand\tnorm*[1]{\left\lvert\!\left\lvert\!\left\lvert #1 \right\rvert\!\right\rvert\!\right\rvert}

\DeclareMathOperator{\cof}{cof}
\DeclareMathOperator{\tr}{tr}
\DeclareMathOperator{\diam}{diam}

\newcommand{\emptyarg}{\:\cdot\:}

\newlist{assumptions}{enumerate}{1}
\setlist[assumptions]{label=(\alph*),labelsep=8pt,labelindent=0.5\parindent,itemindent=0pt,leftmargin=*}

\pgfplotscreateplotcyclelist{mylist}{
  {black,mark=o},
  {black,mark=triangle},
  {black,mark=square},
  {black!50,densely dashed},
  {black!50,loosely dashed},
  {black!50,dashdotdotted}
}
\pgfplotsset{
    compat=1.18,
    grid=major, 
    grid style={dashed,gray!20}, 
    every axis/.append style={
        label style={font=\small},
        tick label style={font=\small},
        legend style={at=({1,0}), anchor=south east, nodes={scale=0.75, transform shape}},
        mark size=1.5,
        cycle list name=mylist
    },
    scaled ticks=true,
    LargePlot/.style={
        height=6cm,
        width=200pt
    },
    SmallPlot/.style={
        y label style={at={(axis description cs:-0.2,0.5)}},
        height=6cm,
        width=160pt
    },
    legend cell align={left},
    try min ticks log=4
}

\hypersetup{%
    citecolor={blue},%
    linkcolor={red},%
}
\begin{document}

\title{Nonconforming virtual element method for the Monge--Amp\`ere equation}

\author{%
Scott Congreve\thanks{Corresponding author (\href{mailto:congreve@karlin.mff.cuni.cz}{congreve@karlin.mff.cuni.cz})} \ \orcidlink{0000-0002-7314-3120}, %
Alice Hodson \orcidlink{0009-0007-5793-679X}, and %
Anwesh Pradhan \orcidlink{0009-0005-4149-7724}}
\date{Charles University, Faculty of Mathematics and Physics, Sokolovsk\'a 83, 186 75, Praha, Czech Republic}

\maketitle
\begin{abstract}
    In this article, we develop the $\Cp[1]$-nonconforming $\Cp[0]$-conforming virtual element method (VEM) for the vanishing moment approximation of the second-order fully nonlinear Monge--Amp\`ere equation in two dimensions. In the vanishing moment equation an artificial biharmonic term is introduced which produces a quasilinear fourth order problem. We derive optimal a priori error estimates in the $\Hk{2}$-, $\Hk{1}$- and $\Lp[2]$-norms for the virtual element method, and show the existence and uniqueness of the virtual element solution. We perform several numerical experiments to validate the convergence rate of the error with respect to the mesh size.
\end{abstract}

\section{Introduction}{\label{sec:intro}}
In this paper we consider a virtual element method approximation of the Monge--Amp\`ere equation
\begin{alignat}{2}
\det (\hess u_0 )&= f && \qquad\text{in } \Omega, \label{eqn:monge_ampere}\\
u_0 &= g && \qquad\text{on } \partial\Omega,  \label{eqn:monge_ampere:bc}
\end{alignat}
where $\Omega \subset \real^2$ is a convex domain with smooth boundary $\partial\Omega$ and $f$ is of the form $f(\grad u_0, u_0, \x)$ \citep{gutierrez2016monge}. Equation \eqref{eqn:monge_ampere} is therefore a fully nonlinear second order elliptic partial differential equation, and it arises in various areas, such as differential geometry, optimal mass transport, optics, general relativity, meteorology, and semi-geostrophic equations \citep{neilan2020monge}.

Many problems arise when developing numerical methods for \eqref{eqn:monge_ampere}--\eqref{eqn:monge_ampere:bc} due to the nonlinearity of the problem. It is fundamentally difficult to construct and analyze solutions as a classical solution does not exist for \eqref{eqn:monge_ampere}; cf. \citet{gilbarg2015elliptic}. Therefore, we apply certain assumptions on the domain $\Omega$, which is considered to be convex, and the forcing term $f$ is taken to be strictly positive. Under these assumptions a unique solution exists that falls into the class of convex functions \citep{aleksandrov1960certain}; therefore, we look only for convex solutions $u_0$. Additionally, applying any form of discretization method to this equation is difficult because we cannot apply integration by parts to the fully nonlinear term, so the standard weak solution theory does not apply here.

Nevertheless, several numerical methods have been proposed for \eqref{eqn:monge_ampere}. Some of these include monotone finite difference methods \citep{barles1991convergence, brusca2023convergent}, least-squares finite element methods \citep{caboussat2023adaptive, caboussat2013least}, finite element methods based on the vanishing moment approach \citep{feng2011vanishing, feng2009vanishing, chen2021recovery, feng2009mixed, neilan2010nonconforming}, finite element methods based on $\Lp$ projection \citep{bohmer2008finite, c2011penalty}, and fixed point iteration methods \citep{benamou2010two, kohle2025robust}. The finite element methods mentioned here develop a priori error estimates and, as can be seen, are well explored. For more information we refer to \citet{neilan2020monge}.

The virtual element method (VEM), first introduced in \citet{beirao2013basic}, is a generalization of the finite element method. It has gained a lot of interest in recent years due to its robustness when considering polygonal or polyhedral meshes since the method does not rely on explicitly knowing the basis functions---instead we can compute integrals by knowing just the degrees of freedom on the boundary and interior of the element, which avoids computing shape functions over general polygons or polyhedra. VEM has been successfully applied to equations containing the biharmonic operator \citep{brezzi2013virtual, chinosi2016virtual, zhang2020nonconforming}, and nonlinear equations such as Cahn-Hilliard \citep{antonietti2016c, dedner2024higher} and the von K\'arm\'an equations \citep{lovadina2021virtual, shylaja2024morley}. Recent work has extended VEM to linear non-divergence form second-order equations \citep{bonnet2025conforming} and the fully nonlinear Hamilton--Jacobi--Bellman equation \citep{cai2025virtual}. For more information on VEM, we refer the reader to \citet{antonietti2022virtual}.

A direct VEM discretization of the second order Monge--Amp\`ere equation would require the development of a stable and computable treatment of the fully nonlinear Hessian determinant without the coercive bilinear structure normally available in the discretization of elliptic second order problems. Rather than addressing these difficulties directly, this paper adopts the vanishing moment method, which was introduced by \citet{feng2009vanishing} to approximate the viscosity solutions of \eqref{eqn:monge_ampere}. Here, an artificial biharmonic term, with a small constant coefficient $\varepsilon>0$ and suitable extra boundary condition, is introduced to give the following fourth order quasilinear partial differential equation:
\begin{alignat}{2}
- \varepsilon \biharmonic \uEps + \det(\hess \uEps) &= f && \qquad\text{in } \Omega, \label{eqn:vanishing_moment} \\
\uEps &= g &&  \qquad\text{on } \partial\Omega, \label{eqn:vanishing_moment:dbc} \\
\laplace \uEps &= \varepsilon && \qquad\text{on } \partial\Omega. \label{eqn:vanishing_moment:extra_bc}
\end{alignat}
It was shown in \citet{feng2009vanishing} that if $ f > 0 $ in $ \Omega $, then the problem \eqref{eqn:vanishing_moment}--\eqref{eqn:vanishing_moment:extra_bc} has a unique solution $ \uEps $, which is a convex function over $ \Omega $ and, moreover, that $\uEps$ uniformly converges as $ \varepsilon \to 0 $ to the unique viscosity solution of \eqref{eqn:monge_ampere}. This then shows that \eqref{eqn:monge_ampere} has a unique moment solution, which coincides with the unique viscosity solution. Therefore, we assume that there exists a unique, strictly convex solution to \eqref{eqn:vanishing_moment}--\eqref{eqn:vanishing_moment:extra_bc} for $\varepsilon > 0$.
Note that there are other possible choices for the additional boundary condition instead of (\ref{eqn:vanishing_moment:extra_bc}), which include $\nicefrac{\partial \laplace \uEps}{\partial \normal} = \varepsilon$ and $\hess \uEps \normal \cdot \normal = \varepsilon$.

In this paper, we consider a virtual element formulation for the vanishing moment formulation \eqref{eqn:vanishing_moment}--\eqref{eqn:vanishing_moment:extra_bc} of the Monge--Amp\`ere equation \eqref{eqn:monge_ampere}. We extend the work of \citet{neilan2010nonconforming} to the VEM setting (without the Morley element) in two dimensions. As the vanishing moment formulation is a fourth order partial differential equation we use a fourth order VEM space; however, as we are using this to approximate the second-order Monge--Amp\`ere equation we choose not to impose $\Cp[1]$-conformity on the numerical solution. To that end, we consider the \ConeCzero{} VEM first introduced in \citet{zhao2016nonconforming} for the plate bending problem, which is $\Cp[0]$-conforming but $\Cp[1]$-nonconforming in the sense that our discrete virtual element space is not a subspace of $\Hk{2}(\Omega)$. The virtual element method, therefore, presents an advantage for the consideration of the Monge--Amp\`ere equation, as it allows natural handling of a fourth order formulation, for a original problem that is actually second order, with support for high order polynomial degree and hanging nodes in the mesh without additional complications. The objective of this paper is to establish whether for a polytopal method like VEM, the polygonal flexibility can be retained for a fully-nonlinear problem of the Monge--Amp\`ere type. The vanishing moment formulation provides a natural test case: it retains the original nonlinear Hessian determinant term while embedding it within a fourth order problem, for which stable VEM machinery is already available. Additionally, we note that the vanishing moment method  will provide a useful method for computing a sensible initial guess for a future fully nonlinear formulation. For this formulation we consider the projection approach detailed in \citet{dedner2022robust,dedner2024framework}, which has also been applied to the fourth order nonlinear Cahn--Hilliard equation in \citet{dedner2024higher}. In this approach, the projection operators are defined without using the underlying variational problem, allowing us to apply the method directly to nonlinear problems. In this technique all projections are fully computable from the degrees of freedom, and are $\Lp[2]$-projections constructed as a hierarchy starting with the value projection. As a consequence of this construction we can discretise the nonlinear term containing the determinant of the Hessian directly with the Hessian projection, rather than requiring complicated averaging approaches used in, e.g., \citet{antonietti2016c}. To the best of our knowledge this paper appears to be the first VEM discretization of the Monge--Amp\`ere equation.

The structure of this paper is as follows. In Section~\ref{sec:continuous}, we introduce the continuous variational formulation of \eqref{eqn:vanishing_moment}--\eqref{eqn:vanishing_moment:extra_bc} and the corresponding linearized version. We introduce the VEM spaces, polynomial projections, and virtual element discretizations in Section~\ref{sec:vem}. The VEM formulation and optimal $\Hk{2}$-error bounds are derived for the discrete VEM version of the linearized form of \eqref{eqn:vanishing_moment}--\eqref{eqn:vanishing_moment:extra_bc} in Section~\ref{sec:linear_error}. This is then extended to the $\Hk{2}$-, $\Hk{1}$-, and $\Lp[2]$-error bounds of the nonlinear discrete VEM formulation of \eqref{eqn:vanishing_moment}--\eqref{eqn:vanishing_moment:extra_bc} in Section~\ref{sec:nonlinear_error}. We additionally show the existence and uniqueness of this VEM solution. Finally, in Section~\ref{sec:numerics}, we perform numerical experiments to validate the convergence rates of the error bounds.

Throughout this paper we use the standard notation $\Wkp{k,p}(\mathcal{D})$, with non-negative integer $k$ and domain $\mathcal{D}$, for Sobolev spaces of functions in $\Lp[p](\mathcal{D})$, $p\in[1,\infty]$, with the norm and seminorm denoted by $\seminorm{\emptyarg}_{k,p,\mathcal{D}}$ and $\norm{\emptyarg}_{k,p,\mathcal{D}}$, respectively. For $p=2$, we use the standard notation $\Hk{k}(\mathcal{D})\equiv\Wkp{k,2}(\mathcal{D})$ with norm and seminorm $\seminorm{\emptyarg}_{k,\mathcal{D}}$ and $\norm{\emptyarg}_{k,\mathcal{D}}$, respectively. We additionally note by $\Hk{1}_0(\mathcal{D})$ and $\Hk{1}_g(\mathcal{D})$ the subspaces of $\Hk{1}(\mathcal{D})$ with zero trace and trace equal to the function $g$ on the boundary $\partial\mathcal{D}$, respectively. When $\mathcal{D} = \Omega$ we may omit the subscript from the norms and seminorms. For a function space $X(\mathcal{D})$, we let $[X(\mathcal{D})]^2$ and $[X(\mathcal{D})]^{2\times2}$ denote the spaces of vector and tensor fields, respectively, whose components belong to $X(\mathcal{D})$. These spaces are equipped with the usual product norms which, for simplicity, we denote in the same way as the norm in $X(\mathcal{D})$. For tensor-valued functions $\underline{\sigma},\underline{\tau}$ we define $\underline{\sigma}:\underline{\tau}$ as the Frobenius inner product inducing the Frobenius norm $\abs{\emptyarg}$.

We denote by $\grad u$ and $\diver \vect{v}$ the standard gradient and divergence operators on a scalar-valued function $u$ and vector-valued function $\vect{v}$, respectively. Additionally, we define these operators for a vector-valued function $\vect{v}$ and a tensor-valued function $\underline{\sigma}$ by
\[
    (\grad \vect{v})_{ij} \coloneqq \frac{\partial v_i}{\partial x_j}, \quad i,j=1,2, \qquad \text{and}\qquad (\diver \underline{\sigma})_i \coloneqq \sum_{j=1}^2 \frac{\partial \sigma_{ij}}{\partial x_j}, \quad i=1,2,
\]
respectively. For a scalar-valued function $u$, we denote by $\hess u\equiv\grad(\grad u)$, $\laplace\equiv\diver(\grad u)$, and $\biharmonic=\laplace(\laplace u)$ the Hessian, Laplacian, and biharmonic operators, respectively.

\section{Continuous problem}\label{sec:continuous}

In this section, we derive the continuous variational formulation of both the nonlinear vanishing moment formulation \eqref{eqn:vanishing_moment}--\eqref{eqn:vanishing_moment:extra_bc} and a linearized version of this formulation; cf. \citet{neilan2009numerical}. To this end, we first state some useful results.

For $u$ and $v$ sufficiently smooth, we have
\[
\int_\mathcal{D} \hess u : \hess v \dx
= \int_\mathcal{D} \laplace u \laplace v \dx
+ \int_{\mathcal{D}} \left(
2 \frac{\partial^2 u}{\partial x_1 \partial x_2} \frac{\partial^2 v}{\partial x_1 \partial x_2}
- \frac{\partial^2 u}{\partial x_1^2} \frac{\partial^2 v}{\partial x_2^2}
- \frac{\partial^2 u}{\partial x_2^2} \frac{\partial^2 v}{\partial x_1^2}
\right) \dx.
\]
This equation can then be written as
\begin{equation}
\int_\mathcal{D} \hess u : \hess v \dx
= \int_\mathcal{D} \laplace u  \laplace v \dx
+\int_{\partial\mathcal{D}} \left(
\frac{\partial^2 u}{\partial \normal \partial \tangential} \frac{\partial v}{\partial \tangential}
- \frac{\partial^2 u}{\partial \tangential^2} \frac{\partial v}{\partial \normal}
\right) \ds \label{eqn:hess_eq_biharmonic}
\end{equation}
where $\normal$ and $\tangential$ are the outward normal and counter-clockwise tangential unit vectors, respectively, on the boundary $\partial\mathcal{D}$; cf. \citet{zhao2016nonconforming}.
\begin{proposition}[Mean value theorem] \label{prop:mean_value}
For any $\underline{\sigma},\underline{\upsilon} \in [\Cp(\mathcal{D})]^{2\times 2}$ there exists a constant $t\in[0,1]$ such that
\[
\det(\underline{\sigma})-\det(\underline{\upsilon}) = \cof(t\underline{\sigma} + (1-t)\underline{\upsilon}): (\underline{\sigma} - \underline{\upsilon}).
\]
\end{proposition}
\begin{proof}
    Follows similarly to \citet{awanou2015standard}.
\end{proof}
Finally, for $u \in C^2(\Omega)$, we have by elementary calculations 
\begin{equation}
    \det(\hess u) = \frac{1}{2}\diver(\cof(\hess u)\grad u).\label{eqn:hess_eq_divcof}
\end{equation}

We can now derive the continuous variational formulation for \eqref{eqn:vanishing_moment}--\eqref{eqn:vanishing_moment:extra_bc}: find $\uEps\in V \coloneqq \Hk{2}(\Omega)\cap\Hk{1}_g(\Omega)$ such that
\begin{equation}
    \AQL(\uEps,v) = \int_\Omega fv \dx + \varepsilon \int_{\partial\Omega} \left(\frac{\partial^2 g}{\partial \tangential^2} - \varepsilon\right) \frac{\partial v}{\partial \normal} \ds, \label{eqn:variational}
\end{equation}
for all $v \in W\coloneqq \Hk{2}(\Omega)\cap\Hk{1}_0(\Omega)$, where
\begin{align}
  \AQL(\uEps,v) &\coloneqq \aQL(\uEps,v) + \bQL(\uEps,v), \label{eqn:variational:A} \\
  \aQL(\uEps,v) &\coloneqq -\varepsilon\int_\Omega \hess \uEps : \hess v \dx, \label{eqn:variational:a} \\
  \bQL(\uEps,v) &\coloneqq \int_\Omega \det(\hess u_\varepsilon)v\dx. \label{eqn:variational:b}
\end{align}

We note from \citet{neilan2009numerical} that the solution $\uEps$ has certain known bounds.
\begin{proposition}[\citet{neilan2009numerical}]\label{prop:ueps_bounds}
    For the solution $\uEps$ to the vanishing moment problem \eqref{eqn:vanishing_moment}--\eqref{eqn:vanishing_moment:extra_bc}
    \begin{align*}
    \norm{\uEps}_{j,2} &= \mathcal{O}\left(\varepsilon^{\nicefrac{(1-j)}{2}}\right), \text{ for } j = 1,2,3, &
    \norm{\uEps}_{j,\infty} &= \mathcal{O}\left(\varepsilon^{-j}\right), \text{ for } j = 1,2, \\
    \norm{\uEpsCof}_{0,2} &= \mathcal{O}\left(\varepsilon^{-\nicefrac{1}{2}}\right), &
    \norm{\uEpsCof}_{0,\infty} &= \mathcal{O}\left(\varepsilon^{-1}\right).
    \end{align*}
    where $\uEpsCof = \cof(\hess\uEps)$ is the cofactor matrix of $\hess\uEps$.
\end{proposition}

Using these results, we can show the nonlinear form $\AQL$ is continuous when applied to the solution $\uEps$.
\begin{lemma} \label{lemma:continuity}
    For the solution $\uEps\in V$ of vanishing moment equation \eqref{eqn:vanishing_moment}--\eqref{eqn:vanishing_moment:extra_bc} and any $v \in W$,
    \begin{align*}
        \AQL(\uEps,v) &\leq C_1(\varepsilon)\norm{u_\varepsilon}_{2,\Omega}\norm{v}_{2,\Omega}, \\
        \aQL(\uEps,v) &\leq C\varepsilon\norm{u_\varepsilon}_{2,\Omega}\norm{v}_{2,\Omega},
    \end{align*}   
    where $C_1(\varepsilon) = C\left(\varepsilon+\dfrac{1}{2}\varepsilon^{-\nicefrac{1}{2}}\right) = \mathcal{O}(\varepsilon^{-\nicefrac{1}{2}})$, with constant $C>0$.
\end{lemma}
\begin{proof}
From Proposition~\ref{prop:ueps_bounds} we have $\norm{\cof(\hess \uEps)}_{0,\Omega} = \mathcal{O}\left(\varepsilon^{-\nicefrac{1}{2}}\right)$. Then, considering \eqref{eqn:hess_eq_divcof} and the fact that $\Hk{2}(\Omega) \hookrightarrow \Wkp{1,4}(\Omega)$ continuously, we have
    \begin{align*}
       \abs{\AQL(\uEps,v)} &= \abs*{-\varepsilon \int_\Omega \hess \uEps : \hess v \dx + \int_\Omega \det(\hess \uEps)v \dx} \\
       &\leq \varepsilon \norm{\hess \uEps}_{0,\Omega}\norm{\hess v}_{0,\Omega} + \frac{1}{2}\norm{\cof(\hess \uEps)}_{0,\Omega}\norm{\grad \uEps}_{0,4,\Omega}\norm{\grad v}_{0,4,\Omega} \\
       &\leq \varepsilon \norm{\uEps}_{2,\Omega}\norm{v}_{2,\Omega} + \frac{1}{2}\varepsilon^{-\nicefrac{1}{2}}\norm{\uEps}_{1,4,\Omega}\norm{v}_{1,4,\Omega}\\
       &\leq C_1(\varepsilon)\norm{\uEps}_{2,\Omega}\norm{v}_{2,\Omega}.
    \end{align*}
    The bound for $\aQL$ follows similarly.
\end{proof}

In order to analyze the VEM method for this problem, we need to consider a linearized version of \eqref{eqn:vanishing_moment}--\eqref{eqn:vanishing_moment:extra_bc} and its corresponding variational formulation. Here, we follow the steps from \citet{neilan2009numerical} and require the following divergence-free property.
\begin{proposition}\label{prop:div_free}
Let $\vect{v} = (v_1, v_2, \cdots , v_n): \Omega \to \mathbb{R}^n $ be a vector-valued function, and assume $ \vect{v} \in [\Cp[2](\Omega)]^n $. Then the cofactor matrix $\cof(\grad\vect{v}) $ of the gradient matrix $\grad\vect{v} $ satisfies the row divergence-free property
\[
\diver(\cof(\grad\vect{v}))_i = \sum_{j=1}^n \frac{\partial}{\partial x_j} (\cof(\grad\vect{v}))_{ij} = 0 \quad \text{for} \quad i = 1,\dots,n,
\]
where $(\cof(\grad\vect{v}))_i$ and $(\cof(\grad\vect{v}))_{ij}$ denote the $ i $-th row and $ (i,j) $-entry of $ \cof(\grad\vect{v}) $, respectively.
\end{proposition}
\begin{proof}
See \citet{evans2022partial}.
\end{proof}

To derive the linearized version, let $\uEps$ be the solution of \eqref{eqn:vanishing_moment}--\eqref{eqn:vanishing_moment:extra_bc}, $v$ be a smooth function, and $ t \in \real $; then, we have
\[
\det(\hess (\uEps + t v)) = \det(\hess \uEps) + t \tr(\uEpsCof : \hess v) + \dots + t^n \det(\hess v),
\]
where the fixed value $\uEpsCof\in\real^{n\times n}$ denotes the cofactor matrix of $\hess u_\varepsilon $; cf. \citet{caffarelli1997properties}. By differentiating this expression with respect to $ t $ and  setting $ t = 0 $, we obtain the linearization of the operator $ M_\varepsilon(\uEps) := \varepsilon\biharmonic \uEps + \det(\hess \uEps) $ at the solution $\uEps$. This leads to the linearized operator
\begin{equation}\label{eqn:linear_op}
L_{\uEps}(v) \coloneqq \varepsilon\biharmonic v - \uEpsCof : \hess v = \varepsilon\biharmonic v - \diver (\uEpsCof \grad v),
\end{equation}
where the final equality follows from Proposition~\ref{prop:div_free}.\\

Therefore, for a given $ \varphi \in \Lp[2](\Omega) $ and $ \psi \in \Hk{\nicefrac{-1}{2}}(\partial \Omega) $, we now consider the problem
\begin{alignat}{2}
L_{\uEps}(v) &= \varphi, && \qquad\text{in } \Omega, \label{eqn:linear} \\
v &= 0, && \qquad\text{on } \partial \Omega, \label{eqn:linear:dbc} \\
\laplace v &= \psi, && \qquad\text{on } \partial \Omega, \label{eqn:linear:extra_bc}
\end{alignat}
with its associated variational formulation:
find $v\in W$ such that
\begin{equation}
    \AL(v,w) =  \int_\Omega \varphi w \dx + \varepsilon\int_{\partial\Omega} \psi \frac{\partial w}{\partial \normal}\ds \label{eqn:linear:variational}
\end{equation}
for all $w\in W$, where
\begin{equation}
    \AL(v,w) \coloneqq \varepsilon\int_\Omega \hess u:\hess v \dx + \int_\Omega\uEpsCof \grad v\cdot \grad w \dx. \label{eqn:linear:variational:A}
\end{equation}

We refer to \citet{neilan2009numerical} for proof of the existence and uniqueness of the solution $v\in W$ to \eqref{eqn:linear:variational}. Additionally, from \eqref{eqn:linear:variational}, we have the following properties of the linear form.
\begin{proposition}[Continuity] \label{prop:linear_weak:continuity}
    For any $v, w \in W$, there exists a positive constant $C_2(\varepsilon) = C(\varepsilon + \varepsilon^{-1}) = \mathcal{O}(\varepsilon^{-1})$ such that
    \[
        A_{L}(v,w) \leq C_2(\varepsilon)\|v\|_{{2,\Omega}}\|w\|_{{2,\Omega}}.
    \]
\end{proposition}
\begin{proposition}[Coercivity] \label{prop:linear_weak:coercivity}
    For any $w \in W$, there exists a positive constant $C_3(\varepsilon) =C\min(\varepsilon,\beta) = \mathcal{O}(\varepsilon)$, such that
    \[
        A_{L}(w,w) \geq C_3(\varepsilon)\|w\|^2_{2,\Omega}.
    \]
    where $\x^\top \uEpsCof \x \geq \beta > 0$, for all $\x\in\real^2$, due to the fact that $\uEpsCof$ is strictly convex.
\end{proposition}
\section{Virtual element method}\label{sec:vem}

In this section, we present the discretization of both the linearized formulation \eqref{eqn:linear}--\eqref{eqn:linear:extra_bc} and vanishing moment approximation \eqref{eqn:vanishing_moment}--\eqref{eqn:vanishing_moment:extra_bc} of the Monge--Amp\`ere equation using the \ConeCzero{} virtual element method.
We use the so-called ``VEM enhancement'' technique from \citet{ahmad2013equivalent}, \citet{cangiani2017conforming}, and extended to fourth order problems in \citet{dedner2024higher}, to discretise these problems.

\subsection{Meshes and spaces}
Let $\mesh$ be a subdivision of $\Omega \subset \mathbb{R}^2$ into simple non-overlapping polygonal elements $\element$ such that $\overline{\Omega} = \bigcup_{\element\in\mesh}$.  We denote by $h_\element\coloneqq\diam(\element)$ the diameter of the element $\element$ and let $h=\max_{\element\in\mesh} h_\element$ denote the maximum diameter of all elements.

We define an interior edge $\edge$ of $\mesh$ as the intersection of two neighboring elements $\element^+,\element^-\in\mesh$; i.e., $e=\partial \element^+\cap \partial \element^-$. Similarly, we define a boundary face $\edge\subset \partial\Omega$ as an entire face of an element $\element$ on the boundary. We denote by $\edgesI$ and $\edgesB$ the set of all boundary and interior edges, respectively, and define $\edge=\edgesI\cup\edgesB$ as the set of all edges. For each edge $\edge\in\edges$ we define $h_\edge$ as the length of the edge. For an edge $\edge\in\edgesI$ shared by the elements $\element^+,\element^-\in\mesh$, i.e.,  $\edge=\partial \element^+\cap \partial \element^-$, we denote by $\normal_\edge$ and $\tangential_\edge$ the outward normal and counter-clockwise tangential unit vectors with respect to the element $\element^+$, respectively, on the edge $\edge$. For $\edge\in\edgesB$, we set $\normal_\edge=\normal$ and $\tangential_\edge=\tangential$.

\begin{assumption}[Mesh Regularity] \label{assumption:mesh}
We assume there exists a positive constant $\varrho > 0$ such that
\begin{assumptions}
    \item $h_\edge \geq \varrho h_K$ for every element $\element\in\mesh$ and for every edge $\edge \subset \partial\element$,\label{assumption:mesh:shape_regular}
    \item for every element $K \in \mathcal{T}_h$ there exists a constant $C\in\nat$ such that the $N(\element) \leq C$, where $N(\element)$ is the total number of edges of $\element$, and\label{assumption:mesh:edges}
    \item every element $\element$ is star-shaped with respect to all points of a ball of radius $\varrho h_\element$.\label{assumption:mesh:star_shaped}
\end{assumptions}
\end{assumption}
For an integer $ s > 0 $, we define the broken Sobolev space $\Hk{s}(\mesh)$ by
\[
\Hk{s}(\mesh) \coloneqq \left\{ v \in \Lp(\Omega) : v|_\element \in \Hk{s}(\element), \; \forall \element \in \mesh \right\},
\]
with the broken $ H^s $ norm and the seminorm,
\[
\norm{v}_{s,h}^2 = \sum_{\element \in \mesh} \norm{v}_{s,\element}^2, \qquad \seminorm{v}_{s,h}^2 = \sum_{\element \in \mesh} \seminorm{v}_{s,\element}^2.
\]

Given two adjacent elements $\element^+,\element^-\in\mesh$ sharing a common edge $\edge\in\edges$, i.e.,  $\edge=\partial \element^+\cap \partial \element^-$, we denote by $v^\pm$ the restriction of the function $v$, which is smooth in each element in $\mesh$, to $\element^\pm$, respectively. With this notation, we define the jump operator
\[
\jmp{v} \coloneqq v^+|_\edge - v^-|_\edge.
\]
For a boundary edge $\edge\in\edgesB$, we set $\jmp{v} \coloneqq v|_\edge$.

For a non-negative integer $\polOrder$ we denote by $\poly(\element)$ the space of polynomials of degree $\leq \polOrder$ on $\element$, and define the piecewise polynomial space $ \poly(\mesh) $ for any $ \polOrder \in \nat$ as
\[
\poly(\mesh) := \left\{ p \in \Lp(\Omega) : p|_\element \in \poly(\element), \; \forall \element \in \mesh \right\}.
\]

Following \citet{zhao2016nonconforming} and \citet{dedner2022robust}, we now define the broken Sobolev spaces
\begin{align*}
\Hnc{g}(\mathcal{T}_h) \coloneqq \Bigg\{v \in \Hk{2}(\mesh) \cap \Hk{1}_{g}(\Omega): &\int_\edge \jmp*{\frac{\partial v}{\partial \normal_\edge}} p \ds = 0\;\forall p \in \poly[\polOrder-2](\edge),\\  &\int_\edge \jmp{v}p \ds = 0\;\forall p \in \poly[\polOrder-3](\edge), \forall \edge \in \edgesI\Bigg\}, \\
\Hnc{0}(\mathcal{T}_h) \coloneqq \Bigg\{v \in \Hk{2}(\mesh) \cap \Hk{1}_0(\Omega):
&\int_\edge \jmp*{\frac{\partial v}{\partial \normal_\edge}} p \ds = 0 \; \forall p \in \poly[\polOrder-2](\edge),\forall \edge \in \edgesI
\Bigg\}.
\end{align*}
\begin{remark}
    It is easy to see that $\norm{v_h}_{2,h}$ is a norm on $V$, $W$and $\Hnc{0}(\mesh)$.
    From \citet{zhao2016nonconforming} and \citet{zhao2018morley}, $\seminorm{v_h}_{2,h}$ is a norm on $\Hnc{0}(\mesh)$, and we also have the Poincar\'e inequality
    \begin{equation}\label{eqn:poincare}
        \seminorm{v_h}_{1,h} \leq C\seminorm{v_h}_{2,h}, \qquad \forall v_h \in \Hnc{0}(\mesh),
    \end{equation}
    with positive constant $C$.
\end{remark}
\begin{definition}
For any $\element \in \mesh$, we denote the orthogonal $\Lp(\element)$-projection onto the space $\poly(\element)$ by
\[
\LtwoProjEle : \Lp(\element) \to \poly(\element).
\]
\end{definition}
We also recall the following approximation results for the $\Lp$-projection. A proof of the following can be obtained, for example, using the theory in \citet{ciarlet2002finite,brenner2008mathematical}.

\begin{lemma} \label{lemma:l2proj}
Under Assumptions~\ref{assumption:mesh}\ref{assumption:mesh:shape_regular}\&\ref{assumption:mesh:star_shaped}, for $\polOrder \geq 0$ and for any $w \in \Hk{m}(\element)$ with $1 \leq m \leq \polOrder + 1$, it follows that
\[
\norm{w - \LtwoProjEle w}_{s,\element} \leq C h_\element^{m-s} \seminorm{w}_{m,\element}, \quad \text{for } s = 0,\dots,2.
\]
Furthermore, for any edge shared by $\element^+, \element^- \in \mesh$, and for any $w \in \Hk{m}(\omega_\edge)$, $1\leq m\leq\polOrder+1$, where $\omega_\edge=\element^+ \cup \element^-$, it follows that
\[
    \norm{w - \LtwoProjEle w}_{s,e} \leq C h_\edge^{m-s - \nicefrac{1}{2}} \seminorm{w}_{m,\omega_\edge}, \quad \text{for } s = 0, \dots, 2.
\]
An identical bounds also holds for edges $\edge\in\edgesB$ with $\omega_\edge=\element$, where $\element\in\mesh$ is the element such that $\edge\subset\partial\element$, and it is given by 
\[
    \norm{w - \LtwoProjEle w}_{s,e} \leq C h_\edge^{m-s - \nicefrac{1}{2}} \seminorm{w}_{m,K}, \quad \text{for } s = 0, \dots, 2.
\]
\end{lemma}

\subsection{Local VEM space and projection operators}

In this section, we construct the local discrete VEM space $\localVem$ for an element $\element\in\mesh$ in two dimensions. Following the idea in \citet{ahmad2013equivalent}, we first introduce the enlarged space $\localEnlargedVemSpace$, whose degrees of freedom (dofs) are used to define those of  $\localVem$. These dofs enable the construction of the projection operators $\valueProj$, $\gradProj$, and $\hessProj$, which are fully computable and used to define the local VEM space and the local discrete forms. For more details, the reader can refer to \citet{dedner2022robust}. Throughout this section, and the remainder of the paper, we assume that $\polOrder \geq 2$.

For $\element \in \mesh$, we define the local enlarged VEM space
\begin{equation*}
    \localEnlargedVemSpace \coloneqq \left\{ v_h \in \Hk{2}(\element) : \biharmonic v_h \in \poly(\element),  v_h|_\edge \in \poly(\edge), \laplace v_h|_\edge \in \poly[\polOrder-2](\edge), \forall \edge \subset \partial K \right\},
\end{equation*}
where this space is characterized by the extended set of degrees of freedom $\widetilde{\Lambda}^\element$ described by the \emph{dof tuple} $(0, -1, \polOrder-2, \polOrder-2, \polOrder)$, using the notation from \citet{dedner2022robust}. We define these degrees of freedom as:
\begin{enumerate}[label=(D\arabic*),labelsep=8pt,labelindent=0.5\parindent,itemindent=0pt,leftmargin=*]
    \item The value of $v_h$ at each vertex of $\element$.\label{dofs:value}
    \item For $\polOrder > 1$, the moments of $v_h$ up to order $\polOrder-2$ on each edge $\edge \subset \partial \element$
        \[
            \frac{1}{\abs{\edge}} \int_\edge v_h p \ds \qquad \forall p \in \poly[\polOrder-2](\edge).
        \]\label{dofs:edge}
    \item For $\polOrder > 1$, the normal moments of $v_h$ up to order $\polOrder-2$ on each edge $\edge \subset \partial \element$
        \[
            \int_{\edge} \partial_n v_h p \ds \qquad \forall p \in \poly[\polOrder-2](\edge).
        \]\label{dofs:normals}
    \item For $\polOrder > 1$, the moments of $v_h$ up to order $\polOrder$ inside $\element$
        \[
            \frac{1}{\abs{\element}} \int_{\element} v_h p \dx \qquad \forall p \in \poly(\element).
        \]\label{dofs:volume}
\end{enumerate}

\begin{assumption}\label{assumption:monomial_linfty}
  We assume the polynomial basis functions $p_\element \in \poly(\element)$ and $p_\edge \in \poly(\edge)$ used in \ref{dofs:edge}--\ref{dofs:volume} to compute the values of the degrees of freedom are defined such that $\seminorm{p_\element}_{0,\infty,\element} \simeq 1$ and $\seminorm{p_\edge}_{0,\infty,\edge} \simeq 1$.
\end{assumption}
\begin{remark}
    We note that the standard choice of scaled monomials satisfies Assumption~\ref{assumption:monomial_linfty},  cf. \citet{beirao2017stability}.
\end{remark}

In order to define $\localVem$, we first introduce for $\element\in\mesh$ and $\edge\in\partial\element$ the interior value projection $\valueProj : \localEnlargedVemSpace\to \poly(\element)$, the edge value projection $\edgeProj : \localEnlargedVemSpace \to \poly(\edge)$, and the edge normal projection $\normalProj : \localEnlargedVemSpace \to \poly[\polOrder-1](\edge)$. These projections will be computable from the degrees of freedom $\Lambda^\element $ of a given $v_h \in \localVem$, cf. \eqref{eqn:localvem}.

\theoremstyle{definition}
\begin{assumption}\label{assumption:projections}
For $\element\in\mesh$, $e\in\partial\element$, and $ v_h \in \localEnlargedVemSpace$, we assume that the value projection $\valueProj$, the edge projection $\edgeProj$ and the normal edge projection $\normalProj$ are a linear combination of the dofs $\Lambda^\element(v_h)$
and satisfy the additional properties:
\begin{itemize}
    \item The value projection $\valueProj v_h\in \poly(\element) $ satisfies
    \[
    \int_\element \valueProj v_h p \dx = \int_K v_h p \dx \qquad \forall p \in \poly[\polOrder-4](\element),
    \]
    and $ \valueProj q = q $ for all $ q \in \poly(\element) $.
    \item For each edge, the edge projection $\edgeProj v_h \in \mathbb{P}_{\polOrder}(e) $ satisfies $\edgeProj v_h(\edge^\pm) = v_h(\edge^\pm)$, where $\edge^\pm\in\real^2$ are the two endpoints of the edge $\edge$,
    \[
    \int_\edge \edgeProj v_h p \ds = \int_\edge v_h p \ds \qquad \forall p \in \poly[\polOrder-2](\edge),
    \]
    and $ \edgeProj q = q|_e $ for all $ q \in \poly(\element) $.
    \item For each edge, the edge normal projection $\normalProj v_h \in \poly[\polOrder-1](\edge) $ satisfies
    \[
    \int_\edge \normalProj v_h p \ds = \int_\edge \frac{\partial v_h}{\partial \normal_\edge}p \ds \qquad \forall p \in \mathbb{P}_{\polOrder-2}(\edge),
    \]
    and $\normalProj q = \nicefrac{\partial q}{\partial\normal_\edge}|_\edge $ for all $ q \in \poly(\element) $.
\end{itemize}
\end{assumption}
There are multiple ways of defining the value, edge, and edge normal projections such that Assumption~\ref{assumption:projections} is satisfied. An example choice, based on constrained least squares problems, can be found in \citet{dedner2022robust, dedner2024framework}, where we refer the reader for more details.

Now, we are able to define the gradient $\gradProj$ and Hessian $\hessProj$ projections.

\begin{definition}
    The \emph{gradient projection} $ \gradProj : \localEnlargedVemSpace \to [\poly[\polOrder-1](\element)]^2 $ is defined as
    \[
        \int_\element \gradProj v_h \cdot \vect{p} \dx = - \int_\element \valueProj v_h \diver \vect{p} \dx + \sum_{\edge \subset \partial \element} \int_\edge \edgeProj v_h \vect{p}\cdot\normal_\edge \ds, \quad \forall \vect{p} \in [\poly[\polOrder-1](\element)]^2.
    \]
    The \emph{Hessian projection} $ \hessProj : \localEnlargedVemSpace \to [\poly[\polOrder-2](\element)]^{2 \times 2 }$ is defined as
    \begin{align*}
        \int_\element \hessProj v_h : p \dx &= - \int_\element \gradProj v_h \cdot (\diver p) \dx \\&\quad+ \sum_{\edge \subset \partial \element} \int_\edge \left( \normalProj v_h( \normal_\edge \otimes \normal_\edge) : p + \partial_\edge (\edgeProj v_h) (\tangential_\edge \otimes \normal_\edge) : p \right) \ds,
    \end{align*}
    for all $p \in [\poly[\polOrder-2](\element)]^{2 \times 2 }$, where $\partial_\edge$ denotes the derivative along the edge $\edge$.
\end{definition}
\begin{remark}
    For the numerical experiments conducted in Section~\ref{sec:numerics}, the definitions for the value, gradient, Hessian, edge, and normal projections in \citet{dedner2024framework} are used.
\end{remark}

\noindent For $\element\in\mesh$, we can now define the local virtual element space $\localVem$ as
\begin{equation}
  \localVem \coloneqq \left\{ v_h \in \localEnlargedVemSpace : \int_\element (v_h - \valueProj v_h) p = 0 \; \forall p \in \poly(\element) \setminus \poly[\polOrder-4](\element)\right\}.  \label{eqn:localvem}
\end{equation}
This space is characterized by the set of degrees of freedom $\Lambda_\element$ described by the dof tuple $(0, -1, \polOrder-2, \polOrder-2, \polOrder-4)$; cf. \citet{dedner2022robust}. Note that this set of degrees of freedom is identical to the set used in \citet{zhao2016nonconforming} for the plate bending problem. As shown in \citet{dedner2022robust} we have the following crucial property for the projection operators.
\begin{lemma} \label{lemma:proj_eq_l2proj}
     Assume that the value, edge, and edge normal projections satisfy Assumption~\ref{assumption:projections}. Then, for any $v_h\in\localVem$, it holds that
    \[
    \valueProj v_h = \LtwoProjEle[\polOrder] (v_h),\qquad
    \gradProj v_h = \LtwoProjEle[\polOrder-1] (\grad v_h),\quad\text{and}\quad
    \hessProj v_h = \LtwoProjEle[\polOrder-2] (\hess v_h);
    \]
    i.e., the value, gradient, and hessian projections are the $\Lp$-orthogonal projections of order $\polOrder$ of the value, order $\polOrder-1$ of the gradient, and order $\polOrder-2$ of the Hessian, respectively. Additionally, as $\poly(\element) \subset \localVem$, it holds that for all $p \in \poly(\element)$
    \[
        \valueProj p =  p,\qquad
        \gradProj p = \grad p,\quad\text{and}\quad
        \hessProj p = \hess p.
    \]
\end{lemma}
\begin{corollary}\label{corol:proj_eq_l2proj}
    For any $v_h \in \localVem$ the value, gradient, and Hessian projections are continuous with respect to the $\Lp[2]$-norm; i.e.,
    \begin{align*}
        \norm{\valueProj v_h}_{0,\element} &= \norm{\LtwoProjEle(v_h)}_{0,\element} \leq \norm{v_h}_{0,\element}, \\
        \norm{\gradProj v_h}_{0,\element} &= \norm{\LtwoProjEle[\polOrder-1](\grad v_h)}_{0,\element} \leq \norm{\grad v_h}_{0,\element} \leq \norm{v_h}_{1,\element}, \\
        \norm{\hessProj v_h}_{0,\element} &= \norm{\LtwoProjEle[\polOrder-2](\hess v_h)}_{0,\element} \leq \norm{\hess v_h}_{0,\element} \leq \norm{v_h}_{2,\element}.
    \end{align*}
\end{corollary}
\begin{proof}
    This follows directly from Lemma~\ref{lemma:proj_eq_l2proj} and the basic properties of $\Lp$-projections.
\end{proof}
We note the following results from \citet{brenner2017some} and \citet{huang2021medius}, respectively, which will be crucial for the proofs in Section~\ref{sec:nonlinear_error}.
\begin{lemma}[{\citet[Eqn. 2.8]{brenner2017some}}]\label{lemma:linfty_bound}
    Under Assumption~\ref{assumption:mesh}\ref{assumption:mesh:star_shaped}, for any $u \in \Hk{2}(\element)$, the following inequality holds:
    \begin{equation*}
        \|u\|_{0,\infty,\element} \leq C\left( h^{-1}_\element\|u\|_{0,\element} +  |u|_{1,\element} + h_K|u|_{2,\element}\right) \qquad \forall \element \in \mathcal{T}_h.
    \end{equation*}
\end{lemma}
\begin{lemma}[{\citet[Lemma 3.2]{huang2021medius}}]\label{l3.9}
    For any $u \in \Hk{2}(\element)$, there exists $\delta > 0$ and constant $C(\delta) > 0$ dependent on $\delta$ such that 
    \begin{equation*}
        |u|_{1,\element} \leq \delta h_\element|u|_{2,\element} + C(\delta)h_K^{-1}\|u\|_{0,\element}.
    \end{equation*}
\end{lemma}
\begin{remark}
    The inequality mentioned in Lemma~\ref{l3.9} can be easily extended to the full $\Hk{1}(\element)$ and $\Hk{2}(\element)$ norms using standard scaling arguments on a reference element.
\end{remark}
\begin{corollary}\label{corol:h2full}
    For all $v \in \localVem$, it holds that
    \[
        \|v\|_{0,\infty,\element} \leq Ch^{-1}_\element\|v\|_{0,2,\element}. 
    \]
\end{corollary}
\begin{proof}
    We get the estimate by substituting the inverse inequality {\citep[Theorem 4.1]{huang2024c}}
    \begin{align*}
        |u|_{2,\element} \leq Ch^{-2}_\element\|u\|_{0,2,\element}
    \end{align*}
    for any $u \in \localVem$, where $C$ is a constant independent of $h$, and Lemma~\ref{l3.9} into Lemma \ref{lemma:linfty_bound}.
\end{proof}

\subsection{Global VEM space and projection operators}
Using the local spaces we can now define the global virtual element spaces
\begin{align*}
    \vemSpace{g} &\coloneqq \left\{ v_h \in \Hnc{g}(\mesh) :
    v_h|_\element \in \localVem,\;\forall \element \in \mesh \right\},\\
    \vemSpace{0} &\coloneqq \left\{ v_h \in \Hnc{0}(\mesh) :
    v_h|_\element \in \localVem,\;\forall \element \in \mesh \right\}.
\end{align*}
The global dofs can be defined in the same way as for the local space. Note here that $\vemSpace{g},\vemSpace{0} \not\subset \Hk{2}(\Omega)$ but $\vemSpace{g}$ and $\vemSpace{0}$ are $\Cp$-conforming. We note that $\poly(\mesh) \subset \vemSpace{g} \subset \Hnc{g}(\mesh)$ and $\poly(\mesh) \subset \vemSpace{0} \subset \Hnc{0}(\mesh)$.

\begin{definition}
    We denote by $\valueGlobal$, $\gradGlobal$, and $\hessGlobal$ the global value, gradient, and Hessian projections, respectively, where
\[
(\valueGlobal v_h)|_\element := \valueProj (v_h|_\element),\qquad
(\gradGlobal v_h)|_\element := \gradProj (v_h|_\element),\quad\text{and}\quad
(\hessGlobal v_h)|_\element := \hessProj (v_h|_\element).
\]
for all $v_h \in \vemSpace{0}\cup\vemSpace{g}$.
\end{definition}
From \citet{zhao2016nonconforming,dedner2022robust}, we have the following result.
\begin{lemma} \label{lemma:interpolation}
    Under Assumptions~\ref{assumption:mesh}\ref{assumption:mesh:shape_regular}\&\ref{assumption:mesh:star_shaped}, for $\polOrder \geq 0$ and for every $w \in \Hk{m}(\element)$ with $1 \leq m \leq \polOrder + 1$, there exists a function $w_I \in \vemSpace{i}$, where $i \in \{g,0\}$, such that
    \[
        \norm{w -  w_I}_{s,K} \leq C h_\element^{m-s} \seminorm{w}_{m,K}, \qquad \text{for } s = 0,\dots,2.
    \]
\end{lemma}

\begin{lemma}
    There exists a constant $C>0$ independent of $h$ such that for any $w \in \vemSpace{0}$, it holds that
    \begin{equation*}\label{inverse}
        \norm{w}_{0,2} \leq C \seminorm{w}_{1,h}.
    \end{equation*}
\end{lemma}
\begin{proof}
    The proof follows from the Poincar\'e inequality and the fact that as $w \in H^1_0(\Omega)$, the discrete $H^1$ norm and
    $H^1$ norm are equivalent.
\end{proof}

\subsection{VEM discrete forms}
We now define the discrete forms of \eqref{eqn:variational:A} and \eqref{eqn:linear:variational:A} using the virtual spaces and projection operators defined above. We first note that the following assumption holds for the continuous forms.
\begin{assumption}\label{assumption:elementwise_operator}
    We assume that there exists forms $\AQL^\element$, $\aQL^\element$, $\bQL^\element$, and $\AL^\element$ which are the restrictions of $\AQL$, $\aQL$, $\bQL$, and $\AL$ to the element $\element\in\mesh$ such that
    \begin{alignat*}{2}
        \AQL(u,v) = \sum_{\element\in\mesh}\AQL^\element(u,v) &= \sum_{\element\in\mesh}\aQL^\element(u,v) + \bQL^\element(u,v) && \qquad\forall u \in V, v \in W \\
        \AL(u,v) &= \sum_{\element\in\mesh}\AL^\element(u,v) && \qquad\forall u,v \in W .
    \end{alignat*}
    Additionally, for $u_h,v_h\in\localVem$, Propositions~\ref{prop:linear_weak:continuity} and~\ref{prop:linear_weak:coercivity} hold element-wise; i.e.,
    \begin{align}
        \AL^\element(u_h,v_h) &\leq C_2(\varepsilon)\norm{u_h}_{2,\element}\norm{v_h}_{2,\element}, \label{eqn:linear:continuity} \\
        \AL^\element(v_h,v_h) &\geq C_3(\varepsilon)\norm{v_h}^2_{2,\element}. \label{eqn:linear:coercivity}
    \end{align}    
\end{assumption}
To define the discrete forms, we first consider for an element $\element\in\mesh$ the forms
\begin{align}
\AQLh^\element(u_h,v_h) &\coloneq \aQLh^\element(u_h,v_h)+\bQLh^\element(u_h,v_h) \qquad \forall u_h \in\vemSpace{g}, v_h \in \vemSpace{0},  \label{eqn:AQLh:element}  \\
\ALh^\element(u_h,v_h) &\coloneqq \varepsilon\int_\element \hessProj u_h : \hessProj v_h \dx +\int_\element \uEpsCof \gradProj u_h \cdot \gradProj v_h \dx \label{eqn:ALh:element} \\
& \quad + \StabL(u_h - \valueProj u_h, v_h - \valueProj v_h), \qquad\forall u_h, v_h \in \vemSpace{0}, \nonumber
\end{align}
where
\begin{align*}
\aQLh^\element(u_h,v_h) &\coloneqq -\varepsilon\int_\element \hessProj u_h : \hessProj v_h \dx + \StabQLa(u_h - \valueProj u_h, v_h - \valueProj v_h) \\
\bQLh^\element(u_h,v_h) &\coloneqq \int_\element \det(\hessProj u_h) \valueProj v_h \dx + \StabQLb(u_h - \valueProj u_h, v_h - \valueProj v_h)
\end{align*}
and $\StabQLa(\cdot,\cdot)$, $\StabQLb(\cdot,\cdot)$, and $\StabL(\cdot,\cdot)$ are \emph{dofi-dofi} stabilizations \citep{beirao2013basic} defined as
\begin{align}
    \StabQLa(u_h,v_h) &\coloneqq
    -\varepsilon h_\element^{-2}
    \sum_{\lambda\in\Lambda^\element} \lambda(u_h)\lambda(v_h), \label{eqn:stab:ql:a}\\
    \StabQLb(u_h,v_h) &\coloneqq
    \norm{\cof(\hessProj[\element,0] u_h)}_{0,\infty,\element}
    \sum_{\lambda\in\Lambda^\element} \lambda(u_h)\lambda(v_h), \label{eqn:stab:ql:b}\\
    \StabL(u_h,v_h) &\coloneqq
    \left(\varepsilon h_K^{-2} + \norm{\uEpsCof}_{0,\infty,\element}\right)
    \sum_{\lambda\in\Lambda^\element}  \lambda(u_h)\lambda(v_h), \label{eqn:stab:linear}
\end{align}
with $\hessProj[\element,0]$ denoting the Hessian projection onto the space of constant polynomials; i.e, $\hessProj[\element,0] : \localEnlargedVemSpace \to [\poly[0](\element)]^{2 \times 2 }$.
Additionally, we define
\begin{equation}\label{eqn:stab:ql}
    \StabQL(u_h,v_h) \coloneqq \StabQLa(u_h,v_h) + \StabQLb(u_h,v_h).
\end{equation}
There exist positive constants $c_*$, $c^*$, $d_*$, and $d^*$, independent of $h$ and $\element$, such that 
\begin{align}
    c_*\AL^\element(v_h,v_h) &\leq \abs{\StabL(v_h,v_h)} \leq c^*\AL^\element(v_h,v_h), \label{eqn:stab:ql:equivalence}\\
    d_*\aQL^\element(v_h,v_h) &\leq \abs{\StabQL(v_h,v_h)} \leq d^*\aQL^\element(v_h,v_h). \label{eqn:stab:linear:equivalence}
\end{align}

We can now define the global discrete forms $\AQLh:\vemSpace{g}\times\vemSpace{0}\to\real$ and $\ALh:\vemSpace{0}\times\vemSpace{0}\to\real$ of \eqref{eqn:variational:A} and \eqref{eqn:linear:variational:A}, respectively, as
\begin{align}
    \AQLh(u_h,v_h) &\coloneqq \sum_{\element \in \mesh}\AQLh^\element(u_h,v_h)
    = \sum_{\element \in \mesh}\left(\aQLh^\element(u_h,v_h) + \bQLh^\element(u_h,v_h)\right), \label{eqn:AQLh}\\
    \ALh(u_h,v_h) &\coloneqq \sum_{\element \in \mesh}\ALh^\element(u_h,v_h). \label{eqn:ALh}
\end{align}

We assume the following stability property holds for the above forms.
\begin{assumption}
    There exist positive constants $\alpha^*$, $\alpha_*$, $\beta^*$, and $\beta_*$, independent of $h$, $\polOrder$ and $\element\in\mesh$ such that for all $v_h\in \localVem$
    \begin{align}
        \alpha_*\AL^\element(v_h,v_h) &\leq \ALh^\element(v_h,v_h) \leq \alpha^*\AL^\element(v_h,v_h), \label{eqn:assumption:equivalence:ql} \\
        \beta_*\aQL^\element(v_h,v_h) &\leq \aQLh^\element(v_h,v_h) \leq \beta^*\aQL^\element(v_h,v_h). \label{eqn:assumption:equivalence:linear}
    \end{align}
\end{assumption}

We can show continuity and coercivity of the discrete linear form. 
\begin{lemma} For all $u_h,v_u\in\vemSpace{0}$ and $w_h,y_h \in \vemSpace{g}$ it holds that
    \begin{align}
        \ALh(u_h, v_h) &\leq C^*_2(\varepsilon) \norm{u_h}_{2,h}\norm{v_h}_{2,h}, \label{eqn:linear:discrete:continuity}\\
        \ALh(v_h, v_h) &\geq C^*_3(\varepsilon)\norm{v_h}^2_{2,h}, \label{eqn:linear:discrete:coercivity}\\
        \aQLh(w_h, y_h) &\leq C(\varepsilon + \beta^*)\norm{w_h}_{2,h}\norm{y_h}_{2,h},
        \label{eqn:nonlinear:discrete:continuity}
    \end{align}

    where $C^*_2(\varepsilon) = C_2(\varepsilon)+\alpha^* = \mathcal{O}(\varepsilon^{-1})$ and $C^*_3(\varepsilon) = (C_3(\varepsilon)+\alpha_*) = \mathcal{O}(\varepsilon)$.

\end{lemma}
\begin{proof}
    These results can be obtained by applying \eqref{eqn:poincare}, \eqref{eqn:assumption:equivalence:ql}, Lemma~\ref{inverse} and \eqref{eqn:assumption:equivalence:linear} respectively.
\end{proof}
\subsection{The discrete problem}
We can now consider the discrete virtual element method formulation for the vanishing moment equation \eqref{eqn:vanishing_moment}--\eqref{eqn:vanishing_moment:extra_bc}. For $\polOrder\geq 2$, find $\uh\in \vemSpace{g}$ such that
\begin{equation}\label{eqn:vem}
    \AQLh(\uh,w_h) = \sum_{\element\in\mesh}\int_\element f_h w_h \dx + \varepsilon \sum_{\edge\in\edgesB} \int_e \left( \frac{\partial^2 g}{\partial \tangential_\edge^2} -\varepsilon \right) \normalProj w_h \ds
\end{equation}
for all $w_h\in \vemSpace{0}$, where $\AQLh(\emptyarg,\emptyarg)$ is defined in \eqref{eqn:AQLh} and $f_h\in \vemSpace{g}^\prime$ is defined as a suitable projection of $f$.

In order to analyze the method, we define a linear operator $T_h : \vemSpace{g}\to\vemSpace{g}$ such that, for any $v_h\in\vemSpace{g}$, $T_h(v_h)$ is defined as the solution of the problem
\begin{multline}\label{eqn:T}
    \ALh(T_h(v_h)-v_h,w_h) \\= \AQLh(v_h,w_h) - \sum_{\element\in\mesh}\int_\element f_h w_h\dx - \varepsilon \sum_{\edge\in\edgesB} \int_e \left( \frac{\partial^2 g}{\partial \tangential_\edge^2} -\varepsilon \right) \normalProj w_h \ds
\end{multline}
for all $w_h\in\vemSpace{0}$. Clearly, a fixed point $v_h$ of $T_h$ is a solution of \eqref{eqn:vem}.
Additionally, we can show the existence of $T(v_h)-v_h\in\vemSpace{0}$ for any $v_h\in\vemSpace{g}$ via the following result.
\begin{theorem}[Existence and uniqueness of linearized discrete solution] \label{thm:existence:linear}
    There exists a unique solution $v_h \in V_{h,\polOrder,0}$ to the discrete linear virtual element formulation
    \begin{equation} \label{eqn:vem:linear}
        \ALh(v_h,w_h) = (\varphi, w_h) + \varepsilon \sum_{\edge \in \edgesB} \int_\edge \psi\normalProj w_h \ds  
    \end{equation}
    for all $w_h \in \vemSpace{0}$, where $\varphi\in \Lp(\Omega)$ and $\psi\in\Hk{\nicefrac{-1}{2}}(\partial \Omega)$.
\end{theorem}
\begin{proof}
    Continuity and coercivity of $\ALh(\emptyarg,\emptyarg)$ follows from \eqref{eqn:stab:ql:equivalence} and \eqref{eqn:linear:continuity}--\eqref{eqn:linear:coercivity}. Hence, an application of the Lax-Milgram Theorem shows the existence of a unique solution. 
\end{proof}
\begin{remark}
    Due to the Riesz representation theorem, there exists a $\varphi\in \Lp(\Omega)$ such that the right-hand side of \eqref{eqn:T} is equivalent to the right-hand side of \eqref{eqn:vem:linear}.
\end{remark}

\section{Analysis of linearized Monge--Amp\`ere}\label{sec:linear_error}
In this section, we show a Strang-type error bound and optimal $\Hk{2}$-error estimates for the linearized discrete form \eqref{eqn:vem:linear}. These results will be used in Section~\ref{sec:nonlinear_error} to prove the existence and uniqueness of a solution to the quasilinear discrete form \eqref{eqn:vem} in the virtual element space $\vemSpace{0}$ with order of approximation $\polOrder\geq 2$, along with optimal error estimates. 

We start this analysis by deriving the following Strang-type estimate; cf. \citet{cangiani2017conforming}.
\begin{theorem} \label{thm:strang}
Let $v\in W$ be the solution of the linear variational form \eqref{eqn:linear:variational} and $v_h\in\vemSpace{0}$ be the solution of the linearized virtual element formulation \eqref{eqn:vem:linear}; then, it holds that
\begin{align}
    C_3(\varepsilon)\alpha_*&\norm{v - v_h}_{2,h} \nonumber\\
    &\leq (C_2(\varepsilon) + C_3(\varepsilon)\alpha_*) \inf_{w_h\in\vemSpace{0}}\norm{v - w_h}_{2,h} \label{eqn:strang} \\
    &\quad+ \inf_{p \in \poly(\mesh)} \left(2 C_2(\varepsilon)\norm{v - p}_{2,h} + \sum_{\element \in \mesh} \sup_{\substack{\delta_h \in \vemSpace{0}\\\delta_h\neq 0}} \frac{\abs{\AL^\element(p,\delta_h) - \ALh^\element(p,\delta_h)}}{\norm{\delta_h}_{2,h}} \right) \nonumber\\
    &\quad+ \sup_{\substack{\delta_h \in \vemSpace{0}\\\delta_h\neq 0}} \frac{\abs{E(v, \delta_h)}}{\norm{\delta_h}_{2,h}} + \sup_{\substack{\delta_h \in \vemSpace{0}\\\delta_h\neq 0}}\frac{1}{\norm{\delta_h}_{2,h}}\abs*{\varepsilon \displaystyle\sum_{\edge \in \edgesB} \int_\edge \psi \left(\normalProj \delta_h-\frac{\partial \delta_h}{\partial \normal_e}\right) \ds} \nonumber
\end{align}
where the \emph{nonconformity error} is given by 
\[
    E(v, \delta_h) \coloneqq (\varphi, \delta_h) - A_L(v, \delta_h)+  \varepsilon \sum_{\edge\in\edgesB} \int_\edge \psi\frac{\partial \delta_h}{\partial \normal_e} \ds.
\]
\end{theorem}
\begin{proof} Let $w_h\in\vemSpace{0}$ be arbitrary and $\delta_h \coloneqq v_h - w_h$; then, by \eqref{eqn:linear:coercivity}, \eqref{eqn:assumption:equivalence:ql}, and \eqref{eqn:vem:linear}, we have that
\begin{align*}
    C_3(\varepsilon)\alpha_*&\norm{v_h - w_h}^2_{2,h} \\&\leq \alpha_*\AL(\delta_h,\delta_h) \\
    &\leq \ALh(v_h,\delta_h) - \ALh(w_h,\delta_h) \\
    &= (\varphi, \delta_h) + \varepsilon \sum_{\edge\in\edgesB} \int_\edge \psi\normalProj v_h \ds 
    - \sum_{\element \in \mesh} \ALh^\element(w_h,\delta_h)\\
    &= \left[\varepsilon \sum_{\edge\in\edgesB}\int_\edge \psi \frac{\partial\delta_h}{\partial\normal_\edge}\ds+(\varphi,\delta_h) -\AL(v,\delta_h)\right] + \sum_{\element\in\mesh} \left(\AL^\element(p,\delta_h) - \ALh^\element(p,\delta_h) \right) \\
    &\quad+ \varepsilon \sum_{\edge \in \edgesB} \int_\edge \psi\left(\normalProj v_h-\frac{\partial\delta_h}{\partial\normal_\edge}\right) \ds+ \sum_{\element\in\mesh} \left(\AL^\element(v- p,\delta_h)-\ALh^\element(w_h - p,\delta_h)\right).
\end{align*}
for any $p \in \poly(\mesh)$. Applying \eqref{eqn:linear:continuity} and \eqref{eqn:linear:discrete:continuity} we have that
\begin{align*}
    C_3(\varepsilon)\alpha_*&\norm{v_h - w_h}_{2,h}^2 \\&\leq \abs*{\varepsilon \displaystyle\sum_{\edge \in \edgesB} \int_\edge \psi\frac{\partial \delta_h}{\partial \normal_\edge} \ds + (\varphi, \delta_h) - \AL(v,\delta_h)}+ \sum_{\element \in \mesh}\abs*{\AL^\element(p,\delta_h) - \ALh^\element(p,\delta_h)}\\
    & \quad+ \abs*{\varepsilon \displaystyle\sum_{\edge \in \edgesB} \int_\edge \psi \left(\normalProj \delta_h-\frac{\partial \delta_h}{\partial \normal_\edge}\right) \ds} + C_2(\varepsilon)\left(\norm{w_h - p}_{2,h}+\norm{v - p}_{2,h}\right)\norm{\delta_h}_{2,h}
\end{align*}
The result follows by dividing through by $\norm{\delta_h}_{2,h}$ and applications of the triangle inequality.
\end{proof}
We are now able to derive $\Hk{2}$-error bounds for the VEM solution of the linearized problem by bounding each term in the right-hand side of \eqref{eqn:strang}. Throughout this section we assume that the solution $v$ of the linear problem \eqref{eqn:linear}--\eqref{eqn:linear:extra_bc} satisfies $v\in \Hk{s+1}(\Omega)$, for $s\geq 3$.

For the nonconformity error, we follow the steps of the proofs from \citet[Corollary 5.3, Corollary 5.4 and Theorem 5.5]{dedner2022robust}. 
\begin{lemma}\label{lemma:nonconforming:error}
    Assume $v \in \Hk{4}(\Omega)$ satisfies \eqref{eqn:linear}--\eqref{eqn:linear:extra_bc}, and as $\vemSpace{0} \subset \Hk{1}_0(\Omega)$, we have that
    \[
        E(v,\delta_h) = \varepsilon \sum_{\edge \in \edgesI}\int_e  \left( \laplace v -\frac{\partial^2v}{\partial\tangential_\edge^2}  \right) \jmp*{\frac{\partial\delta_h}{\partial\normal_\edge}} \ds.
    \]
    for all $\delta_h\in\vemSpace{0}$.
\end{lemma}
\begin{proof}
By applying integration by parts we get that
\begin{align*}
    \AL(v,\delta_h) &= \sum_{\element\in\mesh}\varepsilon \int_\element \hess v:\hess \delta_h\dx +  \sum_{\element\in\mesh}\int_\element \uEpsCof \grad v\cdot \grad\delta_h \dx\\
    &=- \sum_{\element\in\mesh}\varepsilon \int_\element \sum_{i,j = 1}^2 \frac{\partial}{\partial x_j}\left(\dfrac{\partial^2 v}{\partial x_i \partial x_j}\right)\dfrac{\partial \delta_h}{\partial x_i} \dx+\sum_{\element\in\mesh}\varepsilon \int_{\partial \element} \sum_{i,j = 1}^2 \dfrac{\partial^2 v}{\partial x_i \partial x_j}\dfrac{\partial \delta_h}{\partial x_i}\eta_j \ds\\
    &\quad-\sum_{\element\in\mesh}\int_\element \diver (\uEpsCof \grad v)\delta_h \dx+ \sum_{\element\in\mesh}\int_{\partial \element}\uEpsCof \frac{\partial v}{\partial \normal_\element}\delta_h\ds.
\end{align*}
Since we assume that $v \in \Hk{4}(\Omega)$, we use \eqref{eqn:linear} and an application of integration by parts to see that
\begin{align*}
    (\varphi, \delta_h) &= -\sum_{\element\in\mesh}\varepsilon \int_\element \sum_{i,j = 1}^2 \frac{\partial}{\partial x_j}\left(\dfrac{\partial^2 v}{\partial x_i \partial x_j}\right)\dfrac{\partial \delta_h}{\partial x_i}\dx \\
    &\quad+\sum_{\element\in\mesh}\varepsilon \int_{\partial \element} \sum_{i,j = 1}^2 \frac{\partial}{\partial x_j}\left(\dfrac{\partial^2 v}{\partial x_i \partial x_j}\right)\delta_h\eta_i \ds - \sum_{\element\in\mesh}\int_\element \diver (\uEpsCof \grad v)\delta_h\dx
\end{align*}
Therefore, using \eqref{eqn:linear:extra_bc}, we can write the nonconforming error as
\begin{align*}
    E(v,\delta_h) &= \sum_{\element\in\mesh} \varepsilon \int_{\partial\element} \sum_{i,j = 1}^2 \dfrac{\partial }{\partial x_j}\left(\dfrac{\partial^2 v}{\partial x_i \partial x_j}\right)\delta_h\eta_i  \ds -  \sum_{\element\in\mesh}\varepsilon \int_{\partial\element} \sum_{i,j = 1}^2 \dfrac{\partial^2 v}{\partial x_i \partial x_j}\dfrac{\partial \delta_h}{\partial x_i}\eta_j \ds\\
    & \qquad -\sum_{\element\in\mesh}\int_{\partial \element}\uEpsCof \frac{\partial v}{\partial \normal_\element}\delta_h\ds+\varepsilon \sum_{\edge \in \edgesB} \int_\edge  \laplace v \frac{\partial\delta_h}{\partial\normal_\edge} \ds
\end{align*}
The first and third terms disappear as $v\in\Hk{4}(\Omega)$ and $\delta_h\in\vemSpace{0}\subset\Hk{1}_0(\Omega)$. For the second term we note that
\begin{align*}
    \sum_{\element\in\mesh}\varepsilon\int_{\partial\element} &\sum_{i,j = 1}^2 \dfrac{\partial^2 v}{\partial x_i \partial x_j}\dfrac{\partial \delta_h}{\partial x_i}\eta_j \ds\\&= \sum_{\element\in\mesh}\varepsilon \int_{\partial\element} \left(\laplace v - \dfrac{\partial^2 v}{\partial\tangential^2_\element}\right) \dfrac{\partial \delta_h}{\partial \normal_\element}\ds + \sum_{\element\in\mesh}\varepsilon \int_{\partial\element}\dfrac{\partial ^2 v}{\partial \normal_\element\partial \tangential_\element} \dfrac{\partial \delta_h}{\partial \tangential_\element} \ds\\
    &=\sum_{\edge\in\edges}\varepsilon \int_\edge \left(\laplace v - \dfrac{\partial^2 v}{\partial\tangential^2_\edge}\right) \jmp*{\dfrac{\partial \delta_h}{\partial \normal_\edge}}\ds + \sum_{\edge\in\edges}\varepsilon \int_{\edge}\dfrac{\partial ^2 v}{\partial \normal_\edge\partial \tangential_\edge} \jmp*{\dfrac{\partial \delta_h}{\partial \tangential_\edge}} \ds.
\end{align*}
Therefore, it follows that
\begin{align*}
    E(v,\delta_h) &=\varepsilon \sum_{\edge \in \edgesB} \int_e  \laplace v  \frac{\partial\delta_h}{\partial\normal_\edge} \ds - \sum_{\edge\in\edges}\int_\edge \varepsilon \left(\laplace v -\frac{\partial^2v}{\partial\tangential^2_\edge}  \right) \jmp*{\frac{\partial\delta_h}{\partial\normal_\edge}}\ds\\
    &=\varepsilon \sum_{\edge \in \edgesI}\int_e \left( \frac{\partial^2v}{\partial\tangential^2_\edge} -\laplace v \right) \jmp*{\frac{\partial\delta_h}{\partial\normal_\edge}} \ds,
\end{align*}
using the that fact that $\jmp*{\nicefrac{\partial \delta_h}{\partial \tangential_\edge}}=0$ as $\delta_h\in V_{h,\polOrder,0} \subset H^1_0(\Omega)$ and $\nicefrac{\partial^2v}{\partial\tangential^2}=0$ on $\partial\Omega$ due to \eqref{eqn:linear:dbc}.
\end{proof}
\begin{lemma}\label{lemma:nonconforming:error_bound}
    Suppose that Assumption~\ref{assumption:mesh} holds and $v \in \Hk{s+1}(\Omega)$, $s \geq 3$, is the solution of \eqref{eqn:linear}--\eqref{eqn:linear:extra_bc}. Then, there exists a positive constant $C$, independent of $h$ and $\varepsilon$, such that the nonconformity error satisfies the following estimate
    \[
        | E(v,\delta_h)| \leq C\varepsilon h^{r-1}\|v\|_{r+1}\|\delta_h\|_{2,h},
    \]
    for all $\delta_h\in\vemSpace{0}$ where $r \coloneqq \min(\polOrder,s)$.
\end{lemma}
\begin{proof}
    From the definition of the nonconforming spaces we have for every edge $e\in\edgesI$
    \[
    \int_\edge \jmp*{\frac{\partial\delta_h}{\partial\normal_\edge}}p\ds=0, \qquad \text{for all } p\in\poly[\polOrder-2](\edge).
    \]
    Combining this result with properties of the $\Lp$-projection, we have from Lemma~\ref{lemma:l2proj} that
    \begin{align*}
        | E(v,\delta_h)| &\leq \abs*{ \varepsilon \sum_{\edge\in\edgesI}\int_e \left( \frac{\partial^2v}{\partial\tangential^2_\edge}-\laplace v -\LtwoProj{\edge}{\polOrder-2}\left(\frac{\partial^2v}{\partial\tangential^2_\edge}-\laplace v\right) \right) \left(\jmp*{\frac{\partial\delta_h}{\partial\normal_\edge}}-\LtwoProj{\edge}{0}\jmp*{\frac{\partial\delta_h}{\partial\normal_\edge}}\right) \ds} \\
        &\leq \varepsilon \sum_{\edge\in\edgesI} \norm*{\frac{\partial^2v}{\partial\tangential^2_\edge}-\laplace v -\LtwoProj{\edge}{\polOrder-2}\left(\frac{\partial^2v}{\partial\tangential^2_\edge}-\laplace v\right)}_{0,\edge} \norm*{\jmp*{\frac{\partial\delta_h}{\partial\normal_\edge}}-\LtwoProj{\edge}{0}\jmp*{\frac{\partial\delta_h}{\partial\normal_\edge}}}_{0,\edge} \\
        &\leq C \varepsilon \sum_{\edge\in\edgesI} h^{r - \nicefrac{3}{2}}_\edge h^{\nicefrac{1}{2}}_\edge\norm*{\frac{\partial^2v}{\partial\tangential^2_\edge} - \laplace v }_{r-1,\omega_\edge} \seminorm*{\frac{\partial\delta_h}{\partial\normal_\edge}}_{1,\omega_\edge}\\
        &\leq C \varepsilon h^{r - 1}\|v\|_{r+1}|\delta_h|_{2,h}.
    \end{align*}
\end{proof}
We now prove the polynomial consistency error similarly to \citet{cangiani2017conforming}.
\begin{lemma}\label{lemma:polyconsistency}
    Let Assumption~\ref{assumption:mesh} hold and assume $v \in \Hk{s+1}(\Omega)$, $s \geq 3$, is the solution of \eqref{eqn:linear}--\eqref{eqn:linear:extra_bc}; then, 
    \[
        \inf_{p \in \poly(\mesh)} \left(2C_2(\varepsilon)\|v - p\|_{2,h} + \sum_{\element\in\mesh} \sup_{\substack{\delta_h \in \vemSpace{0}\\\delta_h\neq 0}} \frac{|\AL^\element(p,\delta_h) - \ALh^\element(p,\delta_h)|}{\|\delta_h\|_{2,h}} \right) \leq C_4(\varepsilon)h^{r-1}\|v\|_{r+1},
    \]
    where $r \coloneqq \min(\polOrder,s)$, $C_4 = \mathcal{O}\left(\varepsilon^{-1}\right)$ and the positive constant $C$ is independent of $h$ and $\varepsilon$.    
\end{lemma}
\begin{proof}
    As $\valueGlobal v \in \poly(\mesh)$, cf. Assumption~\ref{assumption:projections}, we have that
    \begin{multline*}
         \inf_{p \in \poly(\mesh)} \left(2C_2(\varepsilon)\|v - p\|_{2,h} + \sum_{\element\in\mesh} \sup_{\substack{\delta_h \in \vemSpace{0}\\\delta_h\neq 0}} \frac{|\AL^\element(p,\delta_h) - \ALh^\element(p,\delta_h)|}{\|\delta_h\|_{2,h}} \right)\\
         \leq 2C_2(\varepsilon)\|v - \valueGlobal v\|_{2,h} + \sum_{\element\in\mesh} \sup_{\substack{\delta_h \in \vemSpace{0}\\\delta_h\neq 0}} \frac{|\AL^\element(\valueProj v,\delta_h) - \ALh^\element(\valueProj v,\delta_h)|}{\|\delta_h\|_{2,h}}
    \end{multline*}
    From the definitions of $\ALh^\element$ and $\AL^\element$, we can apply Lemma~\ref{lemma:proj_eq_l2proj} to show that
    \begin{align*}
        |\AL^\element(\valueProj v,\delta_h) - \ALh^\element(\valueProj v,\delta_h)| 
        &\leq \abs*{\varepsilon\int_\element \left(\hess(\valueProj v) \hess\delta_h - \hessProj (\valueProj v) \hessProj \delta_h \right)\dx} \\&\quad+ \abs*{\int_\element \uEpsCof \left( \grad(\valueProj v) \grad\delta_h - \gradProj (\valueProj v)\gradProj \delta_h\right)\dx } \\
        &= \abs*{\varepsilon\int_\element \hess(\valueProj v)  \left(\hess\delta_h - \LtwoProj{\element}{\polOrder-2} (\hess \delta_h) \right)\dx }\\&\quad+ \abs*{\int_\element \uEpsCof \grad(\valueProj v)\left(\grad\delta_h - \LtwoProj{\element}{\polOrder-1} (\grad\delta_h)\right)\dx } \\
        & \eqqcolon E_1+E_2
    \end{align*}
    We note that the stabilization terms disappear due to Lemma~\ref{lemma:proj_eq_l2proj}. Using 
    Lemma~\ref{lemma:l2proj}, Lemma~\ref{lemma:proj_eq_l2proj}, and properties of the $\Lp$-projection we have that
\begin{align*}
    E_1 = \abs*{\varepsilon\int_\element \hess(\Pi^K_0v) \left((I -\LtwoProj{\element}{\polOrder-2})\hess\delta_h\right)\dx}
    &= \abs*{\varepsilon\int_\element \left((I -\LtwoProj{\element}{\polOrder-2})\hess(\Pi^K_0v)\right)\hess\delta_h\dx}\\
    &\leq \varepsilon \|(I -\LtwoProj{\element}{\polOrder-2})\hess(\Pi^K_0v)\|_{0,\element}\|\hess\delta_h\|_{0,\element}\\
    &\leq C\varepsilon h^{r-1} \|\hess\valueProj v\|_{r-1,\element}\|\delta_h\|_{2,\element}\\
    &\leq C\varepsilon h^{r-1} \|\valueProj v\|_{r+1,\element}\|\delta_h\|_{2,\element}\\
    &\leq C\varepsilon h^{r-1} \|v\|_{r+1,\element}\|\delta_h\|_{2,\element}.
\end{align*}
Similarly, and by using Proposition~\ref{prop:ueps_bounds}, we have that
\[
   E_2 \leq C h^{r-1} \norm{\uEpsCof}_{0,\infty,\element}\|v\|_{r+1,\element}\|\delta_h\|_{2,\element}
   \leq C h^{r-1} \varepsilon^{-1}\|v\|_{r+1,\element}\|\delta_h\|_{2,\element}  
\]
Combining $E_1$ and $E_2$, and applying Lemmas~\ref{lemma:l2proj} and~\ref{lemma:proj_eq_l2proj} to bound $\|v - \valueGlobal v\|_{2,h}$, noting that $C_2(\varepsilon)=C(\varepsilon + \varepsilon^{-1}) = \mathcal{O}\left(\varepsilon^{-1}\right)$, completes the proof.
\end{proof}
Finally, we bound the last term from \eqref{eqn:strang}.
\begin{lemma}\label{lemma:strang:boundary}
    Let Assumption~\ref{assumption:mesh} hold and assume $v \in \Hk{s+1}(\Omega)$, $s \geq 3$, is the solution of \eqref{eqn:linear}--\eqref{eqn:linear:extra_bc}. Then, there exists a positive constant $C$, independent of $h$ and $\varepsilon$, such that
\[
    \abs*{\varepsilon \sum_{\edge \in \edgesB} \int_\edge \psi \left(\normalProj \delta_h-\frac{\partial \delta_h}{\partial \normal_\edge}\right) \ds} \leq C\varepsilon h^{r-1}\|v\|_{r+1}\|\delta_h\|_{2,h}
\]
for all $\delta_h\in\vemSpace{0}$, where $r \coloneqq \min(\polOrder,s)$.
\end{lemma}
\begin{proof}
We first apply \eqref{eqn:linear:dbc}--\eqref{eqn:linear:extra_bc} similarly to the proof of Lemma~\ref{lemma:nonconforming:error}. Then, from the definition of the edge normal projection from Assumption~\ref{assumption:projections} and the fact that $\normalProj \delta_h = \LtwoProj{\edge}{\polOrder-1}(\nicefrac{\partial \delta_h}{\partial \normal_e})$, we have that
\begin{multline*}
 \abs*{\varepsilon \sum_{\edge \in \edgesB} \int_\edge \psi \left(\normalProj \delta_h-\frac{\partial \delta_h}{\partial \normal_\edge}\right) \ds} \\= \abs*{\varepsilon \sum_{\edge \in \edgesB} \int_\edge \left(\laplace v - \frac{\partial^2v}{\partial\tangential^2_\edge} -\LtwoProj{\edge}{\polOrder-2}\left(\laplace v - \frac{\partial^2v}{\partial\tangential^2_\edge} \right) \right) \left( \LtwoProj{\edge}{\polOrder-1}\left(\frac{\partial \delta_h}{\partial \normal_e}\right)-\frac{\partial \delta_h}{\partial \normal_\edge}\right) \ds}.
\end{multline*}
The proof then follows similarly to Lemma~\ref{lemma:nonconforming:error_bound}.
\end{proof}

We can now combine the above results to derive the following $\Hk{2}$-error bound for the solution of the linearized problem \eqref{eqn:vem:linear}.
\begin{theorem}[Linearized $H^2$-error estimate] \label{theorem:linear:h2}
    Suppose that Assumption~\ref{assumption:mesh} is satisfied. Let $v \in \Hk{s+1}(\Omega)$, $s \geq 3$, be the solution of \eqref{eqn:linear}--\eqref{eqn:linear:extra_bc} and assume that $\varphi \in H^{r-3}(\Omega)$. Let $v_h \in \vemSpace{0}$ be the corresponding virtual element solution to \eqref{eqn:vem:linear}; then, there exists a positive constant $C_5(\varepsilon)$, independent of $h$, such that
    \[
        \|v - v_h\|_{2,h} \leq C_5(\varepsilon) h^{r-1} (\|v\|_{r+1} +\|\varphi\|_{r-3}).
    \]
    where $C_5(\varepsilon) = C(C_3(\varepsilon)\alpha_{*})^{-1}(C_3(\varepsilon)\alpha_{*}+ \varepsilon+ C_4(\varepsilon)) = \mathcal{O}(\varepsilon^{-1})$ and $r \coloneqq \min(\polOrder,s)$.
\end{theorem}
\begin{proof}
    Let $v_I \in \vemSpace{0}$ be the interpolant of $v$ from Lemma~\ref{lemma:interpolation}; then, by standard scaling arguments, we have that 
    \[
        \inf_{w_h\in\vemSpace{0}} \norm{v-w_h}_{2,h} \leq \|v - v_I\|_{2,h} \leq Ch^{r-1} \|v\|_{r+1}.
    \]
    The proof is completed by applying this result and Lemmas~\ref{lemma:nonconforming:error}--\ref{lemma:strang:boundary} to Theorem~\ref{thm:strang}, noting from the proof of Lemma~\ref{lemma:polyconsistency} that $C_4\geq C_2$.
\end{proof}
\section{Analysis of quasilinear Monge--Amp\`ere}\label{sec:nonlinear_error}

In this section, we prove the existence, uniqueness, and optimal error bounds for the solution of the quasilinear vanishing moment VEM formulation \eqref{eqn:vem}.

We first state the following results that are required for the results in this section.
\begin{lemma}[{\citet[Lemmas 3.1 \& 5.4]{beirao2017stability}}]{\label{lemma:h1_bdry}}
    Let $\element\in\mesh$ be such that Assumption~\ref{assumption:mesh}\ref{assumption:mesh:shape_regular} holds; then, for all $v\in \Hk{1}(\element)$,
    \begin{align*}
        |v|_{\nicefrac{1}{2},\partial \element} &\leq |v|_{1,\element} \\
        h^{-1}_\element\|v\|^2_{0,\partial \element} &\leq C(h^{-2}_\element\|v\|^2_{0, \element} + |v|^2_{1, \element}).
    \end{align*}
    where $C>0$ is a constant independent of $h$.
\end{lemma}

\begin{lemma}[{\citet[Lemma 3.4]{beirao2017stability}}]{\label{lemma:linfty_bdry}}
    Let $\element\in\mesh$, such that Assumption~\ref{assumption:mesh}\ref{assumption:mesh:shape_regular}\&\ref{assumption:mesh:edges} hold; then, for all $v\in \vemSpace{0}^\element$,
    \begin{align*}
        \|v_h\|^2_{0,\infty,\partial \element} \leq C (h^{-1}_\element\|v_h\|^2_{0,\partial \element} + |v_h|^2_{1/2,\partial \element}).
    \end{align*}
    where $C>0$ is a constant independent of $h$.
\end{lemma}
\begin{lemma}[{\citet[Eqn. 4]{botti2025sobolev}}] \label{lemma:sobolev}
    Let $\element\in\mesh$ and $\polOrder\geq 0$; then, for all $v\in \Hk{\ell}(\Omega)$,
\[
    \|v\|_{0,2,\element} \leq C h^{\polOrder}_\element \|v\|_{\polOrder,2,\element}.
\]
\end{lemma}

\subsection{Fixed point}
In order to prove the existence and uniqueness of the solution, we study the linear operator $T_h$ defined by \eqref{eqn:T} as a fixed point of this operator satisfies \eqref{eqn:vem}. From Theorem~\ref{thm:existence:linear} and Theorem~\ref{theorem:linear:h2} we can see that $T_h$ is well defined; therefore, we need to prove that a fixed point exists and is unique, and hence the unique fixed point is then a unique solution of \eqref{eqn:vem}.  We show this by proving that $T_h$ satisfies the Banach fixed point theorem and, hence, that there exists a unique fixed point in a small neighborhood of $\uEpsInterp$ given by the ball
\begin{equation*}
    \ball{\uEpsInterp}{\zeta} \coloneqq \left\{v_h \in \vemSpace{g} : \|v_h - \uEpsInterp\|_{2,h} \leq \zeta \right\}
\end{equation*}
 of radius $\zeta := \zeta(h,\varepsilon) > 0$ and centre $\uEpsInterp \in \vemSpace{g}$, where $\uEpsInterp$ is the interpolant of the solution $\uEps$ of \eqref{eqn:vanishing_moment}--\eqref{eqn:vanishing_moment:extra_bc} in the VEM space $\vemSpace{g}$ defined by Lemma~\ref{lemma:interpolation}.

 We first state the following result about the mapping $T_h$.
\begin{lemma}\label{lemma:fixed_point:operator_bound}
Let $\uEps \in \Hk{s+1}(\Omega)$, $s\geq 3$, be the solution of \eqref{eqn:vanishing_moment}--\eqref{eqn:vanishing_moment:extra_bc} with interpolant $\uEpsInterp \in \vemSpace{g}$ defined by Lemma~\ref{lemma:interpolation} and $T_h$ the operator defined by \eqref{eqn:T}; then, there exists a positive constant $C_6(\varepsilon) = C(1+\varepsilon^{\nicefrac{-3}{2}}) = \mathcal{O}(\varepsilon^{\nicefrac{-3}{2}})$, independent of $h$, such that
\begin{align*}
        \norm{\uEpsInterp - T_h(\uEpsInterp)}_{2,h} \leq Ch^{r-1}(C_6(\varepsilon)\norm{\uEps}_{r+1}+\norm{f}_{r-3}),
\end{align*}
where $r \coloneqq \min(\polOrder,s)$.
\end{lemma}
\begin{proof}
    From \eqref{eqn:T} and \eqref{eqn:AQLh} we have, for all $w_h \in \vemSpace{0}$,
\begin{align*}
\ALh(\uEpsInterp &- T_h(\uEpsInterp),w_h) \\&
\leq \abs*{\AQLh(\uEpsInterp,w_h) - \AQL(\uEps,w_h)}
    +  \abs*{\sum_{\element\in\mesh}\int_\element(f-f_h)w_h\dx} \\
    & \quad + \abs*{\AQL(\uEps,w_h)- \sum_{\element\in\mesh}\int_\element fw_h\dx
    - \varepsilon \sum_{\edge\in\edgesB} \int_\edge \left(\frac{\partial^2 g}{\partial \tangential^2_\edge}
    - \varepsilon, \normalProj w_h\right) \ds}\\
    & \quad +\abs*{\varepsilon \sum_{\edge\in\edgesB} \int_\edge \left(\frac{\partial^2 g}{\partial \tangential^2_\edge}
    - \varepsilon, \normalProj w_h\right) \ds}
\end{align*}
Expanding $\AQL$ and $\AQLh$ gives
\begin{align*}
\ALh(\uEpsInterp &- T_h(\uEpsInterp),w_h) \\&\leq \left|\sum_{\element\in\mesh}\varepsilon\int_\element\left(\hessProj \uEpsInterp:\hessProj w_h-\hess\uEps:\hess w_h\right)\dx\right| \\
    & \quad +\left|\sum_{\element\in\mesh}\int_\element \left(\det(\hess\uEps)w_h - \det(\hessProj \uEpsInterp) \valueProj w_h\right) \dx\right| \\
    & \quad + \abs*{\sum_{\element\in\mesh}\int_\element(f-f_h)w_h\dx} + \left|\sum_{\element\in\mesh}\StabQL(\uEpsInterp - \valueProj \uEpsInterp, w_h - \valueProj w_h)\right| \\
    & \quad + \left|\AQL(\uEps,w_h)- \sum_{\element\in\mesh}\int_\element fw_h\dx- \varepsilon \sum_{\edge\in\edgesB} \int_\edge \left(\frac{\partial^2 g}{\partial \tangential^2_\edge} - \varepsilon, \dfrac{\partial w_h}{\partial \normal_\edge}\right) \ds\right|\\
     & \quad + \left|\varepsilon \sum_{\edge\in\edgesB} \int_\edge \left(\frac{\partial^2 g}{\partial \tangential^2_\edge} - \varepsilon, \dfrac{\partial w_h}{\partial \normal_\edge} - \normalProj w_h\right) \ds \right|\\
    &\eqqcolon Q_1+Q_2+Q_3+Q_4+Q_5+Q_6.
\end{align*}
We proceed by bounding these six terms individually. By applying Lemma~\ref{lemma:proj_eq_l2proj}, the definition of the $\Lp$-projection, and Lemma~\ref{lemma:l2proj}, we have that
\begin{align*}
    Q_1 & = \abs*{\sum_{\element\in\mesh} \varepsilon \int_\element \left( \hess\uEps - \LtwoProj{\element}{\polOrder-2}\hess\uEpsInterp\right):\hess w_h\dx} \\
    & \leq \varepsilon\sum_{\element\in\mesh} \|\hess\uEps - \LtwoProj{\element}{l-2}\hess\uEpsInterp\|_{0,\element} \|\hess w_h\|_{0,\element} \\
    &\leq \varepsilon \sum_{\element\in\mesh} \left(\|\hess\uEps - \LtwoProj{\element}{l-2}\hess\uEps \|_{0,\element} + \|\LtwoProj{\element}{l-2}(\hess\uEps -\hess\uEpsInterp)\|_{0,\element} \right)\|\hess w_h\|_{0,\element} \\
    &\leq \varepsilon \sum_{\element\in\mesh} (Ch^{r-1}_\element\|\hess\uEps\|_{r-1,\element} +\|\hess\uEps -\hess\uEpsInterp\|_{0,\element})\|w_h\|_{2,\element}\\
    &\leq \varepsilon \sum_{\element\in\mesh} (Ch^{r-1}_\element\|\uEps\|_{r+1,\element} +\|\uEps -\uEpsInterp\|_{2,\element})\|w_h\|_{2,\element}.
\end{align*} 
Thus, by applying Lemma~\ref{lemma:interpolation}, we follows that 
\begin{equation}\label{eqn:fixed_point:operator_bound:e1}
Q_1 \leq C\varepsilon h^{r-1}\|\uEps\|_{r+1}\|w_h\|_{2,h}.
\end{equation}

For $Q_2$ we apply the triangle inequality to see that
\begin{align*}
Q_2 &\leq \left|\sum_{\element\in\mesh}\int_\element \det(\hess \uEps)w_h\dx\right| + \left|\sum_{\element\in\mesh}\int_\element \det(\hess \uEps)\valueProj w_h\dx\right|  \\
&\quad + \left|\sum_{\element\in\mesh}\int_\element (\det(\hess \uEps)-\det(\hessProj \uEps))\valueProj w_h\dx \right|\\
&\quad + \left|\sum_{\element\in\mesh} \int_\element (\det(\hessProj \uEps) -\det(\hessProj \uEpsInterp))\valueProj w_h\dx\right|.
\end{align*}
Using the fact that  $\cof (\hess \uEps):(\hess  \uEps) \geq \operatorname{det}(\hess  \uEps)$, we can bound the first term by applying
H\"older's inequality, Proposition~\ref{prop:ueps_bounds},  Corollary~\ref{corol:h2full}, Lemma~\ref{lemma:sobolev} and \eqref{eqn:poincare}
\begin{align}
\left|\sum_{\element\in\mesh}\int_\element \det(\hess \uEps)w_h\dx\right|
    &\leq \sum_{\element\in\mesh} \|\cof(\hess \uEps)\|_{0,2,\element}\|\hess \uEps\|_{0,\element}\|w_h\|_{0,\infty,\element}\nonumber\\
    &\leq C \varepsilon^{\nicefrac{-1}{2}}\sum_{\element\in\mesh}h_\element^{r-1}\|\uEps\|_{r+1,\element}h_\element^{-1}\|w_h\|_{0,\element} \nonumber\\
    &\leq C \varepsilon^{\nicefrac{-1}{2}}h^{r-1}\|\uEps\|_{r+1}|w_h|_{2,h}. \label{eqn:fixed_point:operator_bound:e2:t1}
\end{align}
for the second term of $Q_2$ the bound
\[
\left|\sum_{\element\in\mesh}\int_\element \det(\hess \uEps)\valueProj w_h\dx\right| \leq C \varepsilon^{\nicefrac{-1}{2}}h^{r-1}\|\uEps\|_{r+1}|w_h|_{2,h}
\]
follows almost identically with an application of Corollary~\ref{corol:proj_eq_l2proj}. By applying Proposition~\ref{prop:mean_value} we know there exists a $t\in[0,1]$ such that
\begin{multline*}
    \left|\sum_{\element\in\mesh}\int_\element (\det(\hess \uEps)-\det(\hessProj \uEps))\valueProj w_h\dx\right| \\
    \leq \sum_{\element\in\mesh}  \|\cof((1-t)\hess \uEps +t \hessProj \uEps)) \|_{0,\element} \|\hess \uEps - \hessProj \uEps\|_{0,\element} \|\valueProj  w_h\|_{0,\infty,K}.
\end{multline*}
By linearity, the triangle inequality, Proposition~\ref{prop:ueps_bounds}, and standard inverse inequalities on polynomials, we can show that
\begin{multline*}
    \left|\sum_{\element\in\mesh}\int_\element (\det(\hess \uEps)-\det(\hessProj \uEps))\valueProj w_h\dx\right|
    \\\leq \sum_{\element\in\mesh} \|\uEps \|_{2,\element} \|(I- \LtwoProj{\element}{\polOrder-2})\hess \uEps\|_{0,\element} h^{-1}_K\|\valueProj w_h\|_{0,\element}.
\end{multline*}
Finally, by applying Proposition~\ref{prop:ueps_bounds}, Lemma~\ref{lemma:l2proj}, Lemma~\ref{lemma:sobolev}, and \eqref{eqn:poincare} we have that
\[
    \left|\sum_{\element\in\mesh}\int_\element (\det(\hess \uEps)-\det(\hessProj \uEps))\valueProj w_h\dx\right|
    \leq C\varepsilon^{-\nicefrac{1}{2}}h^{r-1}\|\uEps\|_{r+1}|w_h|_{2,h}.
\]
We can bound the last term similarly by applying Lemma~\ref{lemma:interpolation}, instead of Lemma~\ref{lemma:l2proj}, to show that
\begin{align*}
    &\left|\sum_{\element\in\mesh}(\det(\hessProj \uEps)-\det(\hessProj \uEpsInterp),\valueProj w_h)_K\right|\\
    &\qquad\leq \sum_{\element\in\mesh}  \|\cof((1-t)\hessProj \uEps + t \hessProj\uEpsInterp)\|_{0,\element} \|\hessProj (\uEps -\uEpsInterp)\|_{0,\element} \|\valueProj  w_h\|_{0,\infty,\element}  \\
    &\qquad\leq \sum_{\element\in\mesh}  \|\uEps \|_{2,\element} \|\uEps -\uEpsInterp\|_{2,\element} h^{-1}_K\|\valueProj  w_h\|_{0,\element} \\
    &\qquad\leq C\varepsilon^{-\nicefrac{1}{2}}h^{r-1}\|\uEps\|_{r+1}|w_h|_{2,h}.
\end{align*}
Combining the above results we get that
\begin{equation}\label{eqn:fixed_point:operator_bound:e2}
    Q_2 \leq C\varepsilon^{\nicefrac{-1}{2}} h^{r-1}\|\uEps\|_{r+1}|w_h|_{2,h}.
\end{equation}

From Lemma~\ref{lemma:l2proj} and properties of the $\Lp$-projection,
\begin{align}
    Q_3
    = \abs*{\sum_{\element\in\mesh}\int_\element(f-\LtwoProj{\element}{\polOrder}f)w_h\dx}
    &= \abs*{\sum_{\element\in\mesh} \int_\element(f-\LtwoProj{\element}{\polOrder}f)(w_h - \LtwoProj{\element}{1} w_h)\dx}\notag\\
    &\leq \sum_{\element\in\mesh} \|f-\LtwoProj{\element}{\polOrder}f\|_{0,\element}\|w_h - \LtwoProj{\element}{1}w_h\|_{0,\element}\notag\\
    &\leq C\sum_{\element\in\mesh}h^{r-3}_\element\|\varphi\|_{r-3}h^2_K|w_h|_{2,\element}\notag\\
    &\leq Ch^{r-1}\|f\|_{r-3}\seminorm{w_h}_{2,h}. \label{eqn:fixed_point:operator_bound:e3}
\end{align}

We now bound $Q_4$ by \eqref{eqn:stab:linear:equivalence}, Lemma~\ref{lemma:proj_eq_l2proj} and Lemma~\ref{lemma:l2proj},
\begin{align*}
    Q_4 &= \left|\sum_{\element\in\mesh}\StabQL(\uEpsInterp - \valueProj  \uEpsInterp, w_h - \valueProj  w_h)\right| \\
    &\leq \sum_{\element\in\mesh}\abs*{\StabQL(\uEpsInterp - \valueProj  \uEpsInterp, \uEpsInterp - \valueProj  \uEpsInterp)}^{\nicefrac{1}{2}} \abs*{\StabQL(w_h - \valueProj  w_h, w_h - \valueProj  w_h)}^{\nicefrac{1}{2}} \\
    &\leq C\varepsilon \sum_{\element\in\mesh}\|\uEpsInterp - \valueProj  \uEpsInterp\|_{2,\element}\|w_h - \valueProj  w_h\|_{2,\element} \\
    &\leq C\varepsilon \sum_{\element\in\mesh}\|\uEpsInterp - \valueProj  \uEpsInterp\|_{2,\element}|w_h|_{2,\element}.
\end{align*}
Furthermore, by Lemma~\ref{lemma:l2proj}, Lemma~\ref{lemma:interpolation}, and standard inverse estimates for polynomials, for all $\element\in\mesh$ we have that
\begin{align*}
    \|\uEpsInterp - \valueProj  \uEpsInterp\|_{2,\element} & \leq\|\uEpsInterp -\uEps\|_{2,\element} + \|\uEps - \valueProj  \uEps\|_{2,\element} + \|\valueProj (\uEps -  \uEpsInterp)\|_{2,\element} \\
    & \leq\|\uEpsInterp -\uEps\|_{2,\element} + \|\uEps - \valueProj  \uEps\|_{2,\element} + C h_\element^{-2}\|\valueProj (\uEps -  \uEpsInterp)\|_{0,\element} \\
    &\leq C(2h^{r-1}_\element\|\uEps\|_{r+1,\element} + h^{-2}_\element h^{r+1}_\element\|\uEps\|_{r+1,\element}).
\end{align*}
Therefore,
\begin{equation} \label{eqn:fixed_point:operator_bound:e4}
    Q_4 \leq C\varepsilon h^{r-1}\|\uEps\|_{r+1}|w_h|_{2,h}.
\end{equation}

Assuming $\uEps \in H^4(\Omega)$, then by \eqref{eqn:vanishing_moment}--\eqref{eqn:vanishing_moment:extra_bc}, and following similar steps to the proof of Lemma~\ref{lemma:nonconforming:error} we can show that
\begin{align*}
    \int_\element fw_h \dx &= \varepsilon \int_\element \sum_{i,j = 1}^2 \frac{\partial}{\partial x_j}\left(\dfrac{\partial^2 \uEps}{\partial x_i \partial x_j}\right)\dfrac{\partial w_h}{\partial x_i} \dx- \varepsilon \int_{\partial\element} \sum_{i,j = 1}^2 \left(\dfrac{\partial }{\partial x_j}\left(\dfrac{\partial^2 \uEps}{\partial x_i \partial x_j}\right)\right)w_h \eta_i \ds\\
    &\qquad+ \int_\element \det(\hess  \uEps)w_h \dx,
\end{align*}
and hence,
\begin{align*}
    Q_5 = \Bigg|\sum_{\element\in\mesh} &\varepsilon \int_{\partial \element} \sum_{i,j = 1}^2 \left(\dfrac{\partial }{\partial x_j}\left(\dfrac{\partial^2 \uEps}{\partial x_i \partial x_j}\right)\right)w_h\eta_i \dx \\&- \sum_{\element\in\mesh}\varepsilon \int_{\partial \element} \sum_{i,j = 1}^2 \left(\dfrac{\partial^2 \uEps}{\partial x_i \partial x_j}\right)\dfrac{\partial w_h}{\partial x_i}\eta_j \dx - \varepsilon \sum_{\edge \in \edgesB} \int_\edge \left(\frac{\partial^2 g}{\partial \tangential^2_\edge} - \varepsilon, \dfrac{\partial w_h}{\partial \normal_e}\right) \ds\Bigg|
\end{align*}
Following similar steps to the proof of Lemma~\ref{lemma:nonconforming:error} we deduce that
\begin{align*}
    Q_5 &\leq \abs*{\sum_{\edge\in\edgesB} \varepsilon\int_\edge  \left(\dfrac{\partial^2 \uEps}{\partial \tangential_\edge^2} - \laplace \uEps\right) \dfrac{\partial w_h}{\partial \normal_\edge} \ds - \varepsilon \sum_{\edge\in\edgesB} \int_\edge \left(\frac{\partial^2 g}{\partial \tangential^2_\edge} - \varepsilon\right) \dfrac{\partial w_h}{\partial \normal_\edge}\ds} \\
    &\quad+ \abs*{\sum_{\edge\in\edgesB}\varepsilon \int_\edge\dfrac{\partial \laplace \uEps}{\partial \normal_\edge} w_h \ds}
    + \abs*{\sum_{\edge\in\edgesI}\varepsilon \int_\edge\dfrac{\partial \laplace \uEps}{\partial \normal_\edge} \jmp{w_h} \ds}
    +\abs*{\sum_{\edge\in\edgesB}\varepsilon\int_\edge \dfrac{\partial ^2 \uEps}{\partial \normal_\edge\partial \tangential_\edge} \dfrac{\partial w_h}{\partial \tangential_\edge} \ds} \\
    &\quad
    + \abs*{\sum_{\edge\in\edgesI}\varepsilon\int_\edge \dfrac{\partial^2 \uEps}{\partial \normal_\edge\partial \tangential_\edge} \jmp*{\dfrac{\partial w_h}{\partial \tangential_\edge}} \ds}
    +\abs*{\sum_{\edge\in\edgesI} \varepsilon\int_\edge  \left(\dfrac{\partial^2 \uEps}{\partial \tangential_\edge^2} - \laplace \uEps\right) \jmp*{\dfrac{\partial w_h}{\partial \normal_\edge} } \ds}.
\end{align*}
We note that due to \eqref{eqn:vanishing_moment:dbc}--\eqref{eqn:vanishing_moment:extra_bc} the first term is zero.
For the second term, we follow the approach taken in \citet{zhao2018morley} and introduce the interpolation $\Pi_1 w_h \in H^1_0(\Omega)$ of $w_h$ into the lowest order conforming VEM space defined in, e.g., \citet{beirao2013basic}. Then, by the definition of the $\Lp$-projection and Lemma~\ref{lemma:l2proj},
\begin{align*}
    \abs*{\sum_{\edge\in\edgesB}\varepsilon \int_\edge\dfrac{\partial \laplace \uEps}{\partial \normal_\edge} w_h \ds}
    &= \abs*{\sum_{\edge\in\edgesB}\varepsilon\int_\edge\left(\dfrac{\partial \laplace \uEps}{\partial \normal_\edge} - \LtwoProj{\edge}{\polOrder-3}\left(\dfrac{\partial \laplace \uEps}{\partial \normal_\edge} \right)\right) (w_h - \Pi_1w_h) \ds}\\
    & \leq \varepsilon \sum_{\edge\in\edgesB} \norm*{\dfrac{\partial \laplace \uEps}{\partial \normal_\edge} - \LtwoProj{\edge}{\polOrder-3}\left(\dfrac{\partial \laplace \uEps}{\partial \normal_\edge} \right)}_{0,\edge}\norm*{w_h - \Pi_1w_h}_{0,\edge}\\
    &\leq \varepsilon \sum_{\edge\in\edgesB} h_\edge^{r - 2 -\nicefrac{1}{2}}\norm*{\dfrac{\partial \laplace \uEps}{\partial \normal_\edge}}_{r-2,\omega_\edge}h^{2-\nicefrac{1}{2}}_\edge|w_h|_{2,\omega_\edge}\\
    &\leq \varepsilon h^{r-1}\|\uEps\|_{r+1}|w_h|_{2,h}.
\end{align*}
Similarly, by the definition of the $\Lp$-projection and Lemma~\ref{lemma:l2proj},
\begin{align*}
    \abs*{\sum_{\edge\in\edgesB}\varepsilon\int_\edge \dfrac{\partial ^2 \uEps}{\partial \normal_\edge\partial \tangential_\edge} \dfrac{\partial w_h}{\partial \tangential_\edge} \ds} 
    &\leq \sum_{\edge\in\edgesB}\varepsilon \norm*{\left(\dfrac{\partial ^2 \uEps}{\partial \normal_\edge\partial \tangential_\edge} - \LtwoProj{\edge}{\polOrder - 2} \left(\dfrac{\partial ^2 \uEps}{\partial \normal_\edge\partial \tangential_\edge}\right) \right)}_{0,\edge} \norm*{\dfrac{\partial w_h}{\partial \tangential_\edge} - \LtwoProj{\edge}{0} \left(\dfrac{\partial w_h}{\partial \tangential_\edge} \right)}_{0,\edge}\\
    &\leq \sum_{\edge\in\edgesB}\varepsilon h^{r-1-\nicefrac{1}{2}}_\edge\norm*{ \dfrac{\partial ^2 \uEps}{\partial \normal_\edge\partial \tangential_\edge}}_{r-1,\omega_\edge} h^{1-\nicefrac{1}{2}}_\edge\norm*{\dfrac{\partial w_h}{\partial \tangential_\edge}}_{1,\omega_\edge}\\
    &\leq C\varepsilon h^{r-1}\|\uEps\|_{r+1}|w_h|_{2,h}.
\end{align*}
The final term follows almost identically to the proof of Lemma~\ref{lemma:nonconforming:error_bound}, and the remaining two terms follow similar to the above; therefore,
\begin{align}\label{eqn:fixed_point:operator_bound:e5}
    Q_5 \leq C\varepsilon h^{r-1}\|\uEps\|_{r+1}\|w_h\|_{2,h}.
\end{align}

Finally, the proof of the bound
\begin{align}\label{eqn:fixed_point:operator_bound:e6}
    Q_6 \leq C\varepsilon h^{r-1}\|\uEps\|_{r+1}\|w_h\|_{2,h}
\end{align}
for $Q_6$ follows identically to the proof of Lemma~\ref{lemma:strang:boundary}.

Combining \eqref{eqn:fixed_point:operator_bound:e1}, \eqref{eqn:fixed_point:operator_bound:e2}--\eqref{eqn:fixed_point:operator_bound:e6} and applying the coercivity of $\ALh(\cdot,\cdot)$ completes the proof.
\end{proof}

As we need to prove the conditions of the Banach fixed point theorem are met, we first show that within the ball $\ball{\uEpsInterp}{\zeta}$ that $T_h$ is Lipschitz continuous with a constant dependent on $h$ and $\varepsilon$. This will then allow us to prove a contraction of $T_h$.
\begin{lemma}\label{lemma:fixed_point:contraction}
    For any $w_h, v_h \in \ball{\uEpsInterp}{\zeta}$
    \[
        \|T_h(w_h)-T_h(v_h)\|_{2,h} \leq C_7(\varepsilon,h)\|w_h-v_h\|_{2,h}.
    \]
    where $C_7(\varepsilon,h) = C \left((\zeta \varepsilon^{-1} +  2\varepsilon^{-\nicefrac{3}{2}}) h+\varepsilon^{-2}h^2\right)$ is a positive constant dependent on $h$.
\end{lemma}
\begin{proof}
For all $z_h \in \vemSpace{0}$, defining $\gamma_h=w_h-v_h$, we have from \eqref{eqn:T} and the definitions of $\AQLh$ and $\ALh$ that
\begin{align*}
    \ALh(T_h&(w_h)-T_h(v_h), z_h) \\&= \ALh(T_h(w_h)-w_h, z_h) + \ALh(w_h, z_h)-\ALh(v_h, z_h) - \ALh(T_h(v_h)-v_h, z_h)\\
    &=\AQLh(w_h,z_h)+ \ALh(w_h, z_h)-\ALh(v_h, z_h) - \AQLh(v_h, z_h)\\
    &=\sum_{\element\in\mesh}\int_\element\uEpsCof \gradProj\gamma_h\gradProj z_h\dx + \sum_{\element\in\mesh}\int_\element (\det(\hessProj w_h)-\det(\hessProj v_h))\valueProj z_h\dx + \mathcal{S}\\
    &\eqqcolon T_1+T_2+\mathcal{S}.
\end{align*}
where
\begin{align*}
    \mathcal{S} &= \sum_{\element\in\mesh}(\StabQL(w_h - \valueProj w_h,z_h - \valueProj z_h) - \StabQL(v_h - \valueProj v_h,z_h - \valueProj z_h))\\
    &\qquad+ \sum_{\element\in\mesh}(\StabL(w_h - \valueProj w_h,z_h - \valueProj z_h)-\StabL(v_h - \valueProj v_h,z_h - \valueProj z_h))
\end{align*}
From H\"older's inequality, Proposition~\ref{prop:ueps_bounds}, Corollary~\ref{corol:proj_eq_l2proj}, Lemma~\ref{l3.9} and Lemma~\ref{lemma:sobolev} we have that
\begin{align}
    T_1 &\leq \sum_{\element\in\mesh} \|\Phi_\varepsilon\|_{0,\infty,\element}\|\gradProj \gamma_h\|_{0,2,\element}\|\gradProj z_h\|_{0,2,\element} \notag\\
    & \leq C \sum_{\element\in\mesh} \varepsilon^{-1}\norm{\gamma_h}_{1,\element} \norm{z_h}_{1,\element}\notag\\
    & \leq C \sum_{\element\in\mesh} \varepsilon^{-1} h_\element \norm{\gamma_h}_{2,\element} h_\element \norm{z_h}_{2,\element} \notag\\
    & \leq C\varepsilon^{-1}h^2\|\gamma_h\|_{2,h}\|z_h\|_{2,h}. \label{eqn:fixed_point:contraction:t1}
\end{align}
From Proposition~\ref{prop:mean_value}, there exists a $t\in[0,1]$ such that
\[
    T_2 \leq \sum_{\element\in\mesh} \|\cof((1-t)(\hessProj v_h)+t\hessProj w_h)\|_{0,\element}\|\hessProj \gamma_h\|_{0,\element}\|\valueProj z_h\|_{0,\infty,\element}.
\]
Applying Corollary~\ref{corol:proj_eq_l2proj}, standard inverse inequalities for polynomials, Lemma~\ref{lemma:sobolev}, Lemma~\ref{lemma:interpolation}, Proposition~\ref{prop:ueps_bounds}, and the fact that $v_h,w_h\in\ball{\uEpsInterp}{\zeta}$, we see that
\begin{align}
    T_2 &\leq  C\sum_{\element\in\mesh} (t\| w_h\|_{2,\element}+(1-t)\|v_h\|_{2,\element})\|\gamma_h\|_{2,\element}h^{-1}_\element\|z_h\|_{0,\element}\notag\\
    &\leq  C\sum_{\element\in\mesh} (t\| w_h-\uEpsInterp\|_{2,\element}+(1-t)\|v_h-\uEpsInterp\|_{2,\element} \notag\\
    &\qquad\qquad\qquad+ \| \uEpsInterp - \uEps\|_{2,\element} + \|\uEps\|_{2,\element})\|\gamma_h\|_{2,\element} h^{-1}_\element h_\element^2\|z_h\|_{2,\element}\notag\\
    &\leq C\sum_{\element\in\mesh} (\zeta + 2\|\uEps\|_{2,\element})\|\gamma_h\|_{2,\element}  h_\element\|z_h\|_{2,\element}\notag\\
    &\leq  C (\zeta +  2\varepsilon^{-\nicefrac{1}{2}}) h \|\gamma_h\|_{2,h}\|z_h\|_{2,h}. \label{eqn:fixed_point:contraction:t2}
\end{align}
From \eqref{eqn:stab:ql:a}--\eqref{eqn:stab:ql} we have that
\begin{align*}
    \mathcal{S} &= \sum_{\element\in\mesh}\left(\|\cof(\hessProj[\element,0] w_h)\|_{0,\infty,\element}+\|\Phi_\varepsilon\|_{0,\infty,\element}\right)\sum_{\lambda\in\lambda^\element}  {\lambda(w_h - \valueProj w_h)\lambda(z_h - \valueProj z_h)} \notag \\
    &\quad - \sum_{\element\in\mesh}\left(\|\cof(\hessProj[\element,0] v_h)\|_{0,\infty,\element}+\|\Phi_\varepsilon\|_{0,\infty,\element}\right)\sum_{\lambda\in\Lambda^\element}  {\lambda(v_h - \valueProj v_h)\lambda(z_h - \valueProj z_h)} \notag \\
    &\leq  \sum_{\element\in\mesh}\abs*{\|\cof(\hessProj[\element,0] w_h)\|_{0,\infty,\element} - \|\cof(\hessProj[\element,0] v_h)\|_{0,\infty,\element}}\sum_{\lambda\in\lambda^\element} \abs*{{\lambda(w_h - \valueProj w_h)\lambda(z_h - \valueProj z_h)}} \notag \\
    &\quad +\sum_{\element\in\mesh}\left( \|\cof(\hessProj[\element,0] v_h)\|_{0,\infty,\element}+\|\Phi_\varepsilon\|_{0,\infty,\element}\right) \sum_{\lambda\in\lambda^\element} \abs*{\lambda(\gamma_h - \valueProj \gamma_h)\lambda(z_h - \valueProj z_h)} \notag \\
    &\leq \sum_{\element\in\mesh}\|\cof(\hessProj[\element,0] \gamma_h)\|_{0,\infty,\element}\sum_{\lambda\in\lambda^\element}  \abs*{\lambda(w_h - \valueProj w_h)\lambda(z_h - \valueProj z_h)} \notag \\
    &\quad +\sum_{\element\in\mesh}\left( \|\cof(\hessProj[\element,0] v_h)\|_{0,\infty,\element}+\|\Phi_\varepsilon\|_{0,\infty,\element}\right)  \sum_{\lambda\in\lambda^\element}\abs*{\lambda(\gamma_h - \valueProj \gamma_h)\lambda(z_h - \valueProj z_h)}.
\end{align*}
By standard polynomial inverse inequalities and Corollary~\ref{corol:proj_eq_l2proj}
\[
    \|\cof(\hessProj[\element,0] \gamma_h)\|_{0,\infty,\element} = \|\hessProj[\element,0] \gamma_h\|_{0,\infty,\element}
    \leq Ch^{-1}_\element\|\hessProj[\element,0] \gamma_h\|_{0,2,\element}
    \leq Ch^{-1}_\element\|\gamma_h\|_{2,\element},
\]
and similarly for $\|\cof(\hessProj[\element,0] v_h)\|_{0,\infty,\element}$; therefore, by applying this result and Proposition~\ref{prop:ueps_bounds}, we get that
\begin{align}
\mathcal{S}
    &\leq C\sum_{\element\in\mesh}h_\element^{-1}\|\gamma_h\|_{2,\element}\sum_{\lambda\in\lambda^\element}  \abs*{\lambda(w_h - \valueProj w_h)\lambda(z_h - \valueProj z_h)} \notag \\
    &\quad +\sum_{\element\in\mesh}\left( Ch_\element^{-1}\|v_h\|_{2,\element}+\varepsilon^{-1}\right)  \sum_{\lambda\in\lambda^\element}\abs*{\lambda(\gamma_h - \valueProj \gamma_h)\lambda(z_h - \valueProj z_h)} 
    \label{eqn:fixed_point:contraction:stab}
\end{align}

We now need to bound the degrees of freedom. We consider separately, for $\gamma_h$, the four types of degrees of freedom, \ref{dofs:value}--\ref{dofs:volume}, which we denote as $\lambda_1,\dots,\lambda_4$, respectively. We note that the bounds for the terms containing $w_h$ follow identically.

For the interior volume degrees of freedom \ref{dofs:volume} there exists a function $q\in \poly(\element)$ which satisfies Assumption~\ref{assumption:monomial_linfty} and, therefore,
\begin{align*}
\abs*{\lambda_4(\gamma_h - \valueProj \gamma_h)\lambda_4(z_h - \valueProj z_h)}
    &= \abs*{ \dfrac{1}{h_\element^4}\left(\int_\element (I - \LtwoProj{\element}{\polOrder})\gamma_h q \dx \right)\left(\int_\element (I - \LtwoProj{\element}{\polOrder})z_h q \dx \right)} \\
    &\leq \dfrac{1}{h_\element^4}\|(I - \LtwoProj{\element}{\polOrder})\gamma_h\|_{0,1,\element}\|q\|_{0,\infty,\element}\|(I - \LtwoProj{\element}{\polOrder})z_h\|_{0,1,\element}\|q\|_{0,\infty,\element} \\
    &\leq \dfrac{1}{h_\element^4}h_\element^2\|(I - \LtwoProj{\element}{\polOrder})\gamma_h\|_{0,\element}\|(I - \LtwoProj{\element}{\polOrder})z_h\|_{0,\element}.
\end{align*}
By applying Lemma~\ref{lemma:l2proj} we have that
\begin{equation}\label{eqn:fixed_point:contraction:stab:d4}
    \abs*{\lambda_4(\gamma_h - \valueProj \gamma_h)\lambda_4(z_h - \valueProj z_h)} \leq C h_\element^2|\gamma_h|_{2,\element}|z_h|_{2,\element}.
\end{equation}

For the vertex degrees of freedom \ref{dofs:value} we note that the value at the vertex will be bounded by the $\Lp[\infty]$-norm of the function on the boundary of the element; i.e.,
\[
\abs*{\lambda_1(\gamma_h - \valueProj \gamma_h)\lambda_1(z_h - \valueProj z_h)} 
\leq \|\gamma_h- \valueProj \gamma_h\|_{0,\infty,\partial \element}\|z_h - \valueProj z_h\|_{0,\infty,\partial \element}.
\]
Applying Lemmas~\ref{lemma:linfty_bdry}, \ref{lemma:h1_bdry}, and~\ref{lemma:l2proj},
\begin{align}
&\abs*{\lambda_1(\gamma_h - \valueProj \gamma_h)\lambda_1(z_h - \valueProj z_h)} \notag\\
    &\quad\leq C(h^{-1}_\element\|(I - \LtwoProj{\element}{\polOrder})\gamma_h\|^2_{0,\partial \element} + |(I - \LtwoProj{\element}{\polOrder})\gamma_h|^2_{\nicefrac12,\partial \element})^{\nicefrac{1}{2}}(h^{-1}_\element\|(I - \LtwoProj{\element}{\polOrder})z_h\|^2_{0,\partial \element} + |(I - \LtwoProj{\element}{\polOrder})z_h|^2_{\nicefrac12,\partial \element})^{\nicefrac{1}{2}}\notag\\
    &\quad\leq C(h^{-2}_\element\|(I - \LtwoProj{\element}{\polOrder})\gamma_h\|^2_{0,\element} + 2|(I - \LtwoProj{\element}{\polOrder})\gamma_h|^2_{1,\element})^{\nicefrac{1}{2}}(h^{-2}_\element\|(I - \LtwoProj{\element}{\polOrder})z_h\|^2_{0,\element} + 2|(I - \LtwoProj{\element}{\polOrder})z_h|^2_{1,\element})^{\nicefrac{1}{2}}\notag\\
    &\quad\leq C(h^{2}_\element|\gamma_h|^2_{2,\element} + 2h^{2}_\element|\gamma_h|^2_{2,\element})^{\nicefrac{1}{2}}(h^{2}_\element|z_h|^2_{2,\element} + 2h^{2}_\element|z_h|^2_{2,\element})^{\nicefrac{1}{2}}\notag\\
    &\quad= Ch^{2}_\element|\gamma_h|_{2,\element}|z_h|_{2,\element}.\label{eqn:fixed_point:contraction:stab:d1}
\end{align}

For the edge moment degrees of freedom \ref{dofs:edge}, on the edge $\edge\subset \partial \element$, there exists a function $q\in \poly[\polOrder-2](\edge)$ which satisfies Assumption~\ref{assumption:monomial_linfty} and, therefore,
\begin{align*}
\abs*{\lambda_2(\gamma_h - \valueProj \gamma_h)\lambda_2(z_h - \valueProj z_h)}
    &= \abs*{ \dfrac{1}{h_\edge^2}\left(\int_\edge (I - \LtwoProj{\element}{\polOrder})\gamma_h q \ds\right) \left(\int_\edge (I - \LtwoProj{\element}{\polOrder})z_h q \ds \right)} \\
    &\leq \dfrac{1}{h_\edge^2}\|(I - \LtwoProj{\element}{\polOrder})\gamma_h\|_{0,1,\edge}\|q\|_{0,\infty,\edge}\|(I - \LtwoProj{\element}{\polOrder})z_h\|_{0,1,\edge}\|q\|_{0,\infty,\edge} \\
    &\leq \dfrac{1}{h_\edge^2}h_\edge\|(I - \LtwoProj{\element}{\polOrder})\gamma_h\|_{0,\edge}\|(I - \LtwoProj{\element}{\polOrder})z_h\|_{0,\edge}.
\end{align*}
By applying Lemma~\ref{lemma:l2proj}
\begin{equation}\label{eqn:fixed_point:contraction:stab:d2}
\abs*{\lambda_2(\gamma_h - \valueProj \gamma_h)\lambda_2(z_h - \valueProj z_h)}
    \leq C\dfrac{1}{h_\edge}h_\edge^{2-\nicefrac{1}{2}}|\gamma_h|_{2,\omega_\edge}h_\edge^{2-\nicefrac{1}{2}}|z_h|_{2,\omega_\edge}
    = Ch_\edge^2|\gamma_h|_{2,\omega_\edge}|z_h|_{2,\omega_\edge}.
\end{equation}
Similarly, for the edge normal moments \ref{dofs:normals} we can show that
\begin{equation}\label{eqn:fixed_point:contraction:stab:d3}
\abs*{\lambda_3(\gamma_h - \valueProj \gamma_h)\lambda_3(z_h - \valueProj z_h)}
    \leq Ch_\edge h_\edge^{1-\nicefrac{1}{2}}\seminorm*{\frac{\partial \gamma_h}{\partial\normal_\edge}}_{1,\omega_\edge}h_\edge^{1-\nicefrac{1}{2}}\seminorm*{\frac{\partial z_h}{\partial\normal_\edge}}_{1,\omega_\edge}
    \leq Ch_\edge^2|\gamma_h|_{2,\omega_\edge}|z_h|_{2,\omega_\edge}.
\end{equation}

Applying \eqref{eqn:fixed_point:contraction:stab:d4}--\eqref{eqn:fixed_point:contraction:stab:d3} to \eqref{eqn:fixed_point:contraction:stab} gives
\begin{align*}
\mathcal{S}
    &\leq C\sum_{\element\in\mesh} \left(\left(\|v_h\|_{2,\element}+\|w_h\|_{2,\element}\right)h_\element + \varepsilon^{-1}h_\element^2\right) \|\gamma_h\|_{2,\element}\|z_h\|_{2,\element}\\
    &\leq C\left((\zeta +  2\varepsilon^{-\nicefrac{1}{2}})h + \varepsilon^{-1}h^2\right) \seminorm{\gamma_h}_{2,h}\seminorm{z_h}_{2,h}
\end{align*}
where the inequality follows similarly to \eqref{eqn:fixed_point:contraction:t2} as $v_h,w_h\in\ball{\uEpsInterp}{\zeta}$. 
Combining this result with \eqref{eqn:fixed_point:contraction:t1} and \eqref{eqn:fixed_point:contraction:t2}, and applying the coercivity of $\ALh(\cdot,\cdot)$ completes the proof.
\end{proof}

We can now apply the fixed point theorem to prove the existence of a unique fixed point conditionally on the mesh size $h$.
\begin{lemma}\label{lemma:fixed_point:unique}
There exist constants $h_1>0$ and $\zeta_0>0$ such that for all $h\leq h_1$ the operator $T_h$ has a unique fixed point in $\ball{\uEpsInterp}{\zeta_0}$.
\end{lemma}
\begin{proof}
    Selecting $\zeta_0 = 2C_6(\varepsilon)h^{r-1}\|\uEps\|_{r+1}+2\hat{C}h^{r-1}$, where $\hat{C} = C\|f\|_{r-3} $ and $C > 0$ is a constant, and
    \[
        h_1 := \min \left\{ \left(\dfrac{\varepsilon^{\nicefrac{3}{2}}}{16C}\right), \left(\dfrac{\varepsilon}{16CC_6(\varepsilon)\|\uEps\|_{r+1}}\right)^{\nicefrac{1}{r}}, \left(\dfrac{\varepsilon}{16C\hat{C}}\right)^{\nicefrac{1}{r}}, \left(\dfrac{\varepsilon^2}{8C}\right)^{\nicefrac{1}{2}} \right\}.
    \]
    Then, for any $h \leq h_1$ and $v_h \in \ball{\uEpsInterp}{\zeta_0}$, using Lemma~\ref{lemma:fixed_point:contraction} we see that
    \begin{align*}
        \|T_h(\uEpsInterp)-T_h(v_h)\|_{2,h} &\leq (Ch\varepsilon^{-1}\zeta_0 + 2Ch\varepsilon^{-\nicefrac{3}{2}}+Ch^2\varepsilon^{-2})\|\uEpsInterp-v_h\|_{2,h}\\
        & \leq (2CC_6(\varepsilon)\varepsilon^{-1}h^{r}\|\uEps\|_{r+1}+2C\hat{C}\varepsilon^{-1}h^{r} + 2C\varepsilon^{-\nicefrac{3}{2}}h+Ch^2\varepsilon^{-2})\|\uEpsInterp-v_h\|_{2,h} \\
        & \leq \Bigg[2CC_6(\varepsilon)\varepsilon^{-1}\|\uEps\|_{r+1}\left(\dfrac{\varepsilon}{16CC_6(\varepsilon)\|\uEps\|_{r+1}}\right) + 2C\hat{C}\varepsilon^{-1} \left(\dfrac{\varepsilon}{16C\hat{C}}\right)\\
        & \qquad +2C\varepsilon^{-\nicefrac{3}{2}}\left(\dfrac{\varepsilon^{\nicefrac{3}{2}}}{16C}\right)+C\varepsilon^{-2}\left(\dfrac{\varepsilon^2}{8C}\right)\Bigg]\|\uEpsInterp-v_h\|_{2,h} \\
        &\leq \frac{1}{2}\|\uEpsInterp-v_h\|_{2,h}.
    \end{align*}
    Furthermore $v_h\in\ball{\uEpsInterp}{\zeta_0}$ we can show that $T(v_h)\in\ball{\uEpsInterp}{\zeta_0}$ by applying the triangle inequality, the above result, and Lemma~\ref{lemma:fixed_point:operator_bound}; i.e.,
    \begin{align*}
        \|\uEpsInterp - T_h(v_h)\|_{2,h} &\leq \|\uEpsInterp - T_h(\uEpsInterp)\|_{2,h} + \|T_h(\uEpsInterp) - T_h(v_h)\|_{2,h} \\
        &\leq Ch^{r-1}(C_6(\varepsilon)\norm{\uEps}_{r+1}+\norm{f}_{r-3})+\frac{1}{2}\|\uEpsInterp-v_h\|_{2,h}\\
        &\leq \frac{\zeta_0}{2} + \frac{\zeta_0}{2} = \zeta_0.
    \end{align*}
    Hence, $T_h(\ball{\uEpsInterp}{\zeta_0})\subset \ball{\uEpsInterp}{\zeta_0}$, and we can apply the Banach fixed point theorem to complete the proof.
\end{proof}

\subsection{A priori error estimates}
We are now able to show the existence of a unique solution $\uh\in\vemSpace{g}$ to the virtual element formulation \eqref{eqn:vem}, and derive a priori error bounds for this solution in the $\Hk{2}$-, $\Hk{1}$-, and $\Lp$-norms, where the $\Hk{2}$-error is derived as a consequence of the fixed point iteration.

\begin{theorem}[$\Hk{2}$-error bound]\label{theorem:nonlinear:h2}
Suppose that Assumption~\ref{assumption:mesh} is satisfied. Let $\polOrder \geq 2$ be a positive integer and $\uEps \in \Hk{s+1}(\Omega)$, $s\geq 3$, be the solution of \eqref{eqn:vanishing_moment}--\eqref{eqn:vanishing_moment:extra_bc}. Define $r = \min(\polOrder, s)$ and assume that $f \in \Hk{r-3}(\Omega)$. Then, for a small enough $h$ there exists a unique solution satisfies $\uh \in \vemSpace{g}$ to the virtual element method \eqref{eqn:vem} and
\[
    \|\uEps - \uh\|_{2,h} \leq C_8(\varepsilon) h^{r-1} (\|\uEps\|_{r+1} +\|f\|_{r-3}),
\]
where $C_8(\varepsilon) \coloneqq C\max\{C_6(\varepsilon),1\} = \mathcal{O}(\varepsilon^{\nicefrac{-3}{2}})$ is a positive constant, independent of $h$.
\end{theorem}
\begin{proof}
    From \eqref{eqn:T} we can see that a fixed point of the operator $T_h$ is a solution of \eqref{eqn:vem}, and by Lemma~\ref{lemma:fixed_point:unique} this fixed point is unique for sufficiently small $h$; therefore, this fixed point is the unique solution of \eqref{eqn:vem}. Additionally, from Lemma~\ref{lemma:fixed_point:unique} we know that the solution $\uh\in\ball{\uEpsInterp}{\zeta_{0}}$ and by the triangle inequality we have that
    \begin{align*}
        \|\uEps - \uh\|_{2,h} &\leq \|\uEps - \uEpsInterp\|_{2,h} + \|\uEpsInterp - \uh\|_{2,h}\\
        &\leq Ch^{r-1} \|\uEps\|_{r+1} + \zeta_{0} \\
        &\leq Ch^{r-1} \|\uEps\|_{r+1} + 2C_6(\varepsilon)h^{r-1}\|\uEps\|_{r+1}+2\hat{C}h^{r-1}\|f\|_{r-3} \\
        &\leq C_8(\varepsilon)h^{r-1}(\|\uEps\|_{r+1} + \|f\|_{r-3}).
    \end{align*}
\end{proof}

\begin{theorem}[$\Hk{1}$-error bound]\label{theorem:nonlinear:h1}
Suppose that Assumption~\ref{assumption:mesh} is satisfied. Let $\polOrder \geq 2$ be a positive integer and $\uEps \in \Hk{s+1}(\Omega)$, $s\geq 3$, be the solution of \eqref{eqn:vanishing_moment}--\eqref{eqn:vanishing_moment:extra_bc}. Define $r = \min(\polOrder, s)$ and assume that $f \in \Hk{r-2}(\Omega)$. Let $\uh \in \vemSpace{g}$ be the corresponding virtual element solution to \eqref{eqn:vem}; then,
for small enough $h$,
\[
   \norm{\uEps - \uh}_{1,\Omega} \leq C_{9}(\varepsilon)h^{r}(\|\uEps\|_{r+1}+ \norm{f}_{r-2}) + C_{10}
    (\varepsilon)h^{2r-2}(\|\uEps\|_{r+1}+ \norm{f}_{r-2})^2.
\]
where $C_{9}(\varepsilon) = CC_8(\varepsilon)\varepsilon^{-3} = \mathcal{O}(\varepsilon^{\nicefrac{-9}{2}})$ and $C_{10}(\varepsilon) = C\varepsilon^{-2}C_8^2(\varepsilon) = \mathcal{O}(\varepsilon^{-5})$ are positive constants, independent of $h$.
\end{theorem}
\begin{proof}
See Section~\ref{section:nonlinear:h1}.
\end{proof}

\begin{theorem}[$\Lp$-error bound]\label{theorem:nonlinear:l2}
Suppose that Assumption~\ref{assumption:mesh} is satisfied. Let $\polOrder \geq 2$ be a positive integer and $\uEps \in \Hk{s+1}(\Omega)$, $s\geq 3$, be the solution of \eqref{eqn:vanishing_moment}--\eqref{eqn:vanishing_moment:extra_bc}. Define $r = \min(\polOrder, s)$ and assume that $f \in \Hk{r-1}(\Omega)$. Let $\uh \in \vemSpace{g}$ be the corresponding virtual element solution to \eqref{eqn:vem}; then,
for small enough $h$,
\[
    \|\uEps - \uh\|_{0,\Omega} \leq 
    \begin{cases}
        C_{11}(\varepsilon)h^{2}\|\uEps\|_{3} + C_{12}(\varepsilon)h^{5}\|\uEps\|^2_{3}, & \polOrder =2\\
        C_{11}(\varepsilon)h^{r+1}\|\uEps\|_{r+1}+ C_{12}(\varepsilon)h^{2r+1}\|\uEps\|^2_{r+1},& \polOrder \geq 3. 
    \end{cases} 
\]
where $C_{11}(\varepsilon) = CC_8(\varepsilon)\varepsilon^{-4} = \mathcal{O}(\varepsilon^{\nicefrac{-11}{2}})$, and $ C_{12}(\varepsilon) = C\varepsilon^{-3}C_8^2(\varepsilon)= \mathcal{O}(\varepsilon^{-6})$ are positive constants, independent of $h$.
\end{theorem}
\begin{proof}
The proof follows similarly to Theorem~\ref{theorem:nonlinear:h1} using the duality argument: For any $z \in \Lp(\Omega)$, let $\Psi \in \Hk{4}(\Omega) \cap \Hk{1}_{0}(\Omega)$ be the solution to the linear problem
\begin{alignat*}{2}
L_{\uEps}(\Psi) &= z && \qquad\text{in } \Omega, \\
\Psi &= 0 && \qquad\text{on } \partial \Omega, \\
\Delta \Psi &= 0 && \qquad\text{on } \partial \Omega.
\end{alignat*}
The proof then follows as in Section~\ref{section:nonlinear:h1}, using $\Lp$-projections into polynomial spaces one degree higher compared to the $\Hk{1}$-error bound for $\polOrder \geq 3$, and the same degree for $\polOrder=2$. For details, see \citet{pradhanPhD}.
\end{proof}
\begin{remark}{\label{remark:suboptimal}}
    We get suboptimal $\Lp$ error bounds for $\polOrder = 2$ as the duality argument used in the proof requires the dual solution satisfies $\Psi\in\Hk{4}(\Omega)$ which in turn requires $\polOrder + 1 \geq 4$; therefore, we are only able to derive optimal bounds for $\polOrder \geq 3$. This suboptimality for $\polOrder=2$ is also seen in the numerical experiments in Section~\ref{sec:numerics}.
\end{remark}

\subsection{Proof of Theorem~\ref{theorem:nonlinear:h1}: \texorpdfstring{$\Hk{1}$}{H1}-error estimates}\label{section:nonlinear:h1}
The proof of Theorem~\ref{theorem:nonlinear:h1} is based on the classical duality argument, following, for example, \citet{lascaux1975some}. Let  $\eh=\uEps-\uh\in\Hk{1}_{0}(\Omega)$, and we can define the dual problem: For any $ z \in \Hk{-1}(\Omega)$, find $\Psi \in \Hk{3}(\Omega) \cap \Hk{1}_0(\Omega)$ such that
\begin{alignat}{2}
L_{\uEps}(\Psi) &= z, && \qquad\text{in } \Omega, \label{eqn:dual} \\
\Psi &= 0, && \qquad\text{on } \partial \Omega, \label{eqn:dual:dbc} \\
\laplace \Psi &= 0, && \qquad\text{on } \partial \Omega, \label{eqn:dual:extra_bc}
\end{alignat}
with the linear operator $L_{\uEps}$ defined in \eqref{eqn:linear_op}.
We assume the domain $\Omega$ is convex and, therefore, from \citet{neilan2009numerical},
\begin{equation}\label{eqn:dual:bound}
        \|\Psi\|_{3} \leq C\varepsilon^{-2}\|z\|_{-1}
\end{equation}
and from Lemma~\ref{lemma:interpolation} we have that there exists a $\Psi_I\in\vemSpace{0}$ such that
\begin{equation}\label{eqn:dual:interpolation}
    \|\Psi-\Psi_I\|_{2,h} \leq Ch\|\Psi\|_3 \leq C\varepsilon^{-2}h\|z\|_{-1}.
\end{equation}
where $\norm{\cdot}_{-1}$ denotes the dual norm for $\Hk{-1}(\Omega)$.

Because $\eh\in \Hk{1}_0(\Omega)$, we can write:
\begin{equation}\label{eqn:dualh1}
    \norm{\eh}_{1,\Omega} = \sup_{z \in \Hk{-1}(\Omega)} \frac{|(\eh, z)|}{\|z\|_{-1}}.
\end{equation}
By multiplying \eqref{eqn:dual} with $\eh$, applying element-wise integration by parts and \eqref{eqn:hess_eq_biharmonic}, we have that
\begin{align*}
    (\eh, z) &= \varepsilon\sum_{\element\in\mesh}\int_\element \biharmonic \Psi \eh\dx + \sum_{\element\in\mesh}\int_\element\diver(\uEpsCof \grad\Psi)\eh\dx\\
    &=\varepsilon\sum_{\element\in\mesh}\int_\element\laplace \Psi \laplace \eh\dx+ \sum_{\element\in\mesh}\int_\element \uEpsCof \grad \Psi \cdot \grad \eh\dx - \varepsilon \sum_{\element\in\mesh}\int_{\partial \element\setminus\partial\Omega} \laplace \Psi \frac{\partial \eh}{\partial \normal_\element}\ds \\
    &= \AL(\Psi,\eh)-\varepsilon \sum_{\edge \in \edgesI} \int_\edge \left(\frac{\partial^2 \Psi}{\partial \normal_\edge \partial \tangential_\edge} \jmp*{\frac{\partial \eh}{\partial \tangential_\edge}}
+ \left(\Delta \Psi - \frac{\partial^2 \Psi}{\partial \tangential_\edge^2}\right) \jmp*{\frac{\partial \eh}{\partial \normal_\edge}}\right) \ds.
\end{align*}
Adding \eqref{eqn:vem} with $v_h = \Psi_I$, and applying the definitions of $\AL$ and $\AQLh$, we see that
\begin{align*}
    (\eh, z) &=  \AL(\Psi,\eh) - \aQLh(\uh,\Psi_I)-\bQLh(\uh,\Psi_I)\\&\quad+ (f_h,\Psi_I)+ \varepsilon \sum_{\edge\in\edgesB} \int_e \left( \frac{\partial^2 g}{\partial \tangential_\edge^2} -\varepsilon \right) \normalProj \Psi_I \ds\\
&\quad-\varepsilon \sum_{\edge \in \edgesI} \int_\edge \left(\frac{\partial^2 \Psi}{\partial \normal_\edge \partial \tangential_\edge} \jmp*{\frac{\partial \eh}{\partial \tangential_\edge}} + \left(\Delta \Psi - \frac{\partial^2 \Psi}{\partial \tangential_\edge^2}\right) \jmp*{\frac{\partial \eh}{\partial \normal_\edge}}\right) \ds  \\
 &= \left(\aQL(\uh,\Psi_I) - \aQLh(\uh,\Psi_I)\right)-\varepsilon \sum_{\edge\in\edgesB} \int_e \left( \frac{\partial^2 g}{\partial \tangential_\edge^2} -\varepsilon \right) \left(\normalProj \Psi_I - \frac{\partial\Psi_I}{\partial\normal_\edge} \right)\ds \\
&\quad+\varepsilon \sum_{\edge \in \edgesI} \int_\edge \left(\frac{\partial^2 \Psi}{\partial \normal_\edge \partial \tangential_\edge} \jmp*{\frac{\partial \eh}{\partial \tangential_\edge}} + \left(\Delta \Psi - \frac{\partial^2 \Psi}{\partial \tangential_\edge^2}\right) \jmp*{\frac{\partial \eh}{\partial \normal_\edge}}\right) \ds +  (f_h-f,\Psi_I) \\
&\quad+\left(\sum_{\edge\in\edgesB} \int_e \left( \frac{\partial^2 g}{\partial \tangential_\edge^2} -\varepsilon \right) \frac{\partial\Psi_I}{\partial\normal_\edge} \ds + (f,\Psi_I) - \AQL(\uEps,\Psi_I) \right) \\
&\quad+ \left(\AL(\Psi,\eh) -  \aQL(\uh,\Psi_I) + \AQL(\uEps,\Psi_I) - \bQLh(\uh,\Psi_I)\right) \\
    &\eqqcolon Y_1 + Y_2 + Y_3 + Y_4 + Y_5 + Y_6.
\end{align*}
We proceed to bound these six terms individually. For the first term $Y_1$, we follow \citet{cangiani2017conforming}. We first note that
\begin{align*}
    Y_1 &= \sum_{\element \in \mesh}\left(\aQL^\element(\uh,\Psi_I) - \aQLh^\element(\uh,\Psi_I)\right)\\
    &= \sum_{\element \in \mesh}\left(\aQL^\element(\uh - \LtwoProj{\element}\polOrder \uEps,\Psi_I - \LtwoProj{\element}2 \Psi) - \aQLh^\element(\uh- \LtwoProj{\element}\polOrder \uEps,\Psi_I- \LtwoProj{\element}2 \Psi)\right)\\
    &\quad+\sum_{\element \in \mesh}\left(\aQL^\element( \LtwoProj{\element}\polOrder \uEps,\Psi_I)-\aQLh^\element( \LtwoProj{\element}\polOrder \uEps,\Psi_I)\right)\\
    &\quad+\sum_{\element \in \mesh}\left(\aQL^\element( \uh,\LtwoProj{\element}2 \Psi)-\aQLh^\element( \uh,\LtwoProj{\element}2 \Psi)\right)\\
    &= \sum_{\element \in \mesh}\left(\aQL^\element(\uh - \LtwoProj{\element}\polOrder \uEps,\Psi_I - \LtwoProj{\element}2 \Psi) - \aQLh^\element(\uh- \LtwoProj{\element}\polOrder \uEps,\Psi_I- \LtwoProj{\element}2 \Psi)\right)\\
    &\quad+\sum_{\element \in \mesh}\varepsilon \left( (\hessProj \LtwoProj{\element}\polOrder\uEps, \hessProj\Psi_I)-( \hess \LtwoProj{\element}\polOrder \uEps,\hess\Psi_I)\right)\\
    &\quad+\sum_{\element \in \mesh}\varepsilon\left((\hessProj\uh,\hessProj\LtwoProj{\element}2 \Psi)-(\hess\uh,\hess\LtwoProj{\element}2 \Psi)\right)
\end{align*}
By applying Lemma~\ref{lemma:proj_eq_l2proj} and the definition of the $\Lp$-projection we get that
\begin{align*}
    Y_1 &= \sum_{\element \in \mesh}\left(\aQL^\element(\uh - \LtwoProj{\element}\polOrder \uEps,\Psi_I - \LtwoProj{\element}2 \Psi) - \aQLh^\element(\uh- \LtwoProj{\element}\polOrder \uEps,\Psi_I- \LtwoProj{\element}2 \Psi)\right)\\
    &\quad+\sum_{\element \in \mesh}\varepsilon (\hess \LtwoProj{\element}\polOrder\uEps, (\LtwoProj{\element}{\polOrder-2} - I)\hess\Psi_I)
    +\sum_{\element \in \mesh}\varepsilon((\LtwoProj{\element}{\polOrder-2} - I)\hess\uh,\hess\LtwoProj{\element}2 \Psi)\\
    &= \sum_{\element \in \mesh}\left(\aQL^\element(\uh - \LtwoProj{\element}\polOrder \uEps,\Psi_I - \LtwoProj{\element}2 \Psi) - \aQLh^\element(\uh- \LtwoProj{\element}\polOrder \uEps,\Psi_I- \LtwoProj{\element}2 \Psi)\right)\\
    &\quad+\sum_{\element \in \mesh}\varepsilon\left(((\LtwoProj{\element}{\polOrder-2}-I)\hess \LtwoProj{\element}\polOrder\uEps, \hess(\Psi_I-\Psi)) + ((\LtwoProj{\element}{\polOrder-2} - I)\hess \LtwoProj{\element}\polOrder\uEps, \hess(\Psi_I-\LtwoProj{\element}2\Psi))\right)\\
    &\quad+\sum_{\element \in \mesh}\varepsilon\left( ((\LtwoProj{\element}{\polOrder-2} - I)\hess(\uh-\uEps),\hess(\LtwoProj{\element}2-I)\Psi) + ((\LtwoProj{\element}{\polOrder-2} - I)\hess\uEps,\hess(\LtwoProj{\element}2-I) \Psi) \right. \\
    &\qquad\qquad\qquad+\left.(\hess(\uh-\LtwoProj{\element}\polOrder\uEps),(\LtwoProj{\element}{\polOrder-2}-I)\hess\Psi)
    + (\hess\LtwoProj{\element}\polOrder\uEps,(\LtwoProj{\element}{\polOrder-2}-I)\hess\Psi)\right)
\end{align*}
Using the fact that $\aQL^\element(\cdot,\cdot)$ and $\aQLh^\element(\cdot,\cdot)$ are continuous in the $\Hk{2}$-norm, Cauchy-Schwarz, and the triangle inequality we have that
\begin{align*}
    Y_1&\leq C_1(\varepsilon) \sum_{\element \in \mesh}\left( \norm{\uh - \uEps}_{2,\element} + \norm{\uEps-\LtwoProj{\element}\polOrder \uEps}_{2,\element}\left)\right(\norm{\Psi_I - \Psi}_{2,\element} - \norm{\Psi-\LtwoProj{\element}2 \Psi}_{2,\element}\right)\\
    &\quad+\sum_{\element \in \mesh}\varepsilon\left(\norm{(\LtwoProj{\element}{\polOrder-2}-I)\hess \LtwoProj{\element}\polOrder\uEps}_{0,\element}\norm{\Psi_I-\Psi}_{2,\element} + \norm{(\LtwoProj{\element}{\polOrder-2} - I)\hess \LtwoProj{\element}\polOrder\uEps}_{0,\element} \norm{\Psi_I-\LtwoProj{\element}2\Psi}_{2,\element}\right)\\
    &\quad+\sum_{\element \in \mesh}\varepsilon\left( \norm{\uh-\uEps}_{2,\element}\norm{(\LtwoProj{\element}2-I)\Psi}_{2,\element} + \norm{(\LtwoProj{\element}{\polOrder-2} - I)\hess\uEps}_{2,\element}\norm{(\LtwoProj{\element}2-I) \Psi}_{2,\element}\right) \\
    &\quad+\sum_{\element \in \mesh}\varepsilon \norm{\uh-\LtwoProj{\element}\polOrder\uEps}_{2,\element}\norm{(\LtwoProj{\element}{\polOrder-2}-I)\hess\Psi}_{0,\element} + 0.
\end{align*}
Applying Lemma~\ref{lemma:l2proj}, Lemma~\ref{lemma:interpolation}, \eqref{eqn:dual:interpolation}, Theorem~\ref{theorem:nonlinear:h2} and a Sobolev embedding gives
\begin{align}
    Y_1 \leq (CC_8(\varepsilon)\varepsilon^{-1} + C\varepsilon^{-1} )h^r(\|\uEps\|_{r+1}+\|f\|_{r-2})\|z\|_{-1}. \label{y1}
\end{align}

For $Y_2$ we apply the boundary conditions \eqref{eqn:vanishing_moment:dbc}--\eqref{eqn:vanishing_moment:extra_bc}, Lemma~\ref{lemma:l2proj}, Lemma~\ref{lemma:proj_eq_l2proj} and \eqref{eqn:dual:interpolation} to get that
\begin{align}
    Y_2
    &= \varepsilon\sum_{\edge\in\edgesB} \int_\edge\left(\frac{\partial^2 \uEps}{\partial \tangential^2_\edge} - \laplace \uEps\right)\left(\normalProj \Psi_I - \frac{\partial \Psi_I}{\partial \normal_\edge}\right) \ds \nonumber \\
    &= \varepsilon\sum_{\edge\in\edgesB} \int_\edge\left(\frac{\partial^2 \uEps}{\partial \tangential^2_\edge} - \laplace \uEps -\LtwoProj{\edge}{\polOrder-2}\left(\frac{\partial^2 \uEps}{\partial \tangential^2_\edge} - \laplace \uEps\right) \right)\left(\LtwoProj{\edge}{\polOrder-1} - I\right)\left(\frac{\partial \Psi_I}{\partial \normal_\edge}\right) \ds \nonumber \\
    &= \varepsilon\sum_{\edge\in\edgesB} \int_\edge\left(\frac{\partial^2 \uEps}{\partial \tangential^2_\edge} - \laplace \uEps -\LtwoProj{\edge}{\polOrder-2}\left(\frac{\partial^2 \uEps}{\partial \tangential^2_\edge} - \laplace \uEps\right) \right)\left(\LtwoProj{\edge}{\polOrder-1} - I\right)\left(\frac{\partial (\Psi_I-\Psi)}{\partial \normal_\edge}\right) \ds \nonumber \\
    &\quad+\varepsilon\sum_{\edge\in\edgesB} \int_\edge\left(\frac{\partial^2 \uEps}{\partial \tangential^2_\edge} - \laplace \uEps -\LtwoProj{\edge}{\polOrder-2}\left(\frac{\partial^2 \uEps}{\partial \tangential^2_\edge} - \laplace \uEps\right) \right)\left(\LtwoProj{\edge}{\polOrder-1} - I\right)\left(\frac{\partial \Psi}{\partial \normal_\edge}\right) \ds \nonumber \\
    &\leq C\varepsilon\sum_{\edge\in\edgesB} \left( h^{r-1-\frac{1}{2}}_e\|\uEps\|_{r+1,\omega_\element} h^{1-\frac{1}{2}}_e|\Psi_I - \Psi|_{2,\omega_\element} + h^{r-1-\frac{1}{2}}_e\|\uEps\|_{r+1,\omega_\element} h^{1-\frac{1}{2}}_e|\Psi|_{2,\omega_\element} \right) \nonumber \\
    &\leq C\varepsilon^{-1} h^{r}\|\uEps\|_{r+1}\|z\|_{-1} \label{y2}
\end{align}
The bound for $Y_3$ follows similarly,
\begin{align}
    Y_3 &= \varepsilon \sum_{\edge \in \edgesI} \int_\edge \left( 
\frac{\partial^2 \Psi}{\partial \normal_\edge \partial \tangential_\edge} \jmp*{\frac{\partial \eh}{\partial \tangential_\edge}}
- \left(\laplace \Psi - \frac{\partial^2 \Psi}{\partial \tangential_\edge^2}\right) \jmp*{\frac{\partial \eh}{\partial \normal_\edge}}\right) \ds \nonumber \\
&\leq C\varepsilon^{-1} h^{r}\|\uEps\|_{r+1}\|z\|_{-1} \label{y3},
\end{align}
For $Y_4$ we again apply Lemma~\ref{lemma:l2proj} and \eqref{eqn:dual:interpolation} to show that
\begin{align}
    Y_4 &= \sum_{\element \in \mesh} (f_h -f ,\Psi_I) = \sum_{\element \in \mesh}(\LtwoProj{\element}{\polOrder-2}f -f ,\Psi_I - \LtwoProj{\element}1 \Psi_I) \nonumber\\
    &\leq \sum_{\element \in \mesh}\|\LtwoProj{\element}{\polOrder-2}f -f\|_{0,\element} \|\Psi_I - \LtwoProj{\element}1 \Psi_I\|_{0,\element} \nonumber \\
    &\leq C\sum_{\element \in \mesh}h^{r-2}_\element\|f\|_{r-2,K}h^2_\element\|\Psi_I\|_{2,\element} \nonumber \\
    &\leq C\varepsilon^{-2}h^r \|f\|_{r-2}\|z\|_{-1}. \label{y4}
\end{align}

$Y_5$ is the nonconforming error and can be estimated in the same way as the term $Q_5$ in the proof of Lemma \ref{lemma:fixed_point:operator_bound}; then, we can apply \eqref{eqn:dual:interpolation} to deduce that
\begin{align}
    Y_5 &\leq C\varepsilon h^{r-1}\|\uEps\|_{r+1}\|\Psi - \Psi_I\|_{2,h} \leq C\varepsilon^{-1} h^{r}\|\uEps\|_{r+1}\|z\|_{-1}.\label{y5}
\end{align}

For $Y_6$, we first note that as $\uEps\vert_\element,\uh\vert_\element\in \Hk{2}(\element)$, for all $\element\in\mesh$, we have from \eqref{eqn:variational}--\eqref{eqn:variational:a} and Assumption~\ref{assumption:elementwise_operator} that
\[
    \sum_{\element \in \mesh} \varepsilon \int_\element\hess \eh:\hess \Psi\dx = -\sum_{\element \in \mesh}\aQL^\element(\eh,\Psi) = \sum_{\element \in \mesh}\left(\bQL^\element(\uEps,\Psi)-\bQL^\element(\uh,\Psi)\right).
\]
Hence, from the definition and symmetry of $\AL^\element$, and the definition of $\AQL^\element$, $Y_6$ can be written as
\begin{align*}
    Y_6 &= \sum_{\element \in \mesh} \left(\AL^\element(\Psi,\eh) - \aQL^\element(\uh,\Psi_I) + \AQL^\element(\uEps,\Psi_I)-\bQLh^\element(\uh,\Psi_I)\right)\\
    &= \sum_{\element \in \mesh}\aQL^\element(\eh,\Psi_I) +  \sum_{\element \in \mesh}\left(\int_\element \uEpsCof\nabla\eh\cdot\nabla\Psi\dx+\bQL^\element(\uEps,\Psi)-\bQL^\element(\uh,\Psi)\right) \\
    &\quad+ \sum_{\element \in \mesh}(\bQL^\element(\uEps,\Psi_I)-\bQLh^\element(\uh,\Psi_I)).
\end{align*}
By applying Proposition~\ref{prop:mean_value} and integration by parts there exists a constant $t\in[0,1]$ such that
\begin{align}
    Y_6 &= \sum_{\element \in \mesh}\left(\aQL^\element(\eh,\Psi_I)+ \int_\element (\sigma_\varepsilon - \uEpsCof) : \hess\eh\Psi\dx\right) + \sum_{\element \in \mesh}\int_{\partial\element\setminus\partial\Omega} \uEpsCof \grad\eh \cdot \normal_\element\Psi \ds\nonumber \\
    &\quad+\sum_{\element \in \mesh} \bQL^\element(\uEps,\Psi_I) + \sum_{\element \in \mesh}(\bQLh^\element(\uEps,\Psi_I) -\bQLh^\element(\uh,\Psi_I)) - \sum_{\element \in \mesh}\bQLh^\element(\uEps,\Psi) \nonumber\\
    &\eqqcolon Y_6^{I}+Y_6^{II}+Y_6^{III}+Y_6^{IV}+Y_6^{V}\label{y6:split}
\end{align}
where $\sigma_\varepsilon = \cof(\hess (\uEps - t\eh)) = \uEpsCof - t\cof(\hess\eh)$.

We consider the five terms of $Y_6$ separately. For $Y_6^I$, we apply Lemma~\ref{lemma:interpolation}, Corollary~\ref{corol:h2full}, Theorem~\ref{theorem:nonlinear:h2}, standard scaling arguments and Sobolev embeddings to deduce that
\begin{align}
    Y_6^I &= \sum_{\element \in \mesh}\left(-\varepsilon\int_\element \hess\eh:\hess\Psi_I\dx- \int_\element t\cof(\hess\eh) : \hess\eh\Psi\dx\right) \nonumber\\
    &\leq \varepsilon \sum_{\element \in \mesh} \norm{\eh}_{2,\element}\left(\norm{\Psi-\Psi_I}_{2,\element}+\norm{\Psi}_{2,\element}\right) + \sum_{\element\in\mesh} t\norm{\cof(\hess\eh)}_{0,2,\element}\norm{\eh}_{2,\element}\norm{\Psi}_{0,\infty,\element} \nonumber\\
    &\leq  C\left(\varepsilon^{-1}C_8(\varepsilon)h^r(\norm{\uEps}_{r+1} + \norm{f}_{r-2})+ \varepsilon^{-2} C_8^2(\varepsilon) h^{2r-2} (\norm{\uEps}_{r+1}+ \norm{f}_{r-2})^2\right)\norm{z}_{-1}.\label{y6:1}
\end{align}

For the boundary term, $Y_6^{II}$, we note that for an edge $\edge\in\edgesI$ of an element $K\in\mesh$ that 
as $\|v\|_{0,\infty,\edge} \leq \|v\|_{0,\infty,\element}$, for all $v \in \Wkp{m,p}(\element)$, $mp > 2$. Then, by
Corollary~\ref{corol:h2full}, properties of the $\Lp$-projection, cf. Lemma~\ref{lemma:l2proj}, and standard scaling arguments,
\begin{align*}
    \int_\edge \uEpsCof \grad\eh \cdot \normal_\edge\Psi \ds
    &= \int_\edge (\uEpsCof - \LtwoProj{\edge}{0} (\uEpsCof)(\grad\eh - \LtwoProj{\edge}{0}(\grad\eh))\cdot \normal_\edge\Psi \ds \nonumber \\
    &\quad+ \int_\edge \LtwoProj{\edge}{0} (\uEpsCof)(\grad\eh - \LtwoProj{\edge}{0}\grad\eh)\cdot \normal_\edge(\Psi - \LtwoProj{\edge}{0}\Psi) \ds\nonumber \\
    &\leq \|\uEpsCof - \LtwoProj{\edge}{0} (\uEpsCof)\|_{0,\edge}\|\grad\eh - \LtwoProj{\edge}{0}\grad\eh\|_{0,\edge}\|\Psi\|_{0,\infty,\edge}\nonumber \\
    &\quad+C\|\LtwoProj{\edge}{0} (\uEpsCof)\|_{0,\infty,\edge}\|\grad\eh - \LtwoProj{\edge}{0}\grad\eh\|_{0,\edge}\|\Psi - \LtwoProj{\edge}{0}\Psi\|_{0,\edge} \nonumber\\
    &\leq C\big(h^{\nicefrac{1}{2}}_\edge\|\uEpsCof\|_{1,\omega_\edge}h^{\nicefrac{1}{2}}_\edge\|\grad\eh\|_{1,\omega_\edge}h^{-1}_{\element}\|\Psi\|_{0,2,\element}\\&\quad\qquad+ h^{-1}_{\element}\|\uEpsCof\|_{0,2,\element}h^{\nicefrac{1}{2}}_\edge\|\grad\eh\|_{1,\omega_\edge}h^{\nicefrac{1}{2}}_\edge\|\Psi\|_{1,\omega_\edge}\big).
\end{align*}
Hence, by the fact that $\|\uEpsCof\|_{1} = \|\hess\uEps\|_1 = \|\uEps\|_3 = \mathcal{O}(\varepsilon^{-1})$, Proposition~\ref{prop:ueps_bounds}, \eqref{eqn:dual:interpolation}, Theorem~\ref{theorem:nonlinear:h2} and a Sobolev embedding,
\begin{align}
    Y_6^{II} &\leq Ch\|\eh\|_{2,h}(\|\uEpsCof\|_{1}\|\Psi\|_{3}+\|\uEpsCof\|_{1}\|\Psi\|_{3})\notag\\&\leq CC_8(\varepsilon)\varepsilon^{-3} h^{r}(\|\uEps\|_{r+1}+ \norm{f}_{r-2})\|z\|_{-1}.\label{y6:2}
\end{align}

For $Y_{6}^{III}$, we follow identical steps to \eqref{eqn:fixed_point:operator_bound:e2:t1} and apply \eqref{eqn:dual:interpolation}, to get that
\begin{align}
    Y_{6}^{III} =\sum_{\element \in \mesh} \bQL^\element(\uEps,\Psi_I)
    &\leq C \varepsilon^{-\nicefrac{1}{2}}h^{r-1}\|\uEps\|_{r+1}\left(|\Psi_I-\Psi|_{2,h}+|\Psi|_{2,h}\right) \notag\\&\leq C \varepsilon^{-\nicefrac{5}{2}}h^{r}\|\uEps\|_{r+1}\|z\|_{-1}. \label{y6:3}
\end{align}

For $Y_{6}^{IV}$, we note that bounding the stabilization term follows similarly to the bound of $\mathcal{S}$ in the proof of Lemma~\ref{lemma:fixed_point:contraction}; and so, by applying Proposition~\ref{prop:mean_value}, Corollary~\ref{corol:h2full}, and Lemma~\ref{lemma:proj_eq_l2proj} it follows that
\begin{align*}
    Y_6^{IV} &=\sum_{\element \in \mesh}\int_\element(\det(\hessProj\uEps)-\det(\hessProj\uh)):\hessProj\Psi_I\dx \\
    &\quad+ \sum_{\element \in \mesh}\left(\StabQLb(\uEps-\valueProj\uh,\Psi_I-\valueProj\Psi_I)-\StabQLb(\uh-\valueProj\uh,\Psi_I-\valueProj\Psi_I)\right) \\
    &\leq \sum_{\element \in \mesh}\|\cof(\hessProj\uEps - t\hessProj\eh)\|_{0,2,\element}\|\hessProj(\uEps -\uh)\|_{0,2,\element}\| \hessProj\Psi_I \|_{\infty,K} \\
    &\quad+ \sum_{\element\in\mesh}\|\cof(\hessProj[\element,0] \eh)\|_{0,\infty,\element}\sum_{\lambda\in\lambda^\element}  \abs*{\lambda(\uEps - \valueProj \uEps)\lambda(\Psi_I - \valueProj \Psi_I)} \\
    &\quad +\sum_{\element\in\mesh} \|\cof(\hessProj[\element,0] \uh)\|_{0,\infty,\element}  \sum_{\lambda\in\lambda^\element}\abs*{\lambda(\eh - \valueProj \eh)\lambda(\Psi_I - \valueProj \Psi_I)} \\
    &\leq C\sum_{\element \in \mesh} (\|\uEps\|_{2,\element}+\|\eh\|_{2,\element})\|\eh\|_{2,\element}h^{-1}_\element(\|\Psi\|_{0,2,\element}+\|\Psi - \Psi_I\|_{0,2,\element}) \\
    &\quad+C\sum_{\element\in\mesh} \left(\|\uEps\|_{2,\element}+\|\uh\|_{2,\element}\right)h_\element\|\eh\|_{2,\element}(\|\Psi\|_{2,\element}+\|\Psi - \Psi_I\|_{2,\element}).
\end{align*}
By applying Lemma~\ref{lemma:interpolation}, Sobolev embeddings, Proposition~\ref{prop:ueps_bounds} and \eqref{eqn:dual:interpolation},
\begin{align}
    Y_6^{IV} &\leq C\sum_{\element \in \mesh} (\|\uEps\|_{2,\element}\|\eh\|_{2,\element}+\|\eh\|_{2,\element}\|\eh\|_{2,\element})h_\element(\|\Psi\|_{2,\element}+\|\Psi - \Psi_I\|_{2,\element}) \nonumber \\
    &\leq C\varepsilon^{-\nicefrac{5}{2}}C_8(\varepsilon)h^r(\|\uEps\|_{r+1}+ \norm{f}_{r-2})\|z\|_{-1} \nonumber \\&\quad+ C\varepsilon^{-2}C_8^2(\varepsilon)h^{2r-2}(\|\uEps\|_{r+1}+ \norm{f}_{r-2})^2\|z\|_{-1} \label{y6:4}.
\end{align}
Finally, we note that
\begin{align}
    Y_6^{V} &\leq \sum_{\element \in \mesh}\int_\element\abs{\det(\hessProj\uEps):\hessProj\Psi}\dx \nonumber \\&\quad+ \sum_{\element\in\mesh}\|\cof(\hessProj[\element,0] \uEps)\|_{0,\infty,\element}\sum_{\lambda\in\lambda^\element}  \abs*{\lambda(\uEps - \valueProj \uEps)\lambda(\Psi - \valueProj \Psi)}  \nonumber\\
    &\leq C\varepsilon^{-\nicefrac{5}{2}}h^r\|\uEps\|_{r+1}\|z\|_{-1}. \label{y6:5}
\end{align}
Therefore, by combining \eqref{y1}--\eqref{y6:5} with \eqref{eqn:dualh1},
\begin{align*}
    \norm{\uEps - \uh}_{1,\Omega}&\leq C_{9}(\varepsilon)h^{r}(\|\uEps\|_{r+1}+ \norm{f}_{r-2}) + C_{10}
    (\varepsilon)h^{2r-2}(\|\uEps\|_{r+1}+ \norm{f}_{r-2})^2,
\end{align*}
which completes the proof. \qed

\section{Numerical experiments}\label{sec:numerics}
In this section, we present numerical experiments to validate the theoretical results. In particular, we investigate the vanishing moment methodology and the VEM formulation, i.e., we examine whether the vanishing moment solution $u_\varepsilon$ converges to the viscosity solution $u_0$ as $\varepsilon \rightarrow 0$, and verify the convergence rates in Theorems~\ref{theorem:nonlinear:h2}, \ref{theorem:nonlinear:h1} and~\ref{theorem:nonlinear:l2} for polynomial orders $\polOrder = 2,3,4$.

The code to perform the numerical experiments is implemented within the Distributed and Unified Numerics Enviroment (DUNE) software framework \citep{bastian2008generic}. The virtual element method has been implemented within the DUNE-FEM module \citep{dedner2010generic} and further implementation details can be found in \citet{dedner2024framework}. The code, and corresponding numerical results, for the experiments in this section can be found in \citet{zenodo}.

All experiments are performed on the domain $\Omega = (0,1)^2$, with both uniform quadrilateral meshes and quasi-uniform Voronoi meshes with mesh size $h$ such that $h_\element\approx h$, for all $\element\in\mesh$.
We note that concave meshes are not considered, as is standard in the VEM literature, because we assume in Section~\ref{sec:intro} that the computational domain will be convex; consequently, due to the mesh construction, the resulting mesh elements are also convex and meet Assumption~\ref{assumption:mesh}.

For these experiments, we will consider the alternative stabilizations
\begin{align*}
    \StabQLb(u_h,v_h) &\coloneqq
    C
    \sum_{\lambda\in\Lambda^\element} \lambda(u_h)\lambda(v_h), \\
    \StabL(u_h,v_h) &\coloneqq
    \left(\varepsilon h_K^{-2} + C\right)
    \sum_{\lambda\in\Lambda^\element}  \lambda(u_h)\lambda(v_h), 
\end{align*}
with $C=1$ instead of $\norm{\cof(\hessProj[\element,0] u_h)}_{0,\infty,\element}$ and $ \norm{\uEpsCof}_{0,\infty,\element}$, respectively; cf. \eqref{eqn:stab:ql:b}--\eqref{eqn:stab:linear}. As both are second order, if we use the fixed point map in \eqref{eqn:T} as a nonlinear solver, the constants for the stabilization should be the same, i.e., $\norm{\cof(\hessProj[\element,0] u_h)}_{0,\infty,\element} =  \norm{\uEpsCof}_{0,\infty,\element} = C$. We choose $C=1$ based on the observation that $\norm{\uEpsCof}_{0,\infty,\element} = \mathcal{O}(\varepsilon^{-1}) = \mathcal{C}\varepsilon^{-1}$, where $\mathcal{C}$ is some constant independent of the mesh size $h_\element$, and the exact value of $\mathcal{C}\varepsilon^{-1}$ is not known. Moreover, it is also not the correct scaling for $\uEpsCof \gradProj u_h \cdot \gradProj v_h$ as $\uEpsCof$ is treated as a constant in theory, but its numerical value is unknown. Therefore, for robustness, we choose to stick with a simple global constant. 

\subsection{Convergence of vanishing moment method to Monge--Amp\`ere}
\begin{figure}[t]
    \centering
    \subfloat[][Quadrilateral mesh]{\label{fig:eps:quads}%
    \begin{tikzpicture}
        \begin{loglogaxis}[
                xlabel={$\varepsilon$}, 
                ylabel={$\norm{\uh-u_0}$},
                LargePlot
            ]
            \addplot+[] table[x=Epsilon,y=H2] {results/eps_quads_p3.dat};
            \addplot+[] table[x=Epsilon,y=H1] {results/eps_quads_p3.dat};
            \addplot+[] table[x=Epsilon,y=L2] {results/eps_quads_p3.dat};
            \addplot+[] coordinates {(1,5) (0.0005,1.12151159)};
            \addplot+[] coordinates {(1,2.5) (0.0005,0.0167185076)};
            \addplot+[] coordinates {(1,2.5) (0.0005,0.00125)};
            \legend{$\Hk{2}$-error,$\Hk{1}$-error,$\Lp$-error}
        \end{loglogaxis}
    \end{tikzpicture}}\quad
    \subfloat[Voronoi mesh]{\label{fig:eps:voronoi}%
    \begin{tikzpicture}
        \begin{loglogaxis}[
                xlabel={$\varepsilon$}, 
                ylabel={$\norm{\uh-u_0}$},
                LargePlot
            ]
            \addplot+[] table[x=Epsilon,y=H2] {results/eps_voronoi_p3.dat};
            \addplot+[] table[x=Epsilon,y=H1] {results/eps_voronoi_p3.dat};
            \addplot+[] table[x=Epsilon,y=L2] {results/eps_voronoi_p3.dat};
            \addplot+[] coordinates {(1,5) (0.0005,1.12151159)};
            \addplot+[] coordinates {(1,2.5) (0.0005,0.0167185076)};
            \addplot+[] coordinates {(1,2.5) (0.0005,0.00125)};
            \legend{$\Hk{2}$-error,$\Hk{1}$-error,$\Lp$-error}
        \end{loglogaxis}
    \end{tikzpicture}}\medskip\\
    \caption{Convergence of $\uh$ to $u_0$ in the $\Hk{2}$-, $\Hk{1}$-, and $\Lp$-norms as $\varepsilon\to 0$ on \protect\subref{fig:eps:quads} uniform quadrilateral and \protect\subref{fig:eps:voronoi} quasi-uniform Voronoi meshes. Dashed lines show $\mathcal{O}(\varepsilon^{\nicefrac14})$, $\mathcal{O}(\varepsilon^{\nicefrac34})$, and $\mathcal{O}(\varepsilon)$ convergence in $\Hk{2}$-, $\Hk{1}$-, and $\Lp$-norms, respectively.}
    \label{fig:eps}
\end{figure}
We first perform the experiment \citet[Test 1(b)]{neilan2010nonconforming} for the virtual element method to examine whether $\uEps \rightarrow u_0$ as $\varepsilon \rightarrow 0$. To this end, we consider the error $\|\uh - u_0\|$ for $\varepsilon=1, 0.5, 0.25, 0.125, 0.05, 0.025, 0.0125, 0.005, 0.0025, 0.00125, 0.0005$, in the $\Lp$-, $\Hk{1}$- and $\Hk{2}$-norms, with $\polOrder=3$ on a fixed uniform quadrilateral or quasi-uniform Voronoi mesh with $h \approx 0.019$. Since $h$ is sufficiently small, this is effectively equivalent to computing the error $\uEps - u_0$ as the discretization error is negligible. We select the right-hand forcing function $f=4$ and boundary condition $g$ such that the analytical solution to \eqref{eqn:monge_ampere}--\eqref{eqn:monge_ampere:bc} is given by
\begin{equation}\label{eqn:u0}
    u_0(x,y) = x^2+y^2.
\end{equation}
We solve the resulting nonlinear system in \eqref{eqn:vem} using a Newton-Krylov method, and at each Newton step the linearized system is solved using a sparse direct solver.

The system resulting from \eqref{eqn:vem} will become ill-conditioned as $\varepsilon \rightarrow 0$. Therefore, to ensure convergence, we adopt a continuation strategy, using the solution to \eqref{eqn:vem} for a specific $\varepsilon \in (0,1)$ as the initial guess for solving with the next smaller value of $\varepsilon$. In Fig.~\ref{fig:eps}, we plot the $\Lp$-, $\Hk{1}$- and $\Hk{2}$-errors of $\|\uh - u_0\|$ versus $\varepsilon$. We see that all errors decrease as $\varepsilon$ tends to zero, which confirms that the vanishing moment method works in the VEM setting. From Fig.~\ref{fig:eps}, we see that $\norm{\uEps - u_0}_{2,h} \approx \mathcal{O}(\varepsilon^{\nicefrac{1}{4}}), \norm{\uEps - u_0}_{1,\Omega} \approx \mathcal{O}(\varepsilon^{\nicefrac{3}{4}})$ and $\norm{\uEps - u_0}_{0,\Omega} \approx \mathcal{O}(\varepsilon)$. These rates match with the existing results in the vanishing moment literature; cf., \citet{neilan2009numerical, neilan2010nonconforming}.

\begin{figure}%
    \centering%
    \subfloat[Quadrilateral mesh; $H^2$-error]{%
    \begin{tikzpicture}\label{fig:epshcoupling:h2}
        \begin{axis}[
            xlabel={$\varepsilon$}, 
            ylabel={$\| \hess u_0 - \hessProj \uh \|_{0,h}$},
            LargePlot,
            scaled ticks=false,
            yticklabel style={/pgf/number format/fixed},
            xticklabel style={/pgf/number format/fixed}
        ]
            \addplot+[] table[x=epsilon,y=H2] {results/stage3_paper_H2_q_1p00_quads_p3.dat};
            \pgfplotsset{cycle list shift=2}
            \addplot+[domain=0.0001:0.05, samples=100] {4.8074554137216658*x^0.25};
            \legend{Error ($h=\varepsilon$),$\mathcal{O}(\varepsilon^{\nicefrac14})$}
        \end{axis}
    \end{tikzpicture}}%
    \quad\subfloat[Quadrilateral mesh; $H^1$-error]{%
    \begin{tikzpicture}\label{fig:epshcoupling:h1}
        \begin{axis}[
            xlabel={$\varepsilon$}, 
            ylabel={$\| \grad u_0 - \gradProj \uh \|_{0,h}$},
            LargePlot,
            scaled ticks=false,
            yticklabel style={/pgf/number format/fixed},
            xticklabel style={/pgf/number format/fixed}
        ]
            \addplot+[] table[x=epsilon,y=H1] {results/stage3_paper_H1_q_1p00_quads_p3.dat};
            \pgfplotsset{cycle list shift=2}
            \addplot+[domain=0.0001:0.1, samples=100] {2.7502633841169182*x^.75};
            \legend{Error ($h=\varepsilon$),$\mathcal{O}(\varepsilon^{\nicefrac34})$}
        \end{axis}
    \end{tikzpicture}}\\%
    \quad\subfloat[Quadrilateral mesh; $L^2$-error]{%
    \begin{tikzpicture}\label{fig:epshcoupling:l2}
        \begin{axis}[
            xlabel={$\varepsilon$}, 
            ylabel={$\| u_0 - \valueProj \uh \|_{0,h}$},
            LargePlot,
            scaled ticks=false,
            yticklabel style={/pgf/number format/fixed},
            xticklabel style={/pgf/number format/fixed},
            yticklabel style={/pgf/number format/precision=3},
            xticklabel style={/pgf/number format/precision=3}
        ]
            \addplot+[] table[x=epsilon,y=L2] {results/stage3_paper_L2_q_2p00_quads_p3.dat};
            \pgfplotsset{cycle list shift=2}
            \addplot+[domain=0:0.01, samples=2] {1.7028578702482258*x};
            \legend{Error ($h=\varepsilon^{\nicefrac12}$),$\mathcal{O}(\varepsilon)$}
        \end{axis}
    \end{tikzpicture}}\medskip%
    \caption{Convergence for sequence of uniform quadrilateral meshes of $N\times N$ elements, for $N=10,20,30,40,50$, in the \protect\subref{fig:epshcoupling:h2} $\Hk{2}$-norm with $\varepsilon=h$, \protect\subref{fig:epshcoupling:h1} $\Hk{1}$-norm with $\varepsilon=h$, and \protect\subref{fig:epshcoupling:l2} $\Lp$-norm with $\varepsilon=h^2$.}
    \label{fig:epshcoupling}
\end{figure}%
We now perform a similar experiment to \citet[Test 3]{feng2009mixed}, where we attempt to numerical deduce a relationship between the mesh size $h$ and $\varepsilon$. To that end, we solve the problem where $u_0$ is given by \eqref{eqn:u0} on a uniform $N\times N$ quadrilateral meshes, for $N=10,20,30,40,50$, with $\varepsilon$ set as a function of $h$. We then perform least squares fitting to fit the error to the expected convergence rate in $\varepsilon$ from the previous experiment to find the best matching relationship for the expected convergence rate; namely, $\norm{\uEps - u_0}_{2,h} \approx \mathcal{O}(\varepsilon^{\nicefrac{1}{4}}), \norm{\uEps - u_0}_{1,\Omega} \approx \mathcal{O}(\varepsilon^{\nicefrac{3}{4}})$ and $\norm{\uEps - u_0}_{0,\Omega} \approx \mathcal{O}(\varepsilon)$. Fig.~\ref{fig:epshcoupling} shows the converge rate roughly matches the expected when we set $\varepsilon$ such that $h=\varepsilon$ for the $\Hk{1}$- and $\Hk{2}$-norms, and $h=\varepsilon^{\nicefrac12}$ for the $\Lp$-norm. A similar experiment was conducted for a quasi-uniform Voronoi mesh with identical results and, hence, are not included here; cf. \citet{pradhanPhD}.

\subsection{Validation of Theorems~\ref{theorem:nonlinear:h2}--\ref{theorem:nonlinear:l2}}
We now consider three problems with a known solution to examine whether the rate of convergence of the error in $\Hk{2}$-, $\Hk{1}$-, and $\Lp$-norm corresponds to the analytical bounds in Theorems~\ref{theorem:nonlinear:h2}, \ref{theorem:nonlinear:h1}, and~\ref{theorem:nonlinear:l2}, respectively. To this end, we fix $\varepsilon = 0.01$ and solve \eqref{eqn:vem} on a sequence of uniform quadrilateral and quasi-uniform Voronoi meshes with $11^2, 20^2, 40^2, 80^2, 160^2$, and $320^2$ elements, for polynomial orders $\polOrder = 2,3,4$. In this experiment we use an alterative boundary condition $\laplace \uEps = \phi_\varepsilon$, cf. \cite{feng2011vanishing}, instead of the original boundary condition $\laplace \uEps = \varepsilon$ in order to derive an exact known solution.

The numerical experiments are performed using the fixed point defined in \eqref{eqn:T} with added damping and a tolerance of $10^{-9}$. From our experience, we observed that the method is very sensitive to the initial value; therefore, we initially solve with a relatively large $\varepsilon = 1$ with a tolerance of $0.1$ and a maximum of $10$ iterations to generate the initial value.

\subsubsection{Problem 1}
\begin{figure}[p]%
    \centering%
    \subfloat[Quadrilateral mesh; $H^2$-error]{%
    \begin{tikzpicture}
        \begin{loglogaxis}[
            xlabel={$h$}, 
            ylabel={$\| \hess \uEps - \hessProj \uh \|_{0,h}$},
            LargePlot
        ]
            \addplot+[] table[x=h,y=H2] {results/test1_quads_p2.dat};
            \addplot+[] table[x=h,y=H2] {results/test1_quads_p3.dat};
            \addplot+[] table[x=h,y=H2] {results/test1_quads_p4.dat};
            \addplot+[domain=0.003125:0.1, samples=2] {x};
            \addplot+[domain=0.003125:0.1, samples=2] {0.1*x^2};
            \addplot+[domain=0.003125:0.1, samples=2] {0.01*x^3};
            \legend{$\polOrder=2$,$\polOrder=3$,$\polOrder=4$}
        \end{loglogaxis}
    \end{tikzpicture}}%
    \quad\subfloat[Voronoi mesh; $H^2$-error]{%
    \begin{tikzpicture}
        \begin{loglogaxis}[
            xlabel={$h$}, 
            ylabel={$\| \hess \uEps - \hessProj \uh \|_{0,h}$},
            LargePlot
        ]
            \addplot+[] table[x=h,y=H2] {results/test1_voronoi_p2.dat};
            \addplot+[] table[x=h,y=H2] {results/test1_voronoi_p3.dat};
            \addplot+[] table[x=h,y=H2] {results/test1_voronoi_p4.dat};
            \addplot+[domain=0.003125:0.1, samples=2] {x};
            \addplot+[domain=0.003125:0.1, samples=2] {0.1*x^2};
            \addplot+[domain=0.003125:0.1, samples=2] {0.01*x^3};
            \legend{$\polOrder=2$,$\polOrder=3$,$\polOrder=4$}
        \end{loglogaxis}
    \end{tikzpicture}}\\\bigskip%
    \subfloat[Quadrilateral mesh; $H^1$-error]{%
    \begin{tikzpicture}
        \begin{loglogaxis}[
            xlabel={$h$}, 
            ylabel={$\| \grad \uEps - \gradProj \uh \|_{0,h}$},
            LargePlot
        ]
            \addplot+[] table[x=h,y=H1] {results/test1_quads_p2.dat};
            \addplot+[] table[x=h,y=H1] {results/test1_quads_p3.dat};
            \addplot+[] table[x=h,y=H1] {results/test1_quads_p4.dat};
            \addplot+[domain=0.003125:0.1, samples=2] {0.4*x^2};
            \addplot+[domain=0.003125:0.1, samples=2] {0.5*x^3};
            \addplot+[domain=0.003125:0.1, samples=2] {0.4*x^4};
            \legend{$\polOrder=2$,$\polOrder=3$,$\polOrder=4$}
        \end{loglogaxis}
    \end{tikzpicture}}%
    \quad\subfloat[Voronoi mesh; $H^1$-error]{%
    \begin{tikzpicture}
        \begin{loglogaxis}[
            xlabel={$h$}, 
            ylabel={$\| \grad \uEps - \gradProj \uh \|_{0,h}$},
            LargePlot
        ]
            \addplot+[] table[x=h,y=H1] {results/test1_voronoi_p2.dat};
            \addplot+[] table[x=h,y=H1] {results/test1_voronoi_p3.dat};
            \addplot+[] table[x=h,y=H1] {results/test1_voronoi_p4.dat};
            \addplot+[domain=0.003125:0.1, samples=2] {0.4*x^2};
            \addplot+[domain=0.003125:0.1, samples=2] {0.5*x^3};
            \addplot+[domain=0.003125:0.1, samples=2] {0.2*x^4};
            \legend{$\polOrder=2$,$\polOrder=3$,$\polOrder=4$}
        \end{loglogaxis}
    \end{tikzpicture}}\\\bigskip%
    \subfloat[Quadrilateral mesh; $L^2$-error]{%
    \begin{tikzpicture}
        \begin{loglogaxis}[
            xlabel={$h$}, 
            ylabel={$\| \uEps - \valueProj \uh \|_{0,h}$},
            LargePlot
        ]
            \addplot+[] table[x=h,y=L2] {results/test1_quads_p2.dat};
            \addplot+[] table[x=h,y=L2] {results/test1_quads_p3.dat};
            \addplot+[] table[x=h,y=L2] {results/test1_quads_p4.dat};
            \addplot+[domain=0.003125:0.1, samples=2] {x^2};
            \addplot+[domain=0.003125:0.1, samples=2] {0.1*x^4};
            \addplot+[domain=0.003125:0.1, samples=2] {0.1*x^5};
            \legend{$\polOrder=2$,$\polOrder=3$,$\polOrder=4$}
        \end{loglogaxis}
    \end{tikzpicture}}%
    \quad\subfloat[Voronoi mesh; $L^2$-error]{%
    \begin{tikzpicture}
        \begin{loglogaxis}[
            xlabel={$h$}, 
            ylabel={$\| \uEps - \valueProj \uh \|_{0,h}$},
            LargePlot
        ]
            \addplot+[] table[x=h,y=L2] {results/test1_voronoi_p2.dat};
            \addplot+[] table[x=h,y=L2] {results/test1_voronoi_p3.dat};
            \addplot+[] table[x=h,y=L2] {results/test1_voronoi_p4.dat};
            \addplot+[domain=0.003125:0.1, samples=2] {x^2};
            \addplot+[domain=0.003125:0.1, samples=2] {0.1*x^4};
            \addplot+[domain=0.003125:0.1, samples=2] {0.05*x^5};
            \legend{$\polOrder=2$,$\polOrder=3$,$\polOrder=4$}
        \end{loglogaxis}
    \end{tikzpicture}}\bigskip%
    \caption{Problem 1: Convergence of the $\Hk{2}$- (top), $\Hk{1}$- (middle), and $\Lp$-error (bottom) with respect to mesh size $h$ on a uniform mesh of quadrilateral elements (left) and a quasi-uniform mesh of Voronoi elements (right) with approximation order $\polOrder=2,3,4$. Dashed lines show that the expected order of convergence ($\mathcal{O}(h^{\polOrder-1})$ for $\Hk{2}$-error, $\mathcal{O}(h^{\polOrder})$ for $\Hk{1}$-error and $\Lp$-error with $\polOrder = 2$, and $\mathcal{O}(h^{\polOrder+1})$ for $\Lp$-error with $\polOrder > 2$) matches the actual order.}
    \label{fig:test1}
\end{figure}
For the first problem we consider Test 6.2(b) from \citet{neilan2009numerical}. To this end, we define $f$, $g$, and $\phi_\varepsilon$ such that the solution of \eqref{eqn:vanishing_moment}--\eqref{eqn:vanishing_moment:extra_bc} is given by
\[
    \uEps(x,y) = \frac{x^4+y^4}{12}.
\]
\begin{figure}[p]%
    \centering%
    \subfloat[Quadrilateral mesh; $H^2$-error]{%
    \begin{tikzpicture}
        \begin{loglogaxis}[
            xlabel={$h$}, 
            ylabel={$\| \hess \uEps - \hessProj \uh \|_{0,h}$},
            LargePlot
        ]
            \addplot+[] table[x=h,y=H2] {results/test2_quads_p2.dat};
            \addplot+[] table[x=h,y=H2] {results/test2_quads_p3.dat};
            \addplot+[] table[x=h,y=H2] {results/test2_quads_p4.dat};
            \addplot+[domain=0.003125:0.1, samples=2] {x};
            \addplot+[domain=0.003125:0.1, samples=2] {2*x^2};
            \addplot+[domain=0.003125:0.1, samples=2] {.5*x^3};
            \legend{$\polOrder=2$,$\polOrder=3$,$\polOrder=4$}
        \end{loglogaxis}
    \end{tikzpicture}}%
    \quad\subfloat[Voronoi mesh; $H^2$-error]{%
    \begin{tikzpicture}
        \begin{loglogaxis}[
            xlabel={$h$}, 
            ylabel={$\| \hess \uEps - \hessProj \uh \|_{0,h}$},
            LargePlot
        ]
            \addplot+[] table[x=h,y=H2] {results/test2_voronoi_p2.dat};
            \addplot+[] table[x=h,y=H2] {results/test2_voronoi_p3.dat};
            \addplot+[] table[x=h,y=H2] {results/test2_voronoi_p4.dat};
            \addplot+[domain=0.003125:0.1, samples=2] {x};
            \addplot+[domain=0.003125:0.1, samples=2] {2*x^2};
            \addplot+[domain=0.003125:0.1, samples=2] {0.25*x^3};
            \legend{$\polOrder=2$,$\polOrder=3$,$\polOrder=4$}
        \end{loglogaxis}
    \end{tikzpicture}}\\\bigskip%
    \subfloat[Quadrilateral mesh; $H^1$-error]{%
    \begin{tikzpicture}
        \begin{loglogaxis}[
            xlabel={$h$}, 
            ylabel={$\| \grad \uEps - \gradProj \uh \|_{0,h}$},
            LargePlot
        ]
            \addplot+[] table[x=h,y=H1] {results/test2_quads_p2.dat};
            \addplot+[] table[x=h,y=H1] {results/test2_quads_p3.dat};
            \addplot+[] table[x=h,y=H1] {results/test2_quads_p4.dat};
            \addplot+[domain=0.003125:0.1, samples=2] {2*x^2};
            \addplot+[domain=0.003125:0.1, samples=2] {5*x^3};
            \addplot+[domain=0.003125:0.1, samples=2] {2*x^4};
            \legend{$\polOrder=2$,$\polOrder=3$,$\polOrder=4$}
        \end{loglogaxis}
    \end{tikzpicture}}%
    \quad\subfloat[Voronoi mesh; $H^1$-error]{%
    \begin{tikzpicture}
        \begin{loglogaxis}[
            xlabel={$h$}, 
            ylabel={$\| \grad \uEps - \gradProj \uh \|_{0,h}$},
            LargePlot
        ]
            \addplot+[] table[x=h,y=H1] {results/test2_voronoi_p2.dat};
            \addplot+[] table[x=h,y=H1] {results/test2_voronoi_p3.dat};
            \addplot+[] table[x=h,y=H1] {results/test2_voronoi_p4.dat};
            \addplot+[domain=0.003125:0.1, samples=2] {x^2};
            \addplot+[domain=0.003125:0.1, samples=2] {2*x^3};
            \addplot+[domain=0.003125:0.1, samples=2] {0.5*x^4};
            \legend{$\polOrder=2$,$\polOrder=3$,$\polOrder=4$}
        \end{loglogaxis}
    \end{tikzpicture}}\\\bigskip%
    \subfloat[Quadrilateral mesh; $L^2$-error]{%
    \begin{tikzpicture}
        \begin{loglogaxis}[
            xlabel={$h$}, 
            ylabel={$\| \uEps - \valueProj \uh \|_{0,h}$},
            LargePlot
        ]
            \addplot+[] table[x=h,y=L2] {results/test2_quads_p2.dat};
            \addplot+[] table[x=h,y=L2] {results/test2_quads_p3.dat};
            \addplot+[] table[x=h,y=L2] {results/test2_quads_p4.dat};
            \addplot+[domain=0.003125:0.1, samples=2] {x^2};
            \addplot+[domain=0.003125:0.1, samples=2] {x^4};
            \addplot+[domain=0.003125:0.1, samples=2] {0.1*x^5};
            \legend{$\polOrder=2$,$\polOrder=3$,$\polOrder=4$}
        \end{loglogaxis}
    \end{tikzpicture}}%
    \quad\subfloat[Voronoi mesh; $L^2$-error]{%
    \begin{tikzpicture}
        \begin{loglogaxis}[
            xlabel={$h$}, 
            ylabel={$\| \uEps - \valueProj \uh \|_{0,h}$},
            LargePlot
        ]
            \addplot+[] table[x=h,y=L2] {results/test2_voronoi_p2.dat};
            \addplot+[] table[x=h,y=L2] {results/test2_voronoi_p3.dat};
            \addplot+[] table[x=h,y=L2] {results/test2_voronoi_p4.dat};
            \addplot+[domain=0.003125:0.1, samples=2] {0.25*x^2};
            \addplot+[domain=0.003125:0.1, samples=2] {0.5*x^4};
            \addplot+[domain=0.003125:0.1, samples=2] {0.01*x^5};
            \legend{$\polOrder=2$,$\polOrder=3$,$\polOrder=4$}
        \end{loglogaxis}
    \end{tikzpicture}}\bigskip%
    \caption{Problem 2: Convergence of the $\Hk{2}$- (top), $\Hk{1}$- (middle), and $\Lp$-error (bottom) with respect to mesh size $h$ on a uniform mesh of quadrilateral elements (left) and a quasi-uniform mesh of Voronoi elements (right) with approximation order $\polOrder=2,3,4$. Dashed lines show that the expected order of convergence ($\mathcal{O}(h^{\polOrder-1})$ for $\Hk{2}$-error, $\mathcal{O}(h^{\polOrder})$ for $\Hk{1}$-error and $\Lp$-error with $\polOrder = 2$, and $\mathcal{O}(h^{\polOrder+1})$ for $\Lp$-error with $\polOrder > 2$) matches the actual order.}
    \label{fig:test2}
\end{figure}

Fig.~\ref{fig:test1} show the convergence of the error in the $\Hk{2}$-, $\Hk{1}$-, and $\Lp$-norm for the uniform quadilateral and quasi-uniform Voronoi meshes. Additionally, the dashed lines on these figures display the expected order of convergence of $\mathcal{O}(h^{\polOrder-1})$ for the $\Hk{2}$-error, $\mathcal{O}(h^{\polOrder})$ for the $\Hk{1}$-error, $\mathcal{O}(h^{2})$ for the $\Lp$-error with $\polOrder=2$, and $\mathcal{O}(h^{\polOrder+1})$ for the $\Lp$-error with $\polOrder\geq 3$; cf. Theorems~\ref{theorem:nonlinear:h2}, \ref{theorem:nonlinear:h1}, and~\ref{theorem:nonlinear:l2}, respectively. In general, the actual order of convergence obtained matches the expected order of convergence, although we appear to gain some super-convergence for the $\Hk{1}$-error on the quadrilateral mesh with $\polOrder=4$, and slight suboptimality for the $\Hk{2}$-error on the Voronoi mesh with $\polOrder=4$. Additionally, we note that on the finest mesh the $\Lp$-error for $\polOrder=4$ appears to rise; we note that this is likely due to machine precision issues and potentially the nonlinear solver tolerance. We also note that we get suboptimality for $\polOrder=2$ for the $\Lp$-error as mentioned in Remark~\ref{remark:suboptimal}.

\subsubsection{Problem 2}
We perform Test 6.2(a) from \citet{neilan2009numerical} and set  $f$, $g$, and $\phi_\varepsilon$ such that the solution of \eqref{eqn:vanishing_moment}--\eqref{eqn:vanishing_moment:extra_bc} is given by
\[
  \uEps(x,y) = \exp\left(\frac{x^2+y^2}{2}\right).
\]
In Fig.~\ref{fig:test2} we can see that the convergence in the $\Hk{2}$-, $\Hk{1}$-, and $\Lp$-norm for the uniform quadilateral and quasi-uniform Voronoi meshes conforms with Theorems~\ref{theorem:nonlinear:h2}, \ref{theorem:nonlinear:h1}, and~\ref{theorem:nonlinear:l2}, with the expected suboptimality for $\polOrder=2$ for the $\Lp$-error as mentioned in Remark~\ref{remark:suboptimal}.

\subsubsection{Problem 3}
\begin{figure}[p]%
    \centering%
    \subfloat[Quadrilateral mesh; $H^2$-error]{%
    \begin{tikzpicture}
        \begin{loglogaxis}[
            xlabel={$h$}, 
            ylabel={$\| \hess \uEps - \hessProj \uh \|_{0,h}$},
            LargePlot
        ]
            \addplot+[] table[x=h,y=H2] {results/test3_quads_p2.dat};
            \addplot+[] table[x=h,y=H2] {results/test3_quads_p3.dat};
            \addplot+[] table[x=h,y=H2] {results/test3_quads_p4.dat};
            \addplot+[domain=0.003125:0.1, samples=2] {0.25*x};
            \addplot+[domain=0.003125:0.1, samples=2] {x^2};
            \addplot+[domain=0.003125:0.1, samples=2] {0.01*x^3};
            \legend{$\polOrder=2$,$\polOrder=3$,$\polOrder=4$}
        \end{loglogaxis}
    \end{tikzpicture}}%
    \quad\subfloat[Voronoi mesh; $H^2$-error]{%
    \begin{tikzpicture}
        \begin{loglogaxis}[
            xlabel={$h$}, 
            ylabel={$\| \hess \uEps - \hessProj \uh \|_{0,h}$},
            LargePlot
        ]
            \addplot+[] table[x=h,y=H2] {results/test3_voronoi_p2.dat};
            \addplot+[] table[x=h,y=H2] {results/test3_voronoi_p3.dat};
            \addplot+[] table[x=h,y=H2] {results/test3_voronoi_p4.dat};
            \addplot+[domain=0.003125:0.1, samples=2] {0.25*x};
            \addplot+[domain=0.003125:0.1, samples=2] {0.25*x^2};
            \addplot+[domain=0.003125:0.1, samples=2] {0.02*x^3};
            \legend{$\polOrder=2$,$\polOrder=3$,$\polOrder=4$}
        \end{loglogaxis}
    \end{tikzpicture}}\\\bigskip%
    \subfloat[Quadrilateral mesh; $H^1$-error]{%
    \begin{tikzpicture}
        \begin{loglogaxis}[
            xlabel={$h$}, 
            ylabel={$\| \grad \uEps - \gradProj \uh \|_{0,h}$},
            LargePlot
        ]
            \addplot+[] table[x=h,y=H1] {results/test3_quads_p2.dat};
            \addplot+[] table[x=h,y=H1] {results/test3_quads_p3.dat};
            \addplot+[] table[x=h,y=H1] {results/test3_quads_p4.dat};
            \addplot+[domain=0.003125:0.1, samples=2] {0.5*x^2};
            \addplot+[domain=0.003125:0.1, samples=2] {x^3};
            \addplot+[domain=0.003125:0.1, samples=2] {0.05*x^4};
            \legend{$\polOrder=2$,$\polOrder=3$,$\polOrder=4$}
        \end{loglogaxis}
    \end{tikzpicture}}%
    \quad\subfloat[Voronoi mesh; $H^1$-error]{%
    \begin{tikzpicture}
        \begin{loglogaxis}[
            xlabel={$h$}, 
            ylabel={$\| \grad \uEps - \gradProj \uh \|_{0,h}$},
            LargePlot
        ]
            \addplot+[] table[x=h,y=H1] {results/test3_voronoi_p2.dat};
            \addplot+[] table[x=h,y=H1] {results/test3_voronoi_p3.dat};
            \addplot+[] table[x=h,y=H1] {results/test3_voronoi_p4.dat};
            \addplot+[domain=0.003125:0.1, samples=2] {0.5*x^2};
            \addplot+[domain=0.003125:0.1, samples=2] {0.8*x^3};
            \addplot+[domain=0.003125:0.1, samples=2] {0.05*x^4};
            \legend{$\polOrder=2$,$\polOrder=3$,$\polOrder=4$}
        \end{loglogaxis}
    \end{tikzpicture}}\\\bigskip%
    \subfloat[Quadrilateral mesh; $L^2$-error]{%
    \begin{tikzpicture}
        \begin{loglogaxis}[
            xlabel={$h$}, 
            ylabel={$\| \uEps - \valueProj \uh \|_{0,h}$},
            LargePlot
        ]
            \addplot+[] table[x=h,y=L2] {results/test3_quads_p2.dat};
            \addplot+[] table[x=h,y=L2] {results/test3_quads_p3.dat};
            \addplot+[] table[x=h,y=L2] {results/test3_quads_p4.dat};
            \addplot+[domain=0.003125:0.1, samples=2] {0.1*x^2};
            \addplot+[domain=0.003125:0.1, samples=2] {0.1*x^4};
            \addplot+[domain=0.003125:0.1, samples=2] {0.05*x^5};
            \legend{$\polOrder=2$,$\polOrder=3$,$\polOrder=4$}
        \end{loglogaxis}
    \end{tikzpicture}}%
    \quad\subfloat[Voronoi mesh; $L^2$-error]{%
    \begin{tikzpicture}
        \begin{loglogaxis}[
            xlabel={$h$}, 
            ylabel={$\| \uEps - \valueProj \uh \|_{0,h}$},
            LargePlot
        ]
            \addplot+[] table[x=h,y=L2] {results/test3_voronoi_p2.dat};
            \addplot+[] table[x=h,y=L2] {results/test3_voronoi_p3.dat};
            \addplot+[] table[x=h,y=L2] {results/test3_voronoi_p4.dat};
            \addplot+[domain=0.003125:0.1, samples=2] {0.05*x^2};
            \addplot+[domain=0.003125:0.1, samples=2] {0.1*x^4};
            \addplot+[domain=0.003125:0.1, samples=2] {0.05*x^5};
            \legend{$\polOrder=2$,$\polOrder=3$,$\polOrder=4$}
        \end{loglogaxis}
    \end{tikzpicture}}\bigskip%
    \caption{Problem 3: Convergence of the $\Hk{2}$- (top), $\Hk{1}$- (middle), and $\Lp$-error (bottom) with respect to mesh size $h$ on a uniform mesh of quadrilateral elements (left) and a quasi-uniform mesh of Voronoi elements (right) with approximation order $\polOrder=2,3,4$. Dashed lines show that the expected order of convergence ($\mathcal{O}(h^{\polOrder-1})$ for $\Hk{2}$-error, $\mathcal{O}(h^{\polOrder})$ for $\Hk{1}$-error and $\Lp$-error with $\polOrder = 2$, and $\mathcal{O}(h^{\polOrder+1})$ for $\Lp$-error with $\polOrder > 2$) matches the actual order.}
    \label{fig:test3}
\end{figure}
Finally, we perform the experiment Test 3.3(b) from \citet{neilan2009numerical} and set  $f$, $g$, and $\phi_\varepsilon$ such that the solution of \eqref{eqn:vanishing_moment}--\eqref{eqn:vanishing_moment:extra_bc} is given by
\begin{align*}
  \uEps(x,y) = x\sin(x)+y\sin(y).
\end{align*}
Again, we get the expected order of convergence in the $\Hk{2}$-, $\Hk{1}$-, and $\Lp$-norm as can be seen in Fig.~\ref{fig:test3} for the quadrilateral and quasi-uniform Voronoi meshes. Here, we note that on the finest meshes the errors appear to rise, which again is likely caused by machine precision and nonlinear solver tolerance.

\section{Conclusion}\label{sec:conclusion}
In this article, we have presented a $\Cp[1]$-nonconforming $\Cp[0]$-conforming virtual element formulation of the vanishing moment method for the approximation of the viscosity solution of the Monge--Amp\`ere equation in two dimensions. We have shown the existence of a unique solution of the virtual element approximation for a sufficiently small mesh size $h$. Additionally, we derived optimal a priori error bounds in the $\Hk{2}$- and $\Hk{1}$-norm, as well as an error bound in the $\Lp$-norm which is optimal for $\polOrder \geq 3$ and suboptimal for $\polOrder=2$. We provided numerical experiments on uniform quadrilateral and quasi-uniform Voronoi meshes which show that the vanishing moment method converges to the viscosity solution as the coefficient $\varepsilon$ of the added artificial biharmonic term tends to zero, and which validate the convergence rates in the a priori error bounds. The desired optimal error rates are achieved in the numerical experiments, and additionally note that we only achieve the suboptimal rate for the $\Lp$-norm for $\polOrder=2$ which was predicted by the analysis. These results provide an initial step to the analysis for a fully nonlinear VEM formulation, and the numerical approximation provides a useful method for an initial guess for solving a fully nonlinear formulation.

\section*{Funding}
This work has been supported by Charles University Research programme number PRIMUS/22/SCI/014.

\bibliographystyle{abbrvnat}
\bibliography{references}

\end{document}